\documentclass[11pt]{article}
\usepackage{fullpage}
\usepackage{graphicx}
\usepackage{amssymb}
\usepackage{amsmath}
\usepackage{amsthm}

\usepackage{stmaryrd}
\usepackage{mathabx}
\usepackage{mathrsfs}

\usepackage{latexsym}
\usepackage{rotating}

\newtheorem{theorem}{Theorem}[section]
\newtheorem{lemma}[theorem]{Lemma}
\newtheorem{recurrence}[theorem]{Recurrence}

\newtheorem{conjecture}[theorem]{Conjecture}
\newtheorem{corollary}[theorem]{Corollary}

\theoremstyle{definition}
\newtheorem{definition}[theorem]{Definition}
\newtheorem{remark}[theorem]{Remark}
             
\newcommand{\ignore}[1]{}
             
\newcommand{\hcm}[1][1]{\hspace*{#1 cm}}
\newcommand{\rb}[2]{\raisebox{#1 mm}[0mm][0mm]{#2}}
\newcommand{\istrut}[2][0]{\rule[- #1 mm]{0mm}{#1 mm}\rule{0mm}{#2 mm}}
\newcommand{\zero}[1]{\makebox[0mm][l]{$#1$}}
\newcommand{\ldotsup}[1]{{\rb{#1}{ \rotatebox{10}{ $\ldots$}}}}
\newcommand{\ldotsdown}[1]{{\rb{#1}{ \reflectbox{\rotatebox{10}{ $\ldots$}}}}}
\newcommand{\ldotsUP}[1]{{\rb{#1}{ \rotatebox{15}{ $\ldots$}}}}
\newcommand{\ldotsDOWN}[1]{{\rb{#1}{ \reflectbox{\rotatebox{15}{ $\ldots$}}}}}

\newcommand{\paren}[1]{{\left( #1 \right)}}

\newcommand{\angbrack}[1]{\left< #1 \right>}

\newcommand{\SqBrack}[1]{\Big[ #1 \Big]}

\newcommand{\ceil}[1]{\lceil #1 \rceil}
\newcommand{\floor}[1]{\lfloor #1 \rfloor}
\newcommand{\f}[2]{\frac{#1}{#2}}

\newcommand{\bydef}{\stackrel{\operatorname{def}}{=}}
\newcommand{\poly}{\operatorname{poly}}

\newcommand{\Erdos}{Erd\H{o}s}

\newcommand{\Furedi}{F\"{u}redi}

\newcommand{\Kyncl}{Kyn\v{c}l}

\newcommand{\Ex}{\operatorname{Ex}}

\newcommand{\DS}[1]{\lambda_{#1}}
\newcommand{\dblDS}[1]{\DS{{#1}}^{\scriptscriptstyle\dbl}}
\newcommand{\dblgamma}{\gamma^{\scriptscriptstyle\dbl}}
\newcommand{\Feather}[1]{\Phi_{#1}}

\newcommand{\compose}{\operatorname{\circ}}

\newcommand{\subseq}{\prec}
\newcommand{\nsubseq}{\nprec}

\newcommand{\dbl}{\operatorname{dbl}}

\newcommand{\Gm}{\hat{m}}
\newcommand{\Lm}{m}
\newcommand{\Gn}{\hat{n}}
\newcommand{\Ln}{\check{n}}

\newcommand{\Gfn}{\acute{n}}
\newcommand{\Gln}{\grave{n}}
\newcommand{\Gmn}{\bar{n}}

\newcommand{\Gnnf}{\mathring{n}}
\newcommand{\Gnf}{\tilde{n}}
\newcommand{\GS}{\hat{S}}
\newcommand{\GfS}{\acute{S}}
\newcommand{\GlS}{\grave{S}}
\newcommand{\GmS}{\bar{S}}
\newcommand{\GfSterminal}{\GfS^{\scriptscriptstyle \operatorname{t}}}
\newcommand{\GfSnonterminal}{\GfS^{\scriptscriptstyle \operatorname{nt}}}
\newcommand{\GfeatherS}{\tilde{S}}

\newcommand{\GnonfeatherS}{\mathring{S}}
\newcommand{\GSsingle}{\dot{S}}
\newcommand{\GSnonsingle}{\ddot{S}}
\newcommand{\LS}{\check{S}}
\newcommand{\GZ}{\hat{Z}}
\newcommand{\LZ}{\check{Z}}

\newcommand{\Tree}{\mathcal{T}}
\newcommand{\GTree}{\hat{\Tree}}
\newcommand{\LTree}{\check{\Tree}}
\newcommand{\CTree}{\Tree^\star}
\newcommand{\GfTree}{\acute{\Tree}}

\newcommand{\block}{\mathcal{B}}

\newcommand{\lefthead}{\operatorname{lh}}
\newcommand{\righthead}{\operatorname{rh}}
\newcommand{\crown}{\operatorname{cr}}
\newcommand{\quill}{\operatorname{qu}}
\newcommand{\wing}{\operatorname{wi}}

\newcommand{\ltip}{\operatorname{lt}}
\newcommand{\rtip}{\operatorname{rt}}
\newcommand{\feather}{\operatorname{fe}}
\newcommand{\fcrown}{\acute{\crown}}

\newcommand{\Pot}{\phi}
\newcommand{\PotEgg}{\phi^{\operatorname{eg}}}
\newcommand{\PotFertile}{\phi^{\operatorname{fe}}}
\newcommand{\PotInfertile}{\phi^{\operatorname{in}}}

\newcommand{\K}[1]{\pi_{#1}}
\newcommand{\dblK}[1]{\pi_{#1}^{\scriptscriptstyle\dbl}}

\newcommand{\fea}[1]{\phi_{#1}}
\newcommand{\ds}[1]{\delta_{#1}}
\newcommand{\dblds}[1]{\delta^{\scriptscriptstyle\dbl}_{#1}}

\newcommand{\bl}[1]{\llbracket #1 \rrbracket}

\newcommand{\bu}{\bullet}

\newcommand{\pattern}{\pi}
\newcommand{\ascend}{\rb{.2}{{\scalebox{.75}{${\operatorname{\diagup}}$}}}}
\newcommand{\descend}{\rb{.2}{{\scalebox{.75}{${\operatorname{\diagdown}}$}}}}
\newcommand{\ascendsm}{{\scalebox{.5}{${\operatorname{\diagup}}$}}}
\newcommand{\descendsm}{{\scalebox{.5}{${\operatorname{\diagdown}}$}}}
\newcommand{\preshuffle}{\varolessthan}
\newcommand{\postshuffle}{\varogreaterthan}

\newcommand{\flip}[1]{\operatorname{flip}(#1)}

\newcommand{\Ubot}{U_{\operatorname{bot}}}
\newcommand{\Utop}{U_{\operatorname{top}}}
\newcommand{\Umid}{U_{\operatorname{mid}}}
\newcommand{\Usub}{U_{\operatorname{sub}}}

\newcommand{\Tbot}{T_{\operatorname{bot}}}
\newcommand{\Ttop}{T_{\operatorname{top}}}
\newcommand{\Tmid}{T_{\operatorname{mid}}}
\newcommand{\Tsub}{T_{\operatorname{sub}}}
\newcommand{\Tsh}{T_{\operatorname{sh}}}

\newcommand{\mutop}{{\mu_{\scriptscriptstyle \operatorname{top}}}}
\newcommand{\mumid}{{\mu_{\scriptscriptstyle \operatorname{mid}}}}
\newcommand{\nutop}{{\nu_{\scriptscriptstyle \operatorname{top}}}}
\newcommand{\numid}{{\nu_{\scriptscriptstyle \operatorname{mid}}}}

\newcommand{\livebl}[1]{\llparenthesis\,{#1}\,\rrparenthesis}

\newcommand{\Perm}[1]{\operatorname{Perm}_{#1}}
\newcommand{\BinPerm}[1]{\operatorname{Bin}_{#1}}
\newcommand{\dblPerm}[1]{\operatorname{Perm}^{\scriptscriptstyle\dbl}_{#1}}

\newcommand{\PERM}[1]{\Lambda_{#1}}
\newcommand{\dblPERM}[1]{\Lambda_{#1}^{\scriptscriptstyle\dbl}}
\newcommand{\dblPERMfour}{\Upsilon}

\newcommand{\NN}{N\hspace{-.175cm}N}

\title{Three Generalizations of Davenport-Schinzel Sequences\thanks{This work is supported by NSF CAREER grant CCF-0746673,
NSF grants CCF-1217338 and CNS-1318294,
and a grant from the US-Israel Binational Science Foundation.}}

\author{Seth Pettie\\ University of Michigan}

\begin{document}

\maketitle

\begin{abstract}
We present new, and mostly sharp, bounds on the maximum length
of certain generalizations of Davenport-Schinzel sequences.
Among the results are sharp bounds on order-$s$ {\em double DS} sequences, for all $s$,
sharp bounds on sequences avoiding {\em catenated permutations} (aka formation free sequences),
and new lower bounds on sequences avoiding {\em zig-zagging} patterns.
\end{abstract}

\section{Introduction}

\nocite{Szemeredi73}

A generalized Davenport-Schinzel (DS) sequence is one over a finite alphabet, say $[n] = \{1,\ldots,n\}$, 
none of whose subsequences
are isomorphic to a fixed forbidden sequence $\sigma$ or a set of such sequences. (A sparsity criterion is also included in 
order to prohibit degenerate infinite sequences such as $aaaaa\cdots$.)
When $\sigma$ is the alternating sequence
$abab\cdots$ with length $s+2$ this definition reverts to that of standard order-$s$ DS sequences.
Whereas standard DS sequences have countless applications in discrete and computational geometry,
generalized DS sequences have found fewer applications~\cite{Valtr97,Suk11,FoxPS11,Pettie-SoCG11,Pettie-Deque-08,CibulkaK12}.  
Whereas bounding the length of DS sequences
is now essentially a closed problem~\cite{ASS89,Nivasch10,Pettie-SoCG13},
the most basic questions about generalized DS sequences are open, or have received only
partial answers.

We are mainly interested in answering two questions about forbidden sequences.
A purely quantitative question is to determine the maximum length $\Ex(\sigma,n)$ of a $\sigma$-free sequence
over an $n$-letter alphabet, for specific $\sigma$ or large classes of $\sigma$.
An equally interesting question, particularly when $\Ex(\sigma,n)$ is superlinear in $n$, is to characterize
the {\em structure} of $\sigma$-free sequences.  
There are infinitely many forbidden sequences one could study, 
but some classes of subsequences are more interesting than others, 
either because of their applications, or their intrinsic structure, or for historical reasons.
In this report we focus on 
forbidden sequences that generalize, in various ways, the idea of an alternating sequence.
In order to properly explain our results, in Section~\ref{sect:new-results}, we need to introduce some notation and terminology
and to review the history of DS sequences and their generalizations, in Sections~\ref{sect:notation-and-terminology}--\ref{sect:further-generalizations}.
For the moment we can take a high-level tour of the results.
Following convention, let $\DS{s}(n)=\Ex(abab\cdots,n)$ be the extremal function for order-$s$ DS sequences, 
where the alternating pattern has length $s+2$.

\paragraph{Double DS sequences.} The most modest way to generalize an alternating sequence $abab\cdots$ 
is simply to {\em double} each letter, transforming it to $abbaabb\cdots$.\footnote{It is straightforward to 
show  that repeating letters more than twice, or repeating the first and last at all,
can affect the extremal function by at most a constant factor.  See~\cite{AKV92}.}
Double DS sequences were the first generalized DS sequences to be studied~\cite{DS65b,AKV92,KV94}.
Let $\dblDS{s}$ be the extremal function of order-$s$ double DS sequences.
Davenport and Schinzel~\cite{DS65b} noted that $\dblDS{1}(n)$ is linear (see~\cite[p.~13]{Klazar02}) 
and Adamec, Klazar, and Valtr~\cite{AKV92} proved that $\dblDS{2}(n)$ is also linear, matching
$\DS{1}$ and $\DS{2}$ up to constant factors.  (The forbidden sequences here are $abba$ and $abbaab$.)
Klazar and Valtr~\cite{KV94} claimed without proof
that $\dblDS{3}(n)=\Theta(n\alpha(n))$, which would match $\DS{3}$ asymptotically~\cite{HS86}.
However, this claim was later retracted~\cite{Klazar02}.  Here $\alpha(n)$ is the inverse-Ackermann function.
We prove that $\dblDS{3}(n)$ is, in fact, $\Theta(n\alpha(n))$, and more generally, that
$\dblDS{s}$ and $\DS{s}$ are asymptotically equivalent for every order $s$.

\paragraph{Perm-free Sequences.}  Take any $s+1$ permutations over $\{a,b\}$.  Regardless of one's
choice, the concatenation of these permutations necessarily contains an alternating subsequence
of length $s+2$: the first permutation contributes two symbols and every subsequent permutation at least one.
Define $\Perm{r,s+1}$ to be the set of all sequences obtained by concatenating $s+1$ permutations
over an $r$-letter alphabet, and let $\PERM{r,s}$ be the extremal function of $\Perm{r,s+1}$-free sequences.\footnote{These were called {\em $(r,s+1)$-formation-free} sequences by Nivasch~\cite{Nivasch10}.}
The argument above shows that order-$s$ DS sequences are $\Perm{2,s+1}$-free, which implies
that $\DS{s}(n) \le \PERM{2,s}(n)$.
Klazar~\cite{Klazar92} introduced $\Perm{r,s+1}$-free
sequences as a ``universal'' method for finding upper bounds on $\Ex(\sigma,n)$.
If there exist $r,s$ (and there always do) such that $\sigma$ is contained in every member of $\Perm{r,s+1}$,
then $\Ex(\sigma,n) = O(\PERM{r,s}(n))$. 

It is straightforward to show that $\DS{s}(n)$ and $\PERM{2,s}(n)$
are asymptotically equivalent.  A natural hypothesis, given~\cite{Nivasch10,Pettie-SoCG13}, is that
$\DS{s}$ and $\PERM{r,s}$ are asymptotically equivalent, for all $r$.  We prove that this hypothesis is false,
which is quite surprising.  One upshot of~\cite{ASS89,Nivasch10,Pettie-SoCG13} is that when $s\ge 7$ is odd, 
$\DS{s}(n)$ and $\DS{s-1}(n)$ are essentially indistinguishable, and that $\DS{5}(n)$ and $\DS{4}(n)$ are asymptotically 
distinguishable, but very similar.  In contrast, we prove that, in general, $\PERM{r,s}(n)$ behaves very differently
at odd and even $s$.  The extremal functions $\DS{s}$ and $\PERM{r,s}$ are asymptotically equivalent 
{\em only} when $s\le 3$, or $s\ge 4$ is even, or $r=2$.

Just as DS sequences can be generalized to double DS sequences, $\Perm{r,s+1}$ can be transformed
into a set $\dblPerm{r,s+1}$ by ``doubling'' it.  Let $\dblPERM{r,s}(n)$ be the extremal function of $\dblPerm{r,s+1}$-free
sequences.  The function $\dblPERM{r,s}$ was studied in a different, but essentially equivalent form by Cibulka and \Kyncl~\cite{CibulkaK12}.
We prove that $\dblPERM{r,s}$ is asymptotically
equivalent to $\PERM{r,s}$ for all $r,s$.  This fact is not surprising, but what {\em is} surprising is how 
many new techniques are needed to prove it when $s=3$.

\paragraph{Zig-zagging Patterns.}
One way to view the alternating sequence $abab\cdots$ with length $s+2$
is as a {\em zig-zagging} pattern with $s+1$ {\em zigs} and {\em zags}.
Generalized to larger alphabets, we obtain
the $N$-shaped sequences, of the form $ab\cdots zy\cdots ab\cdots z$, when $s=2$,
the $M$-shaped sequences $ab\cdots zy\cdots ab\cdots zy\cdots a$, when $s=3$,
the $\NN$-shaped sequences $ab\cdots zy\cdots ab\cdots zb\cdots ab\cdots z$, when $s=4$,
and so on.  Klazar and Valtr~\cite{KV94} (see also~\cite{Pettie-SoCG11}) proved
that the extremal function of each $N$-shaped forbidden sequence is linear, matching $\DS{2}(n)$.
See Valtr~\cite{Valtr97} for an application of $N$-shaped sequences to bounding the size of geometric graphs
and Pettie~\cite{Pettie-SoCG11} for an application of $M$-shaped sequences to bounding the complexity of the union of fat triangles.

Given~\cite{KV94,Pettie-SoCG11}, 
one is tempted to guess that the extremal function for a zig-zagging forbidden sequence is, 
if not asymptotically equivalent to the corresponding order-$s$ DS sequence, at least close to it.
We give lower bounds showing that for each $t$, there is an $M$-shaped forbidden sequence
with extremal function $\Omega(n\alpha^t(n))$ and an $\NN$-shaped forbidden sequence
with extremal function $\Omega(n\cdot 2^{(1+o(1))\alpha^t(n)/t!})$.
Put a different way, in terms of their extremal functions
$M$-shaped sequences may be similar to $ababa$
but $\NN$-shaped sequences bear no resemblance to $ababab$.

Our results on zig-zagging patterns are the least conclusive, and therefore 
offer the most opportunities for future research.  They are based
on a general, parameterized method for constructing non-linear sequences.

\subsection{Sequence Notation and Terminology}\label{sect:notation-and-terminology}

Let $|\sigma|$ be the length of a sequence $\sigma = (\sigma_i)_{1\le i\le |\sigma|}$ and let $\|\sigma\|$ be the size
of its alphabet $\Sigma(\sigma) = \{\sigma_i\}$.
Two equal length sequences are {\em isomorphic} if they are the same up to a renaming of their alphabets.
We say $\sigma$ is a {\em subsequence} of $\sigma'$
if $\sigma$ can be obtained by deleting symbols from $\sigma'$.
The predicate $\sigma \subseq \sigma'$ asserts that $\sigma$ is isomorphic to a subsequence of $\sigma'$.
If $\sigma\nsubseq\sigma'$ we say $\sigma'$ is {\em $\sigma$-free}.
If $P$ is a {\em set} of sequences, 
$\sigma \subseq P$ holds if $\sigma\subseq \sigma'$ for every $\sigma' \in P$
and $P\nsubseq \sigma$ holds if $\sigma'\nsubseq \sigma$ for every $\sigma'\in P$.
The assertion that $\sigma$ {\em appears in} or {\em occurs in} or {\em is contained in} $\sigma'$ 
means $\sigma\subseq \sigma'$.
The {\em projection} of a sequence $\sigma$ onto $G\subseteq \Sigma(\sigma)$ is obtained by deleting
all non-$G$ symbols from $\sigma$.
A sequence $\sigma$ is {\em $k$-sparse} if whenever $\sigma_i = \sigma_j$ and $i\neq j$, then $|i-j| \ge k$.
A {\em block} is a sequence of distinct symbols.  If $\sigma$
is understood to be partitioned into a sequence of blocks, $\bl{\sigma}$ is the number of blocks.
The predicate 
$\bl{\sigma}=m$ asserts that $\sigma$ can be partitioned into at most $m$ blocks.
The extremal functions for {\em generalized} Davenport-Schinzel sequences are defined to be
\begin{align*}
\Ex(\sigma,n,m) &= \max\{|S| \;:\; \sigma\nsubseq S, \; \|S\|=n, \mbox{ and } \bl{S}\le m\}\\
\Ex(\sigma,n)	&= \max\{|S| \;:\;  \sigma\nsubseq S, \; \|S\|=n, \mbox{ and } S \mbox{ is } \|\sigma\|\mbox{-sparse}\}\\
\intertext{where $\sigma$ may be a single sequence or a set of sequences.  
The conditions ``$\bl{S}\le m$'' and ``$S$ is $\|\sigma\|$-sparse'' guarantee that the extremal functions are finite.
Note that $\Ex(\sigma,n,m)$ has no sparseness criterion.
The extremal functions for  order-$s$ DS sequences are defined to be}
\DS{s}(n) 
&
= \Ex(\overbrace{abab\cdots}^{\mbox{\scriptsize length $s+2$}},n)
\hcm\mbox{ and }\hcm
\DS{s}(n,m) 
= \Ex(\overbrace{abab\cdots}^{\mbox{\scriptsize length $s+2$}},n,m).
\intertext{%
Since $\|abab\cdots\|=2$, the sparseness criterion forbids only immediate repetitions.}
\end{align*}

\subsection{Davenport, Schinzel, Ackermann, Tarjan}\label{sect:stdDS}

Davenport and Schinzel~\cite{DS65} observed that $\DS{1}(n) = n$ and $\DS{2}(n)=2n-1$.
It took several decades for all the other orders to be understood.  
The following theorem synthesizes results of Hart and Sharir~\cite{HS86},
Agarwal, Sharir, and Shor~\cite{ASS89},
Klazar~\cite{Klazar99},
Nivasch~\cite{Nivasch10},
and Pettie~\cite{Pettie-SoCG13}.

\begin{theorem}\label{thm:main}
Let $\DS{s}(n)$ be the maximum length of a repetition-free sequence over an $n$-letter alphabet
avoiding subsequences isomorphic to $abab\cdots$ (length $s+2$).  Then $\DS{s}$ satisfies:
\[
\DS{s}(n) = \left\{
\begin{array}{l@{\hcm}l@{\istrut[3]{0}}}
n								& s=1\\
2n-1								& s=2\\
2n\alpha(n) + O(n)					& s=3\\
\Theta(n2^{\alpha(n)})				& s=4\\
\Theta(n\alpha(n)2^{\alpha(n)})			& s=5\\
n\cdot 2^{\alpha^t(n)/t! \,+\, O(\alpha^{t-1}(n))}	& s\ge 6, \; t = \floor{\frac{s-2}{2}}.
\end{array}
\right.
\]
\end{theorem}

Here $\alpha(n)$ is the functional inverse of Ackermann's function discovered by Tarjan~\cite{Tar75}, defined as follows.
\begin{align*}
a_{1,j} &= 2^j		& j\ge 1\\
a_{i,1} &= 2		& i\ge 2\\	
a_{i,j}  &= w \cdot a_{i-1,w}		& i,j\ge 2\\
&\mbox{where }  w = a_{i,j-1}
\intertext{One may check that in the table  $(a_{i,j})$, the first column
is constant and the second column merely exponential: $a_{i,1}=2$ and $a_{i,2} = 2^i$.
Ackermann-type growth only appears at the third column, motivating the following 
definition of the inverse functions.}%
\alpha(n,m) &= \zero{\min\{i \;|\; a_{i,j} \ge m,\, \mbox{ where } j = \max\{\ceil{n/m},3\}\}}\\
\alpha(n) &= \alpha(n,n)
\end{align*}
There are numerous variants of Ackermann's function in the literature, 
all of which are equivalent inasmuch as their inverses differ by at most a constant.
Observe that Theorem~\ref{thm:main} is robust to perturbations of $\alpha(n)$ by $O(1)$,
so it does not depend on any particular definition of Ackermann's function or its inverse.\footnote{See Pettie~\cite[p. 4]{Pettie-SoCG13}
for a discussion of this notion of ``Ackermann-invariance.''}

\subsection{Generalizations of DS Sequences}\label{sect:further-generalizations}

Certain classes of forbidden sequences have received significant attention.  We review three systems for
generalizing (standard) DS sequences, then mention some miscellaneous results in the area.

\paragraph{Double DS Sequences.}
Let $\dbl(\sigma)$ be obtained from $\sigma$ by doubling each letter except for the first and last, for example,
$\dbl(abcabc) = abbccaabbc$.
The extremal functions for order-$s$ {\em double DS} sequences 
are $\dblDS{s}(n) = \Ex(\dbl(abab\cdots),n)$ and $\dblDS{s}(n,m) = \Ex(\dbl(abab\cdots),n,m)$,
where the alternating sequence has length $s+2$.
It is known that $\dblDS{1}(n)$ and $\dblDS{2}(n)$ are linear, matching $\DS{1}$ and $\DS{2}$ asymptotically.
See Davenport and Schinzel~\cite{DS65b}, Adamec, Klazar, and Valtr~\cite{AKV92},
and Klazar~\cite[p. 13]{Klazar96,Klazar02}.
Pettie~\cite{Pettie-GenDS11,Pettie-SoCG11} proved that $\dblDS{3}(n) = O(n\alpha^2(n))$
and $\Ex(\{abbaabba,ababab\},n) = \Theta(n\alpha(n))$, and that for $s\ge 4$,
 $\dblDS{s}(n)$ matched what were the best upper bounds on $\DS{s}(n)$ at the time~\cite{Nivasch10}, namely
$\dblDS{s}(n) < n\cdot 2^{\alpha^t(n)/t! \,+\, O(\alpha^{t-1}(n))}$, for even $s$, and 
$\dblDS{s}(n) < n\cdot 2^{\alpha^t(n)(\log(\alpha(n))+O(1))/t!}$, for odd $s$.

\paragraph{Catenated Permutations.}
Recall that $\Perm{r,s+1}$ is defined to be the set of sequences obtained by concatenating $s+1$ permutations
over an $r$-letter alphabet.  For example, $abcd\; cbad\; badc \in \Perm{4,3}$.
Let $\PERM{r,s}(n) = \Ex(\Perm{r,s+1},n)$ to be the extremal function for $\Perm{r,s+1}$-free sequences,
with $\PERM{r,s}(n,m)$ defined analogously.\footnote{The ``$s+1$'' here is chosen to highlight the 
parallels with order-$s$ DS sequences.  Recall that every $\sigma\in\Perm{2,s+1}$ contains an alternating sequence 
$abab\cdots$ with length $s+2$, hence $\DS{s}(n) \le \PERM{2,s}(n)$.}
It is straightforward to show that if $\sigma$ is contained in every member of $\Perm{r,s+1}$ then
\begin{equation*}
\Ex(\sigma,n,m) \le \PERM{r,s+1}(n,m) \hcm\mbox{and}\hcm \Ex(\sigma,n) = O(\PERM{r,s+1}(n)).\label{eqn:PermFree}
\end{equation*}
Nivasch~\cite{Nivasch10} proved that any $\sigma$ is contained in every member of $\Perm{\|\sigma\|,|\sigma|-\|\sigma\|+1}$.
Very recently Geneson, Prasad, and Tidor~\cite{GenesonPT13} showed that it suffices to consider a subset $\BinPerm{r,s+1}\subset \Perm{r,s+1}$
consisting of {\em binary} patterns, where each of the $s+1$ permutations is either $12\cdots (r-1)r\,$ or $\,r(r-1)\cdots 21$.
By repeated application of the \Erdos-Szekeres{} theorem, 
they showed that every member of $\Perm{r',s+1}$ contains a member of $\BinPerm{r,s+1}$,
where $r'=(r-1)^{2^s}+1$.  Consequently, if $\sigma$ is contained in every member of $\BinPerm{r,s+1}$ then $\Ex(\sigma,n) = O(\PERM{r',s}(n))$.

Nivasch~\cite{Nivasch10}, improving~\cite{Klazar92}, 
gave the following upper bounds on $\PERM{r,s}$, for any $r\ge 2, s\ge 1$, where $t = \floor{\frac{s-2}{2}}$.
The lower bounds follow from previous~\cite{HS86,ASS89} and subsequent~\cite{Pettie-SoCG13} constructions of order-$s$ DS sequences.
\[
\PERM{r,s}(n) = \left\{
\begin{array}{l@{\hcm}l@{\istrut[3]{0}}}
\Theta(n)													& s\le 2\\
\Theta(n\alpha(n))											& s=3\\
\Theta(n2^{\alpha(n)})										& s=4\\
\Omega\paren{n\alpha(n)2^{\alpha(n)}} 	\mbox{ and } 	O\paren{n2^{\alpha(n)(\log\alpha(n)+O(1))}}	& s=5\\
n\cdot 2^{\alpha^t(n)/t! \,+\, O(\alpha^{t-1}(n))}						& \mbox{even $s\ge 6$}\\
\Omega\paren{n\cdot 2^{\alpha^t(n)/t! \,+\, O(\alpha^{t-1}(n))}} 
	\mbox{ and } 	O\paren{n\cdot 2^{\alpha^t(n)(\log\alpha(n) + O(1))/t!}}	& \mbox{odd $s\ge 7$}
\end{array}
\right.
\]
Note that $\PERM{r,s}$ matches the behavior of $\DS{s}$ when $s\le 3$ or $s$ is even.

Cibulka and \Kyncl~\cite{CibulkaK12} studied a problem on 0-1 matrices that is essentially equivalent 
to the following generalization of Perm-free sequences.
Define $\dblPerm{r,s+1}$ to be the set of all sequences over $[r] = \{1,\ldots,r\}$ that can be written
$\sigma_1\ldots\sigma_{s+1}$, where $\sigma_1$ and $\sigma_{s+1}$ are permutations of $[r]$
and $\sigma_2,\ldots,\sigma_s$ are sequences containing two copies of each symbol in $[r]$.
Define $\dblPERM{r,s}(n)$ and $\dblPERM{r,s}(n,m)$ to be the extremal functions of $\dblPerm{r,s+1}$-free sequences.
Cibulka and \Kyncl{} only considered $\dblPERM{r,s}(n,m)$.  For consistency we state the bounds on $\dblPERM{r,s}(n)$
they {\em would} have obtained using the available reductions from $r$-sparse to blocked sequences~\cite{Nivasch10}.\footnote{The only notable case here is $s=4$.  Cibulka and \Kyncl{} proved that $\dblPERM{r,1}(n,m) = O(n+m)$, $\dblPERM{r,2}(n,m) = O((n+m)\alpha(n,m))$ and $\dblPERM{r,4}(n,m) = O((n+m)\alpha(n,m)2^{\alpha(n,m)})$, which imply, by~\cite[Lem.~5.7]{Nivasch10}, that $\dblPERM{r,2}(n) = O(n\alpha(n))$ and 
$\dblPERM{r,4}(n) = O(n\alpha^2(n)2^{\alpha(n)})$.}
For any $r\ge 2,s\ge 1,$ and $t = \floor{\frac{s-2}{2}}$,
\[
\dblPERM{r,s}(n) = \left\{
\begin{array}{l@{\hcm}l@{\istrut[3]{0}}}
\Theta(n)													& s=1\\
\Omega(n) \mbox{ and } O(n\alpha(n))							& s=2\\
\Omega(n\alpha(n)) \mbox{ and } O(n\alpha^2(n))					& s=3\\
\Omega(n2^{\alpha(n)}) \mbox{ and } O(n\alpha^2(n)2^{\alpha(n)})		& s=4\\
\Omega\paren{n\alpha(n)2^{\alpha(n)}} 	\mbox{ and } 	O\paren{n2^{\alpha(n)(\log\alpha(n)+O(1))}}	& s=5\\
n\cdot 2^{\alpha^t(n)/t! \,+\, O(\alpha^{t-1}(n))}						& \mbox{even $s\ge 6$}\\
\Omega\paren{n\cdot 2^{\alpha^t(n)/t! \,+\, O(\alpha^{t-1}(n))}} 
	\mbox{ and } 	O\paren{n\cdot 2^{\alpha^t(n)(\log\alpha(n) + O(1))/t!}}	& \mbox{odd $s\ge 7$}
\end{array}
\right.
\]
The definition of $\dblPerm{r,s+1}$ may at first seem unnatural.  Surely $\dbl(\Perm{r,s+1}) = \{\dbl(\sigma) \,|\, \sigma\in\Perm{r,s+1}\}$
would be a more useful way to ``double'' the set $\Perm{r,s+1}$.  For example, it is known that $abcacbc\subseq \Perm{4,4}$, and therefore that $\dbl(abcacbc)\subseq \dbl(\Perm{4,4})$,
but we cannot immediately conclude, as we would like, that $\Ex(\dbl(abcacbc),n) \le \dblPERM{4,3}(n)$.  It turns out that the maximum length of $\dblPerm{r,s+1}$-free
sequences and $\dbl(\Perm{r,s+1})$-free sequences are the same asymptotically.  The proof of Lemma~\ref{lem:dblPermequiv} appears in the appendix.

\begin{lemma}\label{lem:dblPermequiv}
The following bounds hold for any $r\ge 2, s\ge 1$.
\begin{align*}
\Ex(\dbl(\Perm{r,s+1}),n,m) &\le r\cdot \dblPERM{r,s}(n,m) + 2rn\\
\Ex(\dbl(\Perm{r,s+1}),n) &= O(\dblPERM{r,s}(n)).
\end{align*}
\end{lemma}

\paragraph{Zig-zagging Patterns.}
Klazar and Valtr~\cite{KV94} introduced the $N$-shaped zig-zagging patterns $\{N_k\}$, where
\[
\istrut{8}
N_k = 1\rb{1.25}{ 2} \ldotsUP{2.5} \rb{3.75}{ $(k+1)$} \rb{2.5}{ $k$} \ldotsDOWN{1.25} 1 \rb{1.25}{ 2} \ldotsUP{2.5} \rb{3.75}{ $(k+1).$}
\]
Note that $N_k$-free sequences generalize order-2 DS sequences since $N_1 = abab$.  
(The vertical placement of the symbols in $N_k$
carries no meaning.  It is only intended to improve readability.)
It was shown~\cite{KV94,Pettie-SoCG11} that $\Ex(\dbl(N_k),n) = O(n)$, which matches  $\DS{2}(n)$ asymptotically.
Pettie~\cite{Pettie-SoCG11} proved that $\Ex(\{M_k,ababab\},n) = \Theta(n\alpha(n))$, matching $\DS{3}(n)$, 
where $M_k$ is the $k$th $M$-shaped sequence,
\[
\istrut{7}
M_k = 1\rb{1.25}{ 2} \ldotsUP{2.5} \rb{3.75}{ $(k+1)$} \rb{2.5}{ $k$} \ldotsDOWN{1.25} 1 \rb{1.25}{ 2} \ldotsUP{2.5} \rb{3.75}{ $(k+1)$}
\rb{2.5}{ $k$} \ldotsDOWN{1.25} 1.
\]
See~\cite{Valtr97,Suk11,FoxPS11,Pettie-SoCG11} for applications of $N$- and $M$-shaped sequences.

A different way to view even-length alternating patterns $abab\cdots$ with length $s+2$ is as a sequence of $(s+2)/2$ zigs, without corresponding zags.
When generalized to an $r$-letter alphabet we get the sequence $(12\cdots r)^{(s+2)/2}$, which
is contained in every member of $\BinPerm{r,s+1}$ since at least $\ceil{\f{s+1}{2}}$ of the constituent permutations must be identical.  
It follows from \cite{GenesonPT13,ASS89,Nivasch10} 
that $\Ex((1\cdots r)^{(s+2)/2},n) = \Theta(\PERM{r',s}(n)) = n\cdot 2^{(1+o(1))\alpha^t(n)/t!}$,
where $r' = (r-1)^{2^s}+1$ and $t=\floor{\frac{s-2}{2}}$.

\paragraph{Other Forbidden Patterns.}
Much of the research on generalized DS sequences~\cite{AKV92,KV94,Klazar02,Pettie-SoCG11,Pettie-DS-nonlin11,Pettie-GenDS11,Pettie-FH11}
has focussed on delineating linear and non-linear forbidden sequences.  A $\sigma$ is {\em linear} if $\Ex(\sigma,n) = O(n)$.
It is known that $ababa$ and $abcacbc$ are the only 2-sparse 
minimally non-linear sequences over three letters~\cite{KV94,Pettie-GenDS11,Pettie-SoCG11}.
There are only a few varieties of sequences known to be linear.  We have already seen that doubled $N$-shaped sequences ($\dbl(N_k)$)
are in this category.  Pettie~\cite{Pettie-SoCG11,Pettie-FH11} proved that $abcbbccac$ is linear, and showed that if $\pi_1,\pi_2$ are two permutations on the same alphabet, then $\pi_1\dbl(\pi_2)$ is linear.  For example, $\Ex(abcde\, aacceebbd,n) = O(n)$.  
More linear sequences can be generated via Klazar and Valtr's~\cite{KV94} splicing operation.  If $\sigma=\sigma_1aa\sigma_2$ and $\sigma'$ are linear, where $\Sigma(\sigma)\cap\Sigma(\sigma')=\emptyset$, then $\sigma_1a\sigma'a\sigma_2$ is also linear. 

Other research has focussed on identifying cofinal sets of forbidden sequences, with respect to the total order on 
extremal functions.\footnote{A set $\mathcal{A}$ is cofinal if, for any $\sigma$, there is a $\sigma'\in\mathcal{A}$ such that $\Ex(\sigma,n) = o(\Ex(\sigma',n))$.}  Klazar's general upper bounds~\cite{Klazar92} imply that standard DS sequences $\{(ab)^k\}$ are cofinal.  Pettie~\cite{Pettie-GenDS11}, answering a question of Klazar~\cite{Klazar02}, proved that the set of $ababa$-free forbidden sequences is also cofinal.  This fact is witnessed by the two-sided {\em comb}-shaped sequences $\{D_k\}$, which generalize $D_1 = abacacbc$.  Here $D_k$ is defined to be 
\[
\istrut{13}
D_k =
1\rb{2}{ $2$}\mbox{ 1}\rb{3}{ $3$}\mbox{ 1}\rb{4}{ $4$} \ldotsUP{6} { 1}\rb{10}{ $(k+2)$}\mbox{ $1$}
\rb{10}{ $(k+2)$}\rb{2}{ $2$}\rb{10}{ $(k+2)$}\rb{3}{ $3$}\rb{10}{ $(k+2)$}\rb{4}{ $4$}\rb{10}{ $(k+2)$} \ldotsup{6} \rb{7}{ $(k+1)$}\rb{10}{ $(k+2).$}
\]

\subsection{New Results}\label{sect:new-results}

In prior work~\cite{Pettie-SoCG13} we showed that $\DS{s}$ behaves very similarly at the odd and even orders.
In this paper we prove, quite unexpectedly, that $\PERM{r,s}$ matches $\DS{s}$ {\em only} when $s\le 3$, or $s\ge 4$ is even, or $r=2$.
When $s\ge 5$ is odd and $r\ge 3$, $\PERM{r,s}$ and $\DS{s}$ diverge.
Moreover, we prove that $\DS{s}$ and $\dblDS{s}$ are essentially equivalent, and that 
$\PERM{r,s}$ and $\dblPERM{r,s}$ are essentially equivalent.

\begin{theorem}\label{thm:PERM}
(Omnibus Bounds)
For all $s\ge 1$ and $r=2$, $\DS{s}$, $\dblDS{s}$, $\PERM{r,s}$, and $\dblPERM{r,s}$ are asymptotically equivalent, namely,
\begin{align*}
\DS{s}(n), \dblDS{s}(n), \PERM{2,s}(n), \dblPERM{2,s}(n) &= \left\{
\begin{array}{l@{\istrut[3]{0}}l}
\Theta(n)							& s\le 2\\
\Theta(n\alpha(n))					& s=3\\
\Theta(n2^{\alpha(n)})				& s=4\\
\Theta(n\alpha(n)2^{\alpha(n)})			& s=5\\
\zero{n\cdot 2^{\alpha^t(n)/t! \,+\, O(\alpha^{t-1}(n))}}\hcm[6]		& \mbox{$s\ge 6,$ where $t = \floor{\frac{s-2}{2}}$}.
\end{array}
\right.
\intertext{However, the behavior of $\PERM{r,s}$ and $\dblPERM{r,s}$ changes
when $r\ge 3$.  In particular,}
\PERM{r,s}(n),\dblPERM{r,s}(n) &= \left\{
\begin{array}{l@{\istrut[3]{0}}l}
\Theta(n)		& \mbox{$s\le 2$}\\
\Theta(n\alpha(n)) & \mbox{$s=3$}\\
\Theta(n2^{\alpha(n)}) & \mbox{$s=4$}\\
n\cdot 2^{\alpha^t(n)(\log\alpha(n)+O(1))/t!}					& \mbox{odd $s\ge 5$}\\
\zero{n\cdot 2^{\alpha^t(n)/t!  \,+\, O(\alpha^{t-1}(n))}}\hcm[6]						& \mbox{even $s\ge 6$.}
\end{array}\right.
\end{align*}
\end{theorem}

The new parts of Theorem~\ref{thm:PERM} not covered by previous work~\cite{HS86,ASS89,Nivasch10,CibulkaK12,Pettie-SoCG13} are
\begin{enumerate}
\setlength{\itemsep}{-0pt}
\item[(i)] upper bounds on $\dblDS{s}$, for $s\ge 4$, which also cover $\dblPERM{2,s}$,
\item[(ii)] lower bounds on $\PERM{r,s}$ for $r\ge 3$ and odd $s\ge 5$,
\item[(iii)] a linear upper bound on $\dblPERM{r,2}$,
\item[(iv)] an $O(n2^{\alpha(n)})$ upper bound on $\dblPERM{r,4}$, and
\item[(v)] an $O(n\alpha(n))$ upper bound on $\dblPERM{r,3}$, which also covers $\dblDS{3}$.
\end{enumerate}
For task (i) we generalize (and simplify) the recent analysis of~\cite{Pettie-SoCG13} to work for double DS sequences.  This analysis
{\em only} achieves tight bounds for $s\ge 4$.  For task (ii) we give a construction of sequences that are $\Perm{3,s+1}$-free (but necessarily not $\Perm{2,s+1}$-free) with length $n\cdot 2^{\alpha^t(n)(\log\alpha(n) + O(1))/t!}$.  Task (iii) requires no proof. It follows from the linearity of $\dbl(N_k)$-free sequences.
For task (iv) we give a single analysis of $\dblPERM{r,s}$ that is tight for all $r\ge 3,s\ge 4$, but not $s=3$.
Task (v) is far and away the most difficult to prove.  It requires the development of techniques new to the analysis of generalized DS sequences.

\paragraph{Zig-zagging Patterns.}
Recall that the $N$- and $M$-shaped sequences $\{N_k,M_k\}$ generalize $abab=N_1$ and $ababa=M_1$.  Define $Z_k$ to be the 
corresponding generalization of $ababab=Z_1$, that is,
\[
\istrut{7}
Z_k = 1\rb{1.25}{ 2} \ldotsUP{2.5} \rb{3.75}{ $(k+1)$} \rb{2.5}{ $k$} \ldotsDOWN{1.25} 1 \rb{1.25}{ 2} \ldotsUP{2.5} \rb{3.75}{ $(k+1)$}
\rb{2.5}{ $k$} \ldotsDOWN{1.25} 1 \rb{1.25}{ 2} \ldotsUP{2.5} \rb{3.75}{ $(k+1).$}
\]
We give a flexible new way to construct (and succinctly encode) nonlinear sequences that subsumes nearly all prior 
constructions~\cite{HS86,ASS89,Komjath88,Nivasch10,Pettie-DS-nonlin11,Pettie-GenDS11,Pettie-SoCG13}.
Using the new constructions we are able to show that for any $t$, there exists a $k$ such that $\Ex(M_k,n) = \Omega(n\alpha^t(n))$
and an $l$ such that $\Ex(Z_l,n) = \Omega(n\cdot 2^{(1+o(1))\alpha^t(n)/t!})$.  The bounds on $M_k$-free sequences are perhaps not too surprising,
but they demonstrate that the extremal function for a {\em set} of forbidden sequences can be different than any member.
(Recall that $\Ex(\{M_k,ababab\},n) = \Theta(n\alpha(n))$ for any $k$~\cite{Pettie-SoCG11}.)
The new bounds on $Z_l$ show definitively that, in general, zig-zagging sequences are not closely tied to the corresponding DS sequences.
In fact, the set $\{Z_l\}$ is cofinal among all forbidden sequences, the other known cofinal sets being $\{(ab)^k\}$ and two-sided combs $\{D_k\}$.
Our new sequence constructions also let us show that the one-sided combs $\{C_k\}$ behave differently than $C_1=abcacbc$, where
\[
\istrut{11}
C_k =
1\,\rb{1}{ $2$}\,\rb{2}{ $3$} \ldotsUP{4} \;\rb{7}{ $(k+2)$}\mbox{ $1$}
\rb{7}{ $(k+2)$}\rb{1}{ $2$}\rb{7}{ $(k+2)$}\rb{2}{ $3$}\rb{7}{ $(k+2)$} \ldotsup{4}\,\rb{5}{ $(k+1)$}\rb{7}{ $(k+2).$}
\]
We prove $\Ex(C_k,n) = \Omega(n\alpha^{k}(n))$.

\subsection{Organization}
In Section~\ref{sect:lbPERM} we present sharp lower bounds on $\Perm{r,s+1}$-free sequences.
In Section~\ref{sect:basic} we review a number of standard sequence transformations
and review the linear upper bounds on $\DS{s},\dblDS{s},\PERM{r,s},$ and $\dblPERM{r,s}$ when $s\in\{1,2\}$.
In Section~\ref{sect:dblPERM} we establish  sharp upper bounds on $\dblPERM{r,s}$-free sequences, for all $s\ge 4$.
Section~\ref{sect:derivation-tree} reviews the {\em derivation tree} structure introduced in~\cite{Pettie-SoCG13},
which is used in Sections~\ref{sect:dblPERMfour} and \ref{sect:dblDS}.  In Section~\ref{sect:dblPERMfour} we present
sharp upper bounds on $\dblPERM{r,3}$ (and $\dblDS{3}$) and in Section~\ref{sect:dblDS} we give sharp upper bounds
on $\dblDS{s}$ for all $s\ge 4$.  Section~\ref{sect:NMZC} is devoted to a new, generalized construction of nonlinear sequences.
We prove that, under appropriate parameterization, they are $M_k$-free, $Z_k$-free, and $C_k$-free.
Some open problems are discussed in Section~\ref{sect:conclusions}.

%
%
%
%
%
%
%
%
\section{Lower Bounds on Perm-Free Sequences}\label{sect:lbPERM}

\subsection{Composition and Shuffling}\label{sect:composition-and-shuffling}

We consider sequences made up of blocks, each of which is designated {\em live} or {\em dead}.
To distinguish the two we use parentheses to indicate live blocks and angular brackets for dead blocks.
The number of live blocks in $T$ is $\livebl{T}$ and the number of both types is $\bl{T}$.
Our sequences are constructed through {\em composition} and two types of {\em shuffling} operations.
These operations were implicit in all constructions since Hart and Sharir~\cite{HS86} but were
usually presented in an ad hoc manner.

\paragraph{Composition}
A sequence $T$ over the alphabet $\{1,\ldots,\|T\|\}$ is in {\em canonical form} if symbols are ordered according to their first appearance
in $T$.  All sequences encountered in our construction are assumed to be in canonical form.
To {\em substitute} $T$ for a block $B=(a_1,\ldots,a_{\|T\|})$ means to replace $B$ with a copy of $T(B)$
under the alphabet mapping $k \mapsto a_k$.  
If $\Tmid$ is a sequence with $\|\Tmid\|=j$ and $\Ttop$ a sequence in which live blocks have length $j$,
$\Tsub = \Ttop \compose \Tmid$ is obtained by substituting for each live block $B$ in $\Ttop$ 
a copy $\Tmid(B)$. 
The live/dead status of a block in $\Tsub$ is inherited from its status in $\Ttop$ or $\Tmid$,
hence
$\livebl{\Tsub} = \livebl{\Ttop}\cdot \livebl{\Tmid}$ and $\bl{\Tsub} = \bl{\Ttop} + \livebl{\Ttop}(\bl{\Tmid}-1)$.
If all symbols appear in $\mutop$ live blocks and $\nutop$ dead blocks in $\Ttop$,
and $\mumid$ live blocks and $\numid$ dead blocks in $\Tmid$, then the corresponding multiplicities
in $\Tsub$ are $\mutop\cdot\mumid$ and $\nutop + \mutop\cdot\numid$.

\paragraph{Shuffling}
Let $\Tbot = \paren{L_1} \: \angbrack{D_1} \: \paren{L_2} \: \angbrack{D_2} \cdots \paren{L_l} \: \angbrack{D_l}$ be a sequence with $l$ live blocks $L_1,\ldots,L_l$
and $\Tsub = \paren{L_1'}  \: \angbrack{D_1'}  \: \paren{L_2'}  \: \angbrack{D_2'} \cdots \paren{L_{l'}'}  \: \angbrack{D_{l'}'}$ be a sequence whose live blocks 
$L_1',\ldots,L_{l'}'$ have length precisely $l=\livebl{\Tbot}$.  The $D$s here represents zero or more dead blocks appearing
between live blocks.  The {\em post}shuffle $\Tsh = \Tsub \postshuffle \Tbot$
is obtained by first forming the concatenation $\Tbot^*$ of $l'$ copies of $\Tbot$, each over an alphabet disjoint from the other copies.
A copy of $\Tsub$ is shuffled into $\Tbot^*$ as follows.
Let $L_q' = \paren{a_1a_2\cdots a_l}$ be the $q$th live block of $\Tsub$ and $\Tbot^{(q)} = \paren{L_1^{(q)}}  \: \angbrack{D_1^{(q)}}  \cdots \paren{L_l^{(q)}}  \: \angbrack{D_l^{(q)}}$ 
be the $q$th copy of $\Tbot$ in $\Tbot^*$.  We substitute the following for $\Tbot^{(q)}$, for all $q$, yielding $\Tsh$.
\[
\paren{L_1^{(q)} a_1} \:  \angbrack{D_1^{(q)}} \cdots \paren{L_l^{(q)} a_l}  \: \angbrack{D_l^{(q)} D_q'}
\]
\begin{figure}
\centering
\scalebox{.28}{\includegraphics{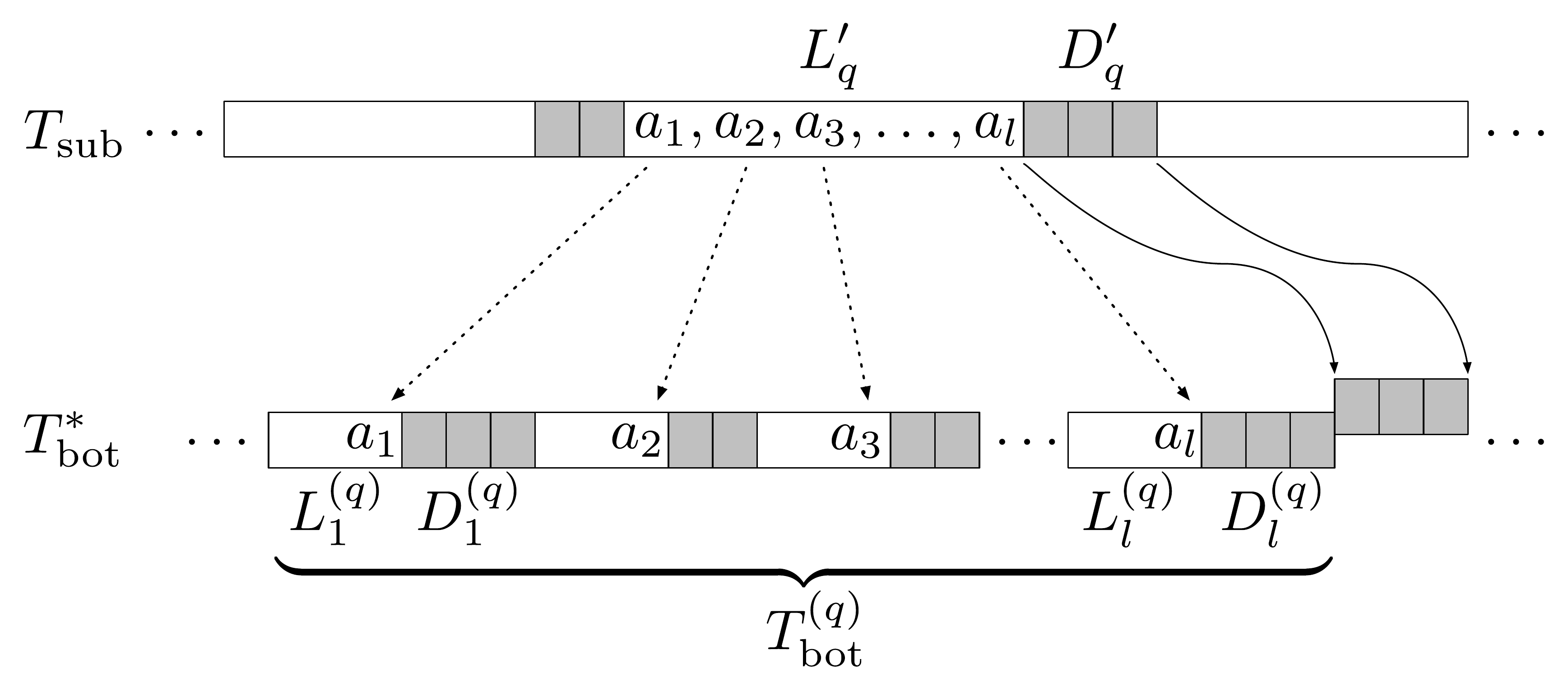}}
\caption{\label{fig:shuffle}Here $L_q' = \paren{a_1\cdots a_l}$ is the $q$th live block of $\Tsub$ and $\Tmid^{(q)}$ is the $q$th copy 
of $\Tmid$ in $\Tmid^*$.  The sequence $\Tsub\postshuffle\Tmid$ is obtained by shuffling $L_q'$ into the live blocks of $\Tmid^{(q)}$
and inserting $D_q'$ after $\Tmid^{(q)}$.}
\end{figure}%
In other words, we insert $a_p$ at the end of the $p$th live block in $\Tbot^{(q)}$ and insert all the dead blocks $D_q'$ following $L_q'$ in $\Tsub$
immediately after $\Tbot^{(q)}$.  See Figure~\ref{fig:shuffle}.  The {\em pre}shuffle $\Tsh = \Tsub \preshuffle \Tbot$ is formed in exactly the same way
except that we insert $a_p$ at the {\em beginning} of the block, that is, we substitute for $\Tbot^{(q)}$ the 
sequence $\paren{a_1 L_1^{(q)}}  \: \angbrack{D_1^{(q)}} \cdots \paren{a_l L_l^{(q)}}  \: \angbrack{D_l^{(q)} D_q'}$.  In this section we consider only postshuffling
whereas both pre- and postshuffling are used in Section~\ref{sect:NMZC}.

\subsection{Construction of the Sequences}\label{sect:PERMlb}

Our $\Perm{r,s+1}$-free sequences are constructed inductively, beginning with $\Perm{r,4}$-free sequences $\{T_{\rho}(i,j)\}_{i\ge 1, j\ge 0, \rho\ge 2}$.
Each $T_{\rho}(i,j)$ consists of a mixture of live and dead blocks.  The parameters $i$ and $j$ control the multiplicity of symbols
and the length of live blocks, respectively.   The length of dead blocks are guaranteed to be a multiple of $\rho$.
This construction is essentially the same as~\cite{Pettie-GenDS11}, and, ignoring the role of $\rho$,
essentially the same as~\cite{HS86,Komjath88,WS88,Pettie-DS-nonlin11}.

\begin{align*}
V(j) &= (1\cdots j) \, \angbrack{j\cdots 1}		& \mbox{ one live block, one dead}\\
T_{\rho}(1,j) &= V(j)\\
T_{\rho}(i,0) &= (\,)^{\rho}					& \mbox{ $\rho\ge 2$ empty live blocks, for $i\ge 2$}\\
T_{\rho}(i,j) &= \zero{\Tsub\postshuffle \Tbot = (\Ttop \compose \Tmid) \postshuffle \Tbot}\\
\mbox{where } & \Tbot = T_{\rho}(i,j-1)\\
			& \Tmid = V(\livebl{\Tbot})\\
			& \Ttop = T_{\rho}(i-1,\livebl{\Tbot})
\end{align*}

Lemma~\ref{lem:T-props} identifies some simple properties of $T_{\rho}(i,j)$ that let us analyze its length and 
forbidden substructures.

\begin{lemma}\label{lem:T-props}
Let $T=T_\rho(i,j)$ for some $\rho\ge 2$.
\begin{enumerate}
\item Live blocks of $T$ consist solely of first occurrences and all first occurrences appear in live blocks.\label{item:T-props1}
\item Live blocks of $T$ have length $j$.\label{item:T-props2}
\item All symbols appear $i+1$ times in $T$.\label{item:T-props3}
\item When $i\ge 2$, the number of live blocks and the length of dead blocks are both multiples of $\rho$.\label{item:T-props4}
\item As a consequence of Parts \ref{item:T-props1}--\ref{item:T-props3}, $|T| = (i+1)\|T\| = (i+1)j\livebl{T}$.
\end{enumerate}
\end{lemma}

\begin{proof}
All the claims trivially hold in the base cases, when $i=1$ or $j=0$.  Assume the claim holds inductively for
pairs lexicographically smaller than $(i,j)$.  
Note that Part~\ref{item:T-props1} holds for $\Tmid = V(\cdot)$.  If it holds for $\Ttop$ and $\Tmid$ it clearly holds for $\Tsub$,
and if it holds for $\Tbot$ as well then it also holds for $T_\rho(i,j) = \Tsub\postshuffle\Tbot$.

Part~\ref{item:T-props2} follows since, by the inductive hypothesis, live blocks in $\Tbot=T_{\rho}(i,j-1)$ have length $j-1$
and exactly one symbol gets shuffled into each live block when forming $T_{\rho}(i,j) = \Tsub\postshuffle \Tbot$.
Part~\ref{item:T-props3} follows since 
the multiplicity of symbols in $\Ttop$ is $i$, by the induction hypothesis, 
and the multiplicity in $V(\cdot)$ is 2, 
so the multiplicity of symbols in $\Tsub$ is $i+1$.  The multiplicity of symbols in $\Tbot$ is already $i+1$, by the induction hypothesis,
so all symbols occur in $T$ with multiplicity $i+1$.

Turning at last to Part~\ref{item:T-props4}, the claim is vacuous when $i=1$ and clearly 
holds when $i\ge 2,j=0$.  In general, if $\livebl{\Tbot}=\livebl{T_{\rho}(i,j-1)}$ is a multiple of $\rho$ then $\livebl{T_{\rho}(i,j)}$ is also a multiple of $\rho$.
All dead blocks in $T_{\rho}(i,j)$ are either (i) inherited from $\Tbot$, or (ii) inherited from $\Ttop$, or (iii) are first introduced in $\Tsub$ as the second
block in a copy of $\Tmid= V(\livebl{\Tbot})$.
The inductive hypothesis implies that the length of category (i) blocks are multiples of $\rho$.
When $i\ge 3$ the inductive hypothesis also implies the length of category (ii) blocks are multiples of $\rho$.
When $i=2$ we have $\Ttop = T_\rho(1,\livebl{\Tbot}) = V(\livebl{\Tbot})$.  By virtue of $\livebl{\Tbot}$ being a multiple of $\rho$,
the length of the lone dead block in $\Ttop$ is a multiple of $\rho$.  Category (iii) blocks satisfy the property for the same reason,
since $\Tmid = V(\livebl{\Tbot})$ and $\livebl{\Tbot}$ is a multiple of $\rho$.
\end{proof}

\begin{lemma}\label{lem:Torder3}
$T_{\rho}(i,j)$ is an order-$3$ DS sequence, and hence $\Perm{r,4}$-free for all $r\ge 2$.
\end{lemma}

\begin{proof}
The claim clearly holds in all base cases, so we can assume $T=T_{\rho}(i,j)$ was formed from $\Ttop,\Tmid,$ and $\Tbot$.
Any occurrence of $ababa$ could not have arisen from a shuffling event.  If $a\in\Sigma(\Ttop)$ and $b\in\Sigma(\Tbot^*)$,
the projection of $T$ onto $\{a,b\}$ is $\left|b^* a b^*\right| a^*$, where the bars
mark the boundary of $b$'s copy of $\Tbot$.  (The live block of $\Tsub$ shuffled into $b$'s $\Tbot$ contains the first occurrence
of $a$.  All other $a$s in $\Tsub$ are inserted after this copy of $\Tbot$.)
We could also not create an occurrence of $ababa$ during a composition event, where $a$ and $b$
shared a live block in $\Ttop$.  The projections of $\Ttop$ and $\Tsub$ onto $\{a,b\}$ would be, respectively, of the form
$(ab) a^*b^*$ and $(ab) \angbrack{ba} a^* b^*$, the latter being $ababa$-free.
\end{proof}

The $U_s(i,j)$ sequences defined below have the property that all blocks are live and have length exactly $j$ and all symbols occur $\mu_{s,i}$ times,
where the $\mu$-values are defined below. 
This contrasts
with $T_\rho(i,j)$, where there is a mixture of live and dead blocks having non-uniform lengths.  We define $U_3(i,j)$
to be identical to $T_j(i,j)$ as a sequence, but we interpret it as a sequence of live blocks of length exactly $j$.
This is possible since, in $T_j(i,j)$, the length of live blocks is $j$ and the length of all dead blocks a multiple of $j$.
Since all blocks in $U_s$ are live we can use the identities $\bl{U_s(i,j)} = \livebl{U_s(i,j)}$ and $|U_s(i,j)| = \mu_{s,i}\|U_s(i,j)\| = j\bl{U_s(i,j)}$.
Sequences essentially the same as $\{U_s\}$ were used in~\cite{Pettie-GenDS11} to prove lower bounds on $\Ex(D_k,n)$, where
$\{D_k\}$ are the two-sided combs defined in Section~\ref{sect:further-generalizations}.

\begin{align*}
U_2(i,j) &= (1\cdots j) \, (j\cdots 1)				& \mbox{ two blocks, for all $i$}\\
U_s(i,1) &= (1)^{\mu_{s,i}}			& \mbox{ $\mu_{s,i}$ identical blocks, for $i\ge 1,s\ge 3$}\\
U_s(0,j) &= (1\cdots j)						& \mbox{ one block, for $s\ge 3$}\\
U_3(i,j) &= T_j(i,j) \; \mbox{ (reinterpreted)}	& \mbox{ for $i\ge 1,$ where $\rho=j\ge 2$}\\
U_s(i,j) &= \zero{\Usub \postshuffle \Ubot	= (\Utop \compose \Umid) \postshuffle \Ubot} \\
\mbox{ where } & \Ubot = U_s(i,j-1)\\
			& \Umid = U_{s-2}(i,\bl{\Ubot})\\
			 & \Utop = U_s(i-1, \|\Umid\|)
\intertext{The multiplicities $\{\mu_{s,i}\}$ are defined as follows.}
\mu_{2,i} &= 2 		& \mbox{ for all $i$}\\
\mu_{3,i} &= i+1		& \mbox{ for all $i$}\\
\mu_{s,0} &= 1		& \mbox{ for all $s\ge 4$}\\
\mu_{s,i} &= \mu_{s,i-1}\mu_{s-2,i} & \mbox{for $s\ge 4$ and $i\ge 1$}\\
\end{align*}

\begin{lemma}\label{lem:Uprops}
Let $U = U_s(i,j)$, where $s\ge 2, i\ge 1, j\ge 1$.
\begin{enumerate}
\item All symbols appear in $U$ with multiplicity precisely $\mu_{s,i}$.\label{item:mult}
\item All blocks in $U$ have length precisely $j$.\label{item:blocklength}
\item If $a$ and $b$ share a common block and $a<b$ according to the canonical ordering of $U$,
then the projection of $U$ onto $\{a,b\}$ has the form\label{item:projection}
either $a^* b^* (ba) b^* a^*$ or $a^* (ab) a^* b^*$.  Moreover, unless $s=2$, every pair of symbols appear in at most one common block.
\end{enumerate}
\end{lemma}

\begin{proof}
Parts~\ref{item:mult} and~\ref{item:blocklength} hold in the base cases and follow easily by induction on $s,i,$ and $j$.
For Part~\ref{item:projection}, if $b$ precedes $a$ in their common block then, in some shuffling event, 
$a\in \Sigma(\Usub)$ was {\em post}shuffled into $b$'s copy of $\Ubot$ and all other copies of $a$ were placed before or after this copy of $\Ubot$,
hence $U$'s projection onto $\{a,b\}$ is $a^* b^* (ba) b^* a^*$.  If $a$ precedes $b$ in their common block then this must be the {\em first}
occurrence of $b$ in $U$ (otherwise $b<a$ in the canonical ordering).  By the same reasoning as above the projection of $U$ onto $\{a,b\}$
must be of the form $a^* (ab) a^* b^*$.
\end{proof}

In Lemma~\ref{lem:perm-U} we analyze the subsequences avoided by $U_s$ and
in Lemma~\ref{lem:Umu} we lower bound the length of $U_s$.

\begin{lemma}\label{lem:perm-U}
When $s=3$ or $s\ge 2$ is even, $U_s$ is an order-$s$ DS sequence and hence $\Perm{2,s+1}$-free.
When $s\ge 5$ is odd and $r\ge 3$, $U_s$ is $\Perm{r,s+1}$-free.
\end{lemma}

\begin{proof}
The claim is clearly true for $s=2$ and Lemma~\ref{lem:Torder3} takes care of $s=3$.  Observe that $ababab$ can never be introduced by a shuffling
event.  If $a\in\Sigma(\Usub)$ and $b\in\Sigma(\Ubot^*)$, only one copy of $a$ can appear between two $b$s; all others precede
or follow $b$'s copy of $\Ubot$ in $\Ubot^*$.  Thus any alternating subsequence $ab\cdots ab$ of length $s+2\ge 6$ must be introduced
in $\Usub=\Utop\compose \Umid$ by composition.  The projection of $\Utop$ onto $\{a,b\}$ is of the form $a^*b^* (ba) b^* a^*$.
Since $\Umid = U_{s-2}(\cdot,\cdot)$ has order $s-2$ and $b$ precedes $a$ in the canonical ordering of $\Umid$, 
its longest alternating subsequence is $bab\cdots ab$ (length $s-1$), hence the longest 
alternating subsequence in $\Usub$ has length $s+1$.

We now consider $U=U_s(i,j)$, where $s\ge 5$ is odd.
Recall that $U$ is regarded as a sequence over the alphabet $\{1,\ldots,\|U\|\}$ in canonical form.
Generalizing our previous terminology, we will say $U$ is $\sigma$-free, where $\Sigma(\sigma) = \{1,\ldots,\|\sigma\|\}$, if $U$ contains no 
subsequences {\em order-isomorphic} to $\sigma$, that is, that are both isomorphic to $\sigma$ and preserve the relative order of symbols in $\sigma$.\footnote{For example, $5678\, 5678$ contains several subsequences isomorphic to 
$2121$, but none are order-isomorphic.  It contains many subsequences order-isomorphic to $1212$ such as $6868$.
We should point out that the concepts of {\em canonical form} and {\em order-isomorphic} were introduced by
none other than Davenport and Schinzel~\cite[p. 691]{DS65}, who noted that order-$s$ DS sequences in canonical form
are $(3(12)^{s/2})$-free, for even $s$, and $(31(21)^{(s-1)/2})$-free, for odd $s$.}
Define $P_{s+1}$ to be the set of $\sigma\in\{1,2,3\}^*$ such that $\dbl(\sigma)$ contains a subsequence
$\sigma_1\sigma_2\cdots \sigma_{s+1}$, where $\sigma_1$ and $\sigma_{s+1}$ are permutations of $\{2,3\}$
and $\sigma_2,\ldots,\sigma_s$ are permutations of $\{1,2,3\}$.\footnote{For example,
$23\: 21\; 23\; 2\in P_{4}$ since doubling the first and second 3 and the first 1 yields a sequence of the desired form.}

We will prove that $U_s(i,j)$ (in canonical form) is $P_{s+1}$-free by induction, which implies that $U_s(i,j)$ is 
$\Perm{r,s+1}$-free for all $r\ge 3$.  The claim holds at $s=3$ since all members of $P_4$ contain $ababa$ as a subsequence,
on the alphabet $\{2,3\}$.
For $s\ge 5$, $P_{s+1}$ could not have arisen from a shuffling event since every member of $P_{s+1}$ contains a sequence isomorphic to $ababab$.
It also could not have arisen from a composition even in which some {\em strict} subset of $\{1,2,3\}$ appears in one block.
Whether this subset is $\{1,2\}$ or $\{2,3\}$ or $\{1,3\}$, 
the 1s can only be involved in two permutations whereas they must be involved in at least four, namely $\sigma_2,\ldots,\sigma_s$.  

We can therefore assume that any $P_{s+1}$ sequence over the alphabet $\{a,b,c\}$ arises from a composition event,
where $a,b,c$ share a common block $B$ in $\Utop$.  
(For reasons that will become clear shortly, it is better to use symbols $a,b,c$ rather than integers $1,2,3$.)
To obtain $\Usub$ we substitute for $B$ a copy $\Umid(B)$
of $\Umid = U_{s-2}(\cdot,\cdot)$.
Without loss of generality $a<b<c$ according to the canonical ordering of $\Utop$.
According to Lemma~\ref{lem:Uprops}(\ref{item:projection}) 
the projection of $\Utop$ onto $\{a,b,c\}$, ignoring immediate repetitions, 
is either
\begin{enumerate}
\item [(i)] $abc (cba) cba$, or
\item [(ii)] $ab (bca) bca$, or
\item [(iii)] $ab (bac) bac$, or
\item [(iv)] $a (abc) abc$.
\end{enumerate}
That is, in cases (ii)--(iv) $B$ contains the first $c$ in $\Utop$ and in case (iv) $B$ also 
contains the first $b$ in $\Utop$.
In case (i) $c<b<a$ according to the canonical ordering of $\Umid(B)$.  
In order for $\Usub$ to contain a $P_{s+1}$ sequence we would need $\Umid(B)$ to contain
\[
\{ab\}\: \overbrace{\{abc\} \cdots \{abc\}}^{s-3} \: \{ab\},
\]
where the curly brackets indicate arbitrary permutations of the enclosed sequences.
(The $\{ab\}$ permutations on either end can be extended to permutations on $\{abc\}$
by borrowing the $c$s adjacent to $B$ in $\Utop$.)
In cases (ii) and (iii), $b < a,c$ according to the canonical ordering of $\Umid(B)$, so 
for $\Usub$ to contain a $P_{s+1}$ sequence, $\Umid(B)$ must contain
\[
\{c\} \: \overbrace{\{abc\} \cdots \{abc\}}^{s-2} \: \{ac\}.
\]
Once again, the permutations on $\{c\}$ and $\{ac\}$ on either end can be extended to $\{bc\}$ and $\{abc\}$
by borrowing the $b$s on either side of $B$.
In case (iv) we have $a<b<c$ according to the canonical ordering of $\Umid(B)$,
which, by the same reasoning, would need to contain
\[
\{bc\}\: \overbrace{\{abc\}\cdots\{abc\}}^{s-2} \: \{bc\}
\]
None of cases (i)--(iv) is possible since $\Umid = U_{s-2}$ is $P_{s-1}$-free, by the induction hypothesis.
\end{proof}

\begin{remark}\label{rem:Ps-1}
Notice that in the proof of Lemma~\ref{lem:perm-U}, 
``$P_{s-1}$-freeness'' is defined with respect to the canonical ordering on $\{a,b,c\}$ \underline{in $\Umid$},
which is identical to their ordering in $B$.  Although $a<b<c$ with respect to $\Utop$, 
identifying $a,b,$ and $c$ with 1,2, and 3 would
be confusing as their canonical ordering is typically different in $\Umid$.
\end{remark}

We have established that $U_s$ is $\Perm{r,s+1}$-free and now need to lower bound its length.

\begin{lemma}\label{lem:Umu}
Fix $s$ and let $t=\floor{(s-2)/2}$.
\begin{enumerate}
\item For even $s$, $\mu_{s,i} = 2^{i+t-1\choose t} = 2^{i^t/t! \,+\, O(i^{t-1})}$.
\item For odd $s$, $\mu_{s,i} = \prod_{l=0}^{i} (i+1 - l)^{l + t-1 \choose t-1}  =  2^{i^t(\log i)/t! \,+\, O(i^t)}$.
\end{enumerate}
\end{lemma}

\begin{proof}
Consider the even case first.  When $i=0$ we have $\mu_{s,0} = 1 = 2^{0 + t-1 \choose t}$
and when $s=2,t=0$ we have $\mu_{2,i} = 2^{i+0-1\choose 0} = 2$.
The claim holds for all even $s\ge 4$ 
since, by Pascal's identity,
$\mu_{s,i} = \mu_{s,i-1}\cdot\mu_{s-2,i} = 2^{{(i-1) + t-1 \choose t} + {i + (t-1) -1 \choose t-1}} = 2^{i + t-1 \choose t}$.
Clearly $2^{i + t-1 \choose t} \geq 2^{i^t/t!}$.

For odd $s$ the base case $i=0$ is trivial.  When $s=5,t=1$
we have
$\mu_{5,i} = \mu_{3,i}\mu_{3,i-1}\cdots \mu_{3,0} = (i+1)!$, which can be expressed as
$\prod_{l = 0}^{i} (i +1 - l)^{l + t-1 \choose t-1}$ since $t=1$ and ${l + 0 \choose 0} = 1$ for all $l$.
For odd $s\ge 7$ the bound follows by induction.
\begin{align*}
\lefteqn{\mu_{s,i} = \mu_{s,i-1}\cdot \mu_{s-2,i} }\\
	&= \prod_{l = 0}^{i-1} ((i-1) + 1 - l)^{l + t-1 \choose t-1}  \cdot  \prod_{l' = 0}^{i} (i+1-l')^{l'+t-2 \choose t-2}\\
	&= \prod_{l'' = 0}^{i} (i + 1 - l'')^{l'' + t-2 \choose t-1}  \cdot  \prod_{l' = 0}^{i} (i+1-l')^{l'+t-2 \choose t-2} 
			& \mbox{\{$l'' \bydef l+1$.  When $l''=0$, $(i+1)^{{t-2}\choose t-1} = 1$.\}}\\
	&= \prod_{l = 0}^{i} (i+1-l)^{{l+t-2\choose t-1} + {l+t-2\choose t-2}}
	\;\, = \;\, \prod_{l=0}^{i} \zero{(i+1-l)^{{l+t-1}\choose t-1}}
\end{align*}
When $s$ is odd, it is simpler to obtain asymptotic bounds on $\log_2(\mu_{s,i})$ 
directly, without analyzing the closed-form expression above.
Assuming inductively that $\log_2(\mu_{s-2,i}) = i^{t-1}(\log i)/(t-1)! \,+\, O(i^{t-2})$, where the constant hidden in the second
term depends on $s-2$, we have 
\begin{align*}
\log_2(\mu_{s,i}) = \log_2(\mu_{s-2,i}) + \log_2(\mu_{s,i-1}) &= \sum_{x=1}^{i} \log_2(\mu_{s-2,x})\\ 
&= \sum_{x=1}^i \SqBrack{\f{x^{t-1}\log x}{(t-1)!} + O(x^{t-2})}\\
&= \f{i^t\log i}{t!} + O(x^{t-1}).
\end{align*}
Note that the sum is faithfully approximated by the integral
$\int_0^{i} x^{t-1}(\log x)/(t-1)! + O(x^{t-2}) \,\mathrm{d}x = i^t(\log i)/t! + O(i^{t-1})$ as the two
differ by $O(i^{t-1})$.
\end{proof}

It is a tedious exercise to show that for $n=\|U_s(i,j)\|$ and $m=\bl{U_s(i,j)}$, $i = \alpha(n,m)+O(1)$ and $i=\alpha(n)+O(1)$ when $j=O(1)$.
(See~\cite{Nivasch10,Pettie-GenDS11} for several examples of such calculations.)
Lemmas~\ref{lem:Torder3}, \ref{lem:perm-U}, and \ref{lem:Umu} establish all the lower bounds
of Theorem~\ref{thm:PERM}, with the exception of $\DS{5}(n) = \Omega(n\alpha(n)2^{\alpha(n)})$, which
is proved in~\cite{Pettie-SoCG13}.

\begin{remark}\label{rem:Nivasch-order-3}
It should be possible to improve the lower  bounds on $\PERM{3,s}$, for odd $s\ge 5$, by substituting
Nivasch's construction of order-3 DS sequences~\cite[\S 6]{Nivasch10} for $T_{j}(i,j)$ in the definition of $U_3(i,j)$.
Nivasch's sequences are roughly twice as long as $T_j(i,j)$, which would lead to a 
$2^{i+O(1)\choose t}$ factor improvement in $\mu_{s,i}$, for odd $s\ge 5$.
The only technical issue is to deal with non-uniform block lengths.  In the~\cite{Nivasch10} construction there is
no straightforward way to force dead blocks to have lengths that are multiples of some $\rho$.  As a consequence, the
block lengths in $U_s(i,j)$ would also be non-uniform, but upper bounded by $j$.
\end{remark}

%
%
%
%
%
%
%
%
%
%
%
%
%
\section{Sequence Transformations and Decompositions}\label{sect:basic}

This section reviews some basic results and notation that is used throughout the article, 
sometimes without direct reference.

\subsection{Sparse Versus Blocked Sequences}\label{sect:SparseVersusBlocked}

An $m$-block sequence can easily be converted to an $r$-sparse one by removing up to $r-1$ symbols in each block, except the first.
This shows, for example, that $\DS{s}(n,m) \le \DS{s}(n) + m-1$ and $\dblPERM{r,s}(n,m) \le \dblPERM{r,s}(n) + (r-1)(m-1)$.
However, converting an $r$-sparse sequence into one with $O(n)$ blocks is, in general, not known to be possible without suffering
some asymptotic loss.  The following lemma generalizes reductions of Sharir~\cite{Sharir87} and Pettie~\cite{Pettie-SoCG13}
to $\dblDS{s},\PERM{r,s},$ and $\dblPERM{r,s}$.  In the interest of completeness we include a proof in Appendix~\ref{appendix:proofs}.

\begin{lemma}\label{lem:SparseVersusBlocked}
(Cf.~Sharir~\cite{Sharir87}, \Furedi{} and Hajnal~\cite{FurediH92}, and Pettie~\cite{Pettie-SoCG13}.)
Define $\gamma_s,\dblgamma_s,\gamma_{r,s},\dblgamma_{r,s} : \mathbb{N}\rightarrow\mathbb{N}$ 
to be non-decreasing functions bounding the leading factors of $\DS{s}(n),\dblDS{s}(n),\PERM{r,s}(n),$ and $\dblPERM{r,s}(n)$,
e.g., $\dblPERM{r,s} \le \dblgamma_{r,s}(n) \cdot n$.  The following bounds hold.
\begin{equation*}
\begin{array}{r@{\hcm[.1]}l@{\hcm[1]}r@{\hcm[.1]}l}
\DS{s}(n) &\le \gamma_{s-2}(n)\cdot \DS{s}(n,2n)
& \dblDS{s}(n) &\le (\dblgamma_{s-2}(n)+4)\cdot \dblDS{s}(n,2n)\istrut[2.5]{0}\\
\DS{s}(n) &\le \gamma_{s-2}(\gamma_{s}(n))\cdot \DS{s}(n,3n)
& \dblDS{s}(n) &\le (\dblgamma_{s-2}(\dblgamma_{s}(n))+4)\cdot \dblDS{s}(n,3n)\istrut[4]{0}\\
\PERM{r,s}(n) &\le \gamma_{r,s-2}(n)\cdot \PERM{r,s}(n,2n) + 2n
& \dblPERM{r,s}(n) &\le (\dblgamma_{r,s-2}(n)+O(1))\cdot \dblPERM{s}(n,2n))\istrut[2.5]{0}\\
\PERM{r,s}(n) &\le \gamma_{r,s-2}(\gamma_{r,s}(n))\cdot \PERM{r,s}(n,3n) + 2n
& \dblPERM{r,s}(n) &\le (\dblgamma_{r,s-2}(\dblgamma_{r,s}(n))+O(1))\cdot \dblPERM{s}(n,3n)),\\
\end{array}
\end{equation*}
where the $O(1)$ term in the last two inequalities depends on $r$ and $s$.
\end{lemma}

\subsection{Reductions Between Perm-Free Sequences and DS Sequences}

It is not immediate from the definitions that $\DS{s}(n) = \Theta(\PERM{2,s}(n))$
and $\dblDS{s}(n) = \Theta(\dblPERM{2,s}(n))$.  These functions are, in fact, asymptotically
equivalent.  Refer to Appendix~\ref{appendix:proofs} for proof of Lemma~\ref{lem:DSPERM}.

\begin{lemma}\label{lem:DSPERM}
The extremal functions for order-$s$ (double) Davenport-Schinzel sequences
and $\Perm{2,s+1}$-free ($\dblPerm{2,s+1}$-free) sequences are equivalent up to constant factors.
In particular,
\begin{equation*}
\begin{array}{rcl}
\DS{s}(n) \le& \PERM{2,s}(n) &< 3\cdot \DS{s}(n) + 2n\\
\DS{s}(n,m) \le& \PERM{2,s}(n,m) &< 2\cdot \DS{s}(n,m) + n\\
\dblDS{s}(n) \le& \dblPERM{2,s}(n) &< 5\cdot \dblDS{s}(n) + 4n\\
\dblDS{s}(n,m) \le& \dblPERM{2,s}(n,m) &< 3\cdot \dblDS{s}(n,m) + 2n
\end{array}
\end{equation*}
\end{lemma}

Given these equivalences, we will only prove upper bounds on $\dblDS{s}$ and not discuss $\dblPERM{2,s}$.

\subsection{Linearity at Orders 1 and 2}\label{sect:order12}

We bound the length of sequences inductively through the use of recurrences.  The induction bottoms out when $s\in\{1,2\}$,
so we need to handle these two orders directly.  Lemma~\ref{lem:orders12} summarizes known linear bounds on $\DS{s},\dblDS{s},\PERM{r,s},$ and $\dblPERM{r,s}$ when $s\le 2$.  A proof of Lemma~\ref{lem:orders12} appears in Appendix~\ref{appendix:proofs}.

\begin{lemma}\label{lem:orders12}
At orders $s=1$ and $s=2$, 
the extremal functions 
$\DS{s},\dblDS{s},\PERM{r,s},$ 
and $\dblPERM{r,s}$ obey the following.
\begin{equation*}
\begin{array}{r@{\hcm[.1]}l@{\hcm[1]}r@{\hcm[.1]}l@{\hcm[1]}r}
\DS{1}(n) &= n		& \DS{1}(n,m) &= n+ m-1\\
\DS{2}(n) &= 2n-1	& \DS{2}(n,m) &= 2n+m-2		& \mbox{(Davenport-Schinzel~\cite{DS65})}\\
\dblDS{1}(n) &= 3n-2	& \dblDS{1}(n,m) &= 2n+m-2	& \mbox{(Dav.-Sch.~\cite{DS65b},Klazar~\cite{Klazar02})}\\
\dblDS{2}(n) &< 8n  & \dblDS{2}(n,m) &< 5n+m		& \mbox{(Klazar~\cite{Klazar96}, \Furedi-Hajnal~\cite{FurediH92})}\\
\PERM{r,1}(n) &= \dblPERM{r,1}(n) < rn & \PERM{r,1}(n,m) &= \zero{\dblPERM{r,1}(n,m) < n + (r-1)m} & \mbox{(Klazar~\cite{Klazar92})}\\
\PERM{r,2}(n) &< 2rn & \PERM{r,2}(n,m) &< 2n + (r-1)m & \mbox{(Klazar~\cite{Klazar92})}\\
\dblPERM{r,2}(n) &< 6^rrn    & \dblPERM{r,2}(n,m) &< 2\cdot 6^{r-1}(n + m/3) & \mbox{(Pettie~\cite{Pettie-SoCG11}, cf.~\cite{KV94})}
\end{array}
\end{equation*}
\end{lemma}

The linear bound on $\dblPERM{r,2}$ is a consequence of bounds on $\dbl(N_{r-1})$-free sequences~\cite{KV94,Pettie-SoCG11},
though this connection was not noted earlier~\cite{CibulkaK12}.

\subsection{Sequence Decomposition}\label{sect:sequence-decomposition}

We adopt and extend the sequence decomposition notation from~\cite{Pettie-SoCG13}.
This style of decomposition goes back to Hart and Sharir~\cite{HS86} and Agarwal, Sharir, and Shor~\cite{ASS89},
and has been used many times since then~\cite{Klazar92,Nivasch10,Pettie-GenDS11,CibulkaK12}.
This notation is used liberally throughout Sections~\ref{sect:dblPERM}--\ref{sect:dblDS}.

Let $S$ be a sequence over an $n=\|S\|$ letter alphabet consisting of $m=\bl{S}$ blocks.
(It may be that $S$ avoids some forbidden sequences, but this has no bearing on the decomposition.)
A partition of $S$ into $\Gm$ intervals $S_1\cdots S_{\Gm}$ is called {\em uniform}
if $m_1 = \cdots = m_{\Gm-1}$ are equal powers of two and $m_{\Gm}$ may be smaller,
where $m_q = \bl{S_q}$ is the number of blocks in the $q$th interval.
A symbol is {\em global} if it appears in multiple intervals and {\em local} otherwise.
Let $\LS = \LS_1\cdots\LS_{\Gm}$ and $\GS = \GS_1\cdots \GS_{\Gm}$
be the projections of $S$ onto local and global symbols, so $|S| = |\LS| + |\GS|$.
Define $\Gn = \|\GS\|$ to be the size of the global alphabet and 
$\Gn_q = \|\GS_q\|$ and $\Ln_q = \|\LS_q\|$ to be number of global and local symbols in $\Sigma(S_q)$,
so $n = \Gn + \sum_{1\le q\le \Gm} \Ln_q$.

A global symbol $a\in\Sigma(\GS_q)$ is classified as {\em first}, {\em last}, or {\em middle}
if no $a$s appear before $S_q$, no $a$s appear after $S_q$, or $a$s appear both before and after $S_q$.%
\footnote{Note that if $a\in \Sigma(\GS_q)$ is classified as first, all of the possibly {\em many} occurrences of $a$ in $S_q$ are ``first'' occurrences.}
Let $\GfS_q,\GlS_q,\GmS_q\subseq \GS_q$ be the projections of $\GS_q$ onto symbols classified as first, last, and middle in $\GS_q$;
let $\Gfn_q,\Gln_q,$ and $\Gmn_q$ be the sizes of the alphabets $\Sigma(\GfS_q),\Sigma(\GlS_q),$ and $\Sigma(\GmS_q)$.
Define $\GfS,\GlS,$ and $\GmS$ to be subsequences of first, last, and middle occurrences, namely
\begin{align*}
\GfS &= \mbox{$\GfS_1$}\,\GfS_2\,\cdots\,\GfS_{\Gm-1}\\
\GlS &= \mbox{\hphantom{$\GfS_1$}}\,\GlS_2\,\cdots\,\GlS_{\Gm-1}\,\GlS_{\Gm}\\
\GmS &= \mbox{\hphantom{$\GfS_1$}}\,\GmS_2\,\cdots\,\GmS_{\Gm-1}
\end{align*}
Note that $\GS_1=\GfS_1$ consists solely of first occurrences
and $\GS_{\Gm} = \GlS_{\Gm}$ consists solely of last occurrences, so $\GmS$ is empty if $\Gm=2$.
These notational conventions will be applied to sequences and other objects defined later.  For example, the diacritical marks
$\check{\ },\, \hat{\ },\, \acute{\ },\, \grave{\ },$ and $\bar{\ }$ will be applied to objects pertaining to local, global, first, last, and middle symbols, respectively.
Moreover, whenever we define a new subsequence of $S_q$, say $\GfeatherS_q$, quantities and objects pertaining to $\GfeatherS_q$
will be indicated with the same diacritical mark, such as $\Gnf_q = \|\GfeatherS_q\|$.

The global contracted sequence $\GS' = B_1\cdots B_{\Gm}$ is obtained by {\em contracting} each interval
$\GS_q$ to a single block $B_q$ consisting of some permutation of $\Sigma(\GS_q)$.  Unless specified otherwise, the symbols in $B_q$
are ordered according to their {\em first} occurrence in $\GS_q$.  It follows that $\GS' \subseq \GS$, so $\GS'$ inherits
any forbidden sequences of $\GS$.

\section{Upper Bounds on $\dblPerm{r,s}$-free Sequences}\label{sect:dblPERM}

In this section we give recurrences for the extremal functions of 
$\Perm{r,s+1}$-free sequences and $\dblPerm{r,s+1}$-free sequences.
Lemmas~\ref{lem:ub-PERM-dblPERM} and \ref{lem:closed-form-K-dblK} give closed-form upper bounds on the length of such sequences
in terms of Ackermann's function.  These bounds on $\PERM{r,s}$ and $\dblPERM{r,s}$ are sharp,
except for $\PERM{2,s}$ and $\dblPERM{2,s}$, when $s\ge 5$ is odd, and $\dblPERM{r,3}$, for any $r\ge 2$.
These exceptions are addressed in Sections~\ref{sect:dblPERMfour} and \ref{sect:dblDS}.

\subsection{A Recurrence for $\PERM{r,s}$}

In reading the proofs of Recurrences~\ref{rec:PERM} and \ref{rec:dblPERM} one should keep in mind
that all extremal functions are superadditive.  For example,
\[
\PERM{r,s}(n_1,m_1) + \PERM{r,s}(n_2,m_2) \le \PERM{r,s}(n_1+n_2, m_1+m_2)
\]

\begin{recurrence}\label{rec:PERM}
Define $n$ and $m$ to be the alphabet size and block count parameters.
For any $\Gm\ge 2$, any block partition $\{\Lm_q\}_{1\le q\le \Gm}$, and any alphabet partition $\{\Gn\}\cup \{\Ln_q\}_{1\le q\le \Gm}$,
$\PERM{r,s}$ obeys the following recurrences, for any fixed $r\ge 2,s\ge 3$.
\begin{align*}
\intertext{When $\Gm=2$,}
\PERM{r,s}(n,m) &\le \sum_{q\in\{1,2\}} \PERM{r,s}(\Ln_q,\Lm_q) + \PERM{r,s-1}(2\Gn,m)\\
\intertext{and when $\Gm>2$,}
\PERM{r,s}(n,m) 
&\leq\; \sum_{q=1}^{\Gm} \PERM{r,s}(\Ln_q,\Lm_q)
 \;+\; 2\cdot \PERM{r,s-1}(\Gn,m)
 \;+\; \PERM{r,s-2}(\PERM{r,s}(\Gn,\Gm) - 2\Gn, m).
\end{align*}
\end{recurrence}

\begin{proof}
We adopt the sequence decomposition notation from Section~\ref{sect:sequence-decomposition}.
The contribution of local symbols is $\sum_q |\LS_q| \le \sum_q \PERM{r,s}(\Ln_q,m_q)$.
As each symbol in $\GfS_q$ appears at least once after $S_q$, each $\GfS_q$
is a $\Perm{r,s}$-free sequence, it follows that
\[
\sum_{q=1}^{\Gm-1} |\GfS_q| \le \sum_{q=1}^{\Gm-1} \PERM{r,s-1}(\Gfn_q,m_q) \le \PERM{r,s-1}\paren{\sum_{q=1}^{\Gm-1} \Gfn_q, \sum_{q=1}^{\Gm-1} m_q} = \PERM{r,s-1}(\Gn,m-m_{\Gm}).
\]
A symmetric statement is true for each $\GlS_q$, hence the contribution of last occurrences is
$\sum_q |\GlS_q| \le \PERM{r,s-1}(\Gn,m-m_1)$.  If $\Gm=2$ then we have accounted for all symbols, and by superadditivity
$\PERM{r,s-1}(\Gn,m_1) + \PERM{r,s-1}(\Gn,m_2) \le \PERM{r,s-1}(2\Gn,m)$.

If $\Gm>2$ then we must also count middle symbols.  
Each symbol in $\GmS_q$ appears at least once before $\GmS_q$ and at least once afterward.
This implies that $\GmS_q$ is $\Perm{r,s-1}$-free, hence
\begin{align}
\sum_q |\GmS_q| &\le \sum_q \PERM{r,s-2}(\Gmn_q,m_q)\nonumber\\
				&\le \PERM{r,s-2}\paren{\sum_q\Gmn_q, \sum_q m_q} & \mbox{superadditivity}\nonumber\\
				&= \PERM{r,s-2}(|\GS'| - 2\Gn,m-m_1-m_{\Gm})	& \label{ln:sumofmiddlesymbols}\\
				&< \PERM{r,s-2}(\PERM{r,s}(\Gn,\Gm) - 2\Gn,m) & \mbox{$\GS'$ is $\Perm{r,s+1}$-free}\nonumber
\end{align}
Equality~(\ref{ln:sumofmiddlesymbols}) follows since $\sum_q\Gmn_q$ counts the number of middle occurrences of symbols in $\GS'$, that is,
the length of $\GS'$ less $2\Gn$ for first and last occurrences.
\end{proof}

\subsection{A Recurrence for $\dblPERM{r,s}$}

Recall that $\dblPERM{r,s}(n,m)$ was defined to be the extremal function for $\dblPerm{r,s+1}$-free, $m$-block sequences over an
$n$-letter alphabet.
Here $\dblPerm{r,s+1}$ is the set of sequences over the alphabet $[r]=\{1,\ldots,r\}$
of the form $\sigma_1\cdots\sigma_{s+1}$, where $\sigma_1$ and $\sigma_{s+1}$ contain one occurrence of each symbol in $[r]$ and $\sigma_2,\ldots,\sigma_s$ contain exactly two occurrences of each symbol in $[r]$.

\begin{remark}\label{rem:interleave}
The definition of $\dblPERM{r,s}$ has one annoying property.
Suppose $S$ is a sequence and $S'$ a contracted version of it in which each occurrence of a symbol represents two or more occurrences in $S$.
We would like to say that if $S$ is $\dblPerm{r,s+1}$-free then $S'$ is $\Perm{r,s+1}$-free, but this is not strictly true.   For example,
suppose $S'$ contained the $\Perm{2,4}$ sequence $ab\: {\big |} \:b (a \:{\big |}\: b) a\:{\big |} \:ab$, 
where the bars separate the four constituent permutations over $\{a,b\}$ and the parentheses mark the boundaries of one block $B$ in $S'$.
If we substitute $aa$ and $bb$ for all $a$s and $b$s outside $B$, and substitute $abab$ for $B$, we find that $S$ may only contain
$aabb\: bb\, (abab)\, aa\: aabb$, which contains no $\dblPerm{2,4}$ sequence. On the other hand, if occurrences in $S'$ represent
at least {\em three} occurrences in $S$, and symbols in the blocks of $S'$ are sorted according to the {\em 2nd} occurrence in the corresponding
subsequence of $S$, then $S'$ is $\Perm{r,s+1}$ free if $S$ is $\dblPerm{r,s+1}$-free.  

We can easily ``force'' blocks in $S'$ to represent at least three corresponding occurrences in the original sequence.
Suppose we are given an initial $\dblPerm{r,s+1}$-free sequence $S^\star$.  Obtain $S$ from $S^\star$ by retaining every other
occurrence of each symbol, so $S$ is also $\dblPerm{r,s+1}$-free and $|S|\ge |S^\star|/2$.
When bounding $|S|$ inductively we may construct a contracted version
$S'$ whose occurrences represent at least two occurrences in $S$, and hence at least three occurrences in $S^\star$.
(One subtlety here is that $S'$ will be a subsequence of $S^\star$, not necessarily $S$, 
since we order symbols in the blocks of $S'$ according to their position in $S^\star$.)

In Recurrence~\ref{rec:dblPERM} (and Recurrences~\ref{rec:dblPERMfour} and \ref{rec:dblDS} later on)
we use the inference $[S$ is $\dbl(\sigma)$-free$]\rightarrow[S'$ is $\sigma$-free$]$,
knowing that the bounds we obtain on the given extremal function may be off by a factor of two.
\end{remark}

\begin{recurrence}\label{rec:dblPERM}
Define $n$ and $m$ to be the alphabet size and block count parameters.
For any $\Gm\ge 2$, block partition $\{\Lm_q\}_{1\le q\le \Gm}$, and alphabet partition $\{\Gn\}\cup \{\Ln_q\}_{1\le q\le \Gm}$,
$\dblPERM{r,s}$ obeys the following recurrences, for any fixed $r\ge 2,s\ge 3$.
\begin{align*}
\intertext{When $\Gm=2$,}
\dblPERM{r,s}(n,m) &\;\leq\; \sum_{q\in\{1,2\}} \dblPERM{r,s}(\Ln_q,m_q) \;+\; \dblPERM{r,s-1}(2\Gn,m) \;+\; 2\Gn
\intertext{and when $\Gm>2$,}
\dblPERM{r,s}(n,m) &\;\leq\; \sum_{q=1}^{\Gm} \dblPERM{r,s}(\Ln_q,\Lm_q) 
\;+\; \dblPERM{r,s}(\Gn,\Gm)
\;+\; 2\cdot \dblPERM{r,s-1}(\Gn,m)\\
&\hcm[2] +\, \dblPERM{r,s-2}(\PERM{r,s}(\Gn,\Gm)-2\Gn, m) \,+\, 2\cdot \PERM{r,s}(\Gn,\Gm)
\end{align*}
\end{recurrence}

\begin{proof}
We consider the case when $\Gm>2$ first.
Let $S$ be a $\dblPerm{r,s+1}$-free sequence.
The contribution of local symbols is $\sum_q |\LS_q| \le \sum_q \dblPERM{r,s}(\Ln_q,m_q)$.
If a global symbol appears exactly once in some $\GS_q$ that occurrence is called a {\em singleton}.
Let $\GSsingle$ be the subsequence of $\GS$ consisting of singletons.  Clearly $\GSsingle$ can be partitioned
into $\Gm$ blocks, hence $|\GSsingle| \le \dblPERM{r,s}(\Gn,\Gm)$.
Remove all singleton occurrences from $\GS$ and let $\GSnonsingle$ be what remains.  Classify occurrences
in $\GSnonsingle_q$ as {\em first}, {\em middle}, and {\em last} according to whether they do not occur before, do not occur after,
or occur both before and after interval $q$ in $\GS$ (not in $\GSnonsingle$.)  Let $\GfS,\GlS,\GmS\subseq\GSnonsingle$ be the subsequences
of first, last, and middle occurrences.  Obtain $\GfS_q^-$ (and $\GlS_q^-$) from $\GfS_q$ (and $\GlS_q$) 
by removing the last (and first) occurrence of each symbol, and obtain $\GmS_q^-$ from $\GmS_q$ by removing
both the first and last occurrence of each symbol.  It follows that both $\GfS_q^-$ and $\GlS_q^-$ are $\dblPerm{r,s}$-free,
and that $\GmS_q^{-}$ is $\dblPerm{r,s-1}$-free.  
The contribution of first and last non-singleton occurrences in $\GSnonsingle$ is therefore at most
\[
\sum_q \SqBrack{\dblPERM{r,s-1}(\Gfn_q,m_q) + \Gfn_q + \dblPERM{r,s-1}(\Gln_q,m_q) + \Gln_q} 
\;\le\; 
2\cdot \SqBrack{\dblPERM{r,s-1}(\Gn,m) + \Gn}.
\]
Form $\GSnonsingle'$ from $\GSnonsingle$ by contracting each interval
into a single block.  Since $\GSnonsingle$ is $\dblPerm{r,s+1}$-free, $\GSnonsingle'$ must be $\Perm{r,s+1}$.  (See Remark~\ref{rem:interleave}.)
Therefore, the contribution of middle non-singleton occurrences is at most
\begin{align*}
\sum_q \SqBrack{\dblPERM{r,s-2}(\Gmn_q, m_q) + 2\Gmn_q}
&\le \dblPERM{r,s-2}\paren{\sum_q \Gmn_q, \sum_q m_q} + 2\cdot \sum_q \Gmn_q\\
&= \dblPERM{r,s-2}(|\GSnonsingle'|-2\Gn,m) + 2(|\GSnonsingle'|-2\Gn)\\
&\le \dblPERM{r,s-2}(\PERM{r,s}(\Gn,\Gm)-2\Gn,m) + 2\cdot \PERM{r,s}(\Gn,\Gm) - 4\Gn.
\end{align*}

When $\Gm=2$ there are no middle occurrences and, in the worst case, no singletons.  The total number
of first and last occurrences is $(\dblPERM{r,s-1}(\Gn,m_1) + \Gn) + (\dblPERM{r,s-1}(\Gn,m_2)+\Gn) \le \dblPERM{r,s-1}(2\Gn,m) + 2\Gn$.
This concludes the proof of the recurrence.  
\end{proof}

Lemma~\ref{lem:ub-PERM-dblPERM} gives explicit upper bounds on $\PERM{r,s}$ and $\dblPERM{r,s}$ in terms
of inductively defined coefficients $\{\K{s,i},\dblK{s,i}\}$ and the $i$th row-inverse of Ackermann's function.  
One should keep in mind, when reading this lemma and similar
lemmas, that we will ultimately substitute $\alpha(n,m)+O(1)$ for $i$, and that this choice makes the 
dependence on the block count $m$ negligible.

\begin{lemma}\label{lem:ub-PERM-dblPERM}
Fix parameters $i\ge 1$, $r\ge 2$, $s\ge 3$, and $c\ge s-2$.
Let $n,m$ be the alphabet size and block count and let $j$ be minimal such that $m \le (a_{i,j})^c$.
Then $\PERM{r,s}$ and $\dblPERM{r,s}$ are bounded as follows.
\begin{align*}
\PERM{r,s}(n,m) &\le \K{s,i}\paren{n + O((cj)^{s-2}m)}\\
\dblPERM{r,s}(n,m) &\le \dblK{s,i}\paren{n + O((cj)^{s-2}m)},
\end{align*}
where the asymptotic notation hides a constant depending only on $r$.  The coefficients $\{\K{s,i},\dblK{s,i}\}$ are defined as follows.
\begin{align}
\K{1,i}=\dblK{1,i} &= 1\nonumber\\
\K{2,i} &= 2\nonumber\\
\K{s,1} &= 2\K{s-1,1} = 2^{s-1}\nonumber\\
\K{s,i} &= 2\K{s-1,i} + \K{s-2,i}(\K{s,i-1}-2)\label{rec:K}\\
\dblK{2,i} &= 2\cdot 6^{r-1}\nonumber\\
\dblK{s,1} &= 2\dblK{s-1,1}+1 < (6^{r-1}+1)2^{s}\nonumber\\
\dblK{s,i} &= \dblK{s,i-1} + 2\dblK{s-1,i} +  (\dblK{s-2,i} + 2)\K{s,i-1}\label{rec:dblK}
\end{align}
\end{lemma}

The proof is by induction over tuples $(s,i,j)$, where $c$ and $r$ are regarded as fixed.
(The base cases when $s\in\{1,2\}$ follow from Lemma~\ref{lem:orders12}.)
At the base case $i=1$ we let $j$ be minimal such that $m\le a_{1,j}$.  
By invoking Recurrence~\ref{rec:PERM} with $\Gm=2$ is it easy to show that
$\PERM{r,s}(n,m) \le \K{s,1}(n + O(j^{s-2}m))$, where the constant hidden by the asymptotic notation does not depend on $s$ or $c$.
This also implies that 
$\PERM{r,s}(n,m) \le \K{s,1}(n + O((cj)^{s-2}m))$ when $j$ is defined to be minimal such that $m\le a_{1,j}^c$, since $a_{1,j}^c = a_{1,cj} = 2^{cj}$.
In the general case, when $i>1$, we apply Recurrence \ref{rec:PERM} 
using a uniform block partition with width $w^c = a_{i,j-1}^c$,
so
\[
\Gm = \ceil{m/w^c} \le (a_{i,j})^c/(a_{i,j-1})^c = (a_{i-1,w})^c.
\]
We invoke the inductive hypothesis with parameters $i,j-1$ on sequences with $w^c$ blocks (namely $\{\LS_q\}$).  
On sequences with $m$ blocks (such as $\GfS,\GlS$) we invoke the inductive hypothesis with $i,j$ and on sequences
with $\Gm$ blocks we invoke it with $i-1,w$.
The induction goes through smoothly so long as the coefficients 
$\{\K{s,i},\dblK{s,i}\}$ are defined as in Lemma~\ref{lem:ub-PERM-dblPERM}, Eqns.~(\ref{rec:K},\ref{rec:dblK}).
See~\cite[Appendices B and C]{Pettie-SoCG13} for several examples of such proofs in this style.\footnote{For an alternative approach
see Nivasch~\cite[\S 3]{Nivasch10}.  It differs in two respects.  First, it refers to the 
slowly growing row-inverses of Ackermann's function rather than using the `$j$' parameter of Ackermann's function.  
Second, there is no equivalent to our `$c$' parameter in~\cite{Nivasch10}, which leads to a system of {\em two} recurrences,
one for the leading factor of the $n$ term, and one for the leading factor of the $j^{s-2}m$ term.
For yet another style of analysis, which leads to the same recurrences for $\K{s,i}$ and $\dblK{s,i}$,
see Nivasch~\cite[\S 4]{Nivasch10}, Cibulka and \Kyncl~\cite[\S 2]{CibulkaK12}, or Sundar~\cite{Sundar92}.}

\begin{lemma}\label{lem:closed-form-K-dblK}
{\bf (Closed Form Bounds)} The ensemble $\{\K{s,i},\dblK{s,i}\}_{s\ge 3,i\ge 1}$ satisfies the following, where $t=\floor{\frac{s-2}{2}}$.
\begin{align*}
\K{3,i} &= 2i+2\\
\dblK{3,i} &= \Theta(i^2)\\
\K{4,i},\dblK{4,i} &= \Theta(2^i)\\
\K{5,i},\dblK{5,i} &\le 2^i(i+O(1))!\\
\K{s,i},\dblK{s,i} &\le 2^{{i+O(1)\choose t}}					& \mbox{for even $s> 4$}\\
\K{s,i},\dblK{s,i} &\le 2^{{i+O(1)\choose t}\log(2(i + 1)/e)}		& \mbox{for odd $s> 5$}
\end{align*}
\end{lemma}

\begin{proof}
First consider the case when $s\in\{3,4\}$.  Eqn.~(\ref{rec:K}) simplifies to
\begin{align}
\K{3,i} &= 2 + \K{3,i-1}\nonumber\\
\K{4,i} &= 2\K{3,i} + 2(\K{4,i-1}-2)\nonumber
\intertext{One proves by induction that $\K{3,i} = 2i+2$ and $\K{4,i} = 10\cdot 2^i - 4(i+2)$.
Using these identities, Eqn.~(\ref{rec:dblK}) can be simplified to}
\dblK{3,i} &=\dblK{3,i-1} +  2\cdot (2\cdot 6^{r-1}) + (1+2)(2i-2)\nonumber\\
\dblK{4,i} &\le \dblK{4,i-1} + 2\cdot \dblK{3,i} +  (2\cdot 6^{r-1}+2)(10\cdot 2^{i-1} - 4(i+1)).\nonumber
\intertext{A short proof by induction shows $\dblK{3,i} \le 6{i+1\choose 2} + 4\cdot 6^{r-1}(i+1)$
and that $\dblK{4,i} \le 20(6^{r-1}+2)2^i$.  In the general case we have, for $s\ge 5$,}
\K{s,i} &\le 2\K{s-1,i} + \K{s-2,i}\K{s,i-1}\nonumber\\
	&= 2\K{s-1,i} + \K{s-2,i}(2\K{s-1,i-1} + \K{s-2,i-1}(2\K{s-1,i-2} + \K{s-2,i-2}(\;\cdots\; + \K{s-2,2}\K{s,1}) \;\cdots\; ))\nonumber\\
	&= \sum_{l=0}^{i-2} 2\K{s-1,i-l}\cdot \prod_{k=0}^{l-1} \K{s-2,i-k} \;+\; \K{s,1}\cdot \prod_{k=0}^{i-2} \K{s-2,i-k}\label{eqn:K-expanded}
\intertext{When $s=5$ we have $\K{s-1,i} = \Theta(2^i)$ and $\K{s-2,i} = 2(i+1)$, so (\ref{eqn:K-expanded}) can be written}
	&= \sum_{l = 0}^{i-2} \Theta(2^{i-l})\cdot 2(i+1)2i \cdots 2(i+2-l) \;+\; \K{s,1}\cdot 2(i+1)2i2(i-1)\cdots 2(3)\nonumber\\
	&= \Theta(2^i\cdot (i+1)!) \;=\; 2^{(i+O(1))\log(2(i+1)/e)}\nonumber
\intertext{We prove that there are constants $\{C_s\}$ such that $\K{s,i} \le 2^{i+C_s\choose t}$ when $s$ is even
and $\K{s,i} \le 2^{{i+C_s\choose t}\log(2(i+1)/e)}$ when $s$ is odd.  The analysis above shows that $C_4$ and $C_5$ exist.
When  $s>4$ is even, (\ref{eqn:K-expanded}) is bounded by}
	&\le \sum_{l=0}^{i-2} 2^{{i-l+C_{s-1}\choose t-1}\log(2(i-l+1)/e)}\cdot \prod_{k=0}^{l-1} 2^{i-k+C_{s-2}\choose t-1} \;+\; \K{s,1}\cdot\prod_{k=0}^{i-2} 2^{i-k+C_{s-2}\choose t-1}\label{eqn:K-expanded-even}
\intertext{By Pascal's identity $\sum_{k=0}^{x} {i-k+C_{s-2}\choose t-1} = {i+1+C_{s-2}\choose t} - {i-x+C_{s-2}\choose t}$,
so (\ref{eqn:K-expanded-even}) is bounded by}
	&\le 2^{i+1+C_{s-2}\choose t} \cdot \paren{
		\sum_{l=0}^{i-2} 2^{{i-l+C_{s-1}\choose t-1}\log(2(i-l+1)/e) \:-\: {i-l+1+C_{s-2}\choose t}} \;+\; \K{s,1}		}\label{eqn:K-ind-even}\\
	&\le 2^{i+1+C_{s}\choose t}, \; \mbox{ for some sufficiently large $C_s$.}\nonumber
\intertext{The sum in (\ref{eqn:K-ind-even}) clearly converges as $i\rightarrow \infty$, 
though for some constant values of $i-l$ (depending on $C_{s-1}$ and $C_{s-2}$),  
${i-l + C_{s-1}\choose t-1}\log(2(i-l+1)/e)$ may be significantly larger than ${i-l+1+C_{s-2}\choose t}$.
When $s>5$ is odd the calculations are similar.  By the inductive hypothesis, (\ref{eqn:K-expanded}) is bounded by}
	&\le \sum_{l=0}^{i-2} 2^{{i-l+C_{s-1}\choose t}} \cdot \prod_{k=0}^{l-1} 2^{{i-k+C_{s-2}\choose t-1}\log(2(i-k+1)/e)} + \K{s,1}\cdot
					\prod_{k=0}^{i-2} 2^{{i-k+C_{s-2}\choose t-1}\log(2(i-k+1)/e)}\label{eqn:K-expanded-odd}\\
	&\le 2^{{i+1+C_{s-2}\choose t}\log(2(i+1)/e)}\cdot \paren{
		\sum_{l=0}^{i-2} 2^{{i-l+C_{s-1}\choose t} - {i-l+1+C_{s-2}\choose t}\log(2(i+1)/e)} + \K{s,1}		}\nonumber\\
	&\le 2^{{i+1+C_s\choose t}\log(2(i+1)/e)}, \; \mbox{ for some sufficiently large $C_s$.}\nonumber
\intertext{Turning to $\dblK{s,i}$, we have}
\dblK{s,i} &= \dblK{s,i-1} \;+\; 2\dblK{s-1,i} +  (\dblK{s-2,i} + 2)\K{s,i-1}\nonumber\\
		&= \dblK{s,1}  \;+\; \sum_{l=0}^{i-2} \left[2\dblK{s-1,i-l} +  (\dblK{s-2,i-l} + 2)\K{s,i-1-l}\right]\label{eqn:dblK-expanded}
\intertext{It is straightforward to show that when $s\ge 4$, the bounds on $\K{s,i}$ also hold for $\dblK{s,i}$ 
with respect to different constants $\{D_s\}$.  When $s=5$, Eqn.~(\ref{eqn:dblK-expanded}) becomes}
\dblK{5,i} &= \dblK{5,1} \;+\; \sum_{l=0}^{i-2} \paren{2\cdot \Theta(2^{i-l}) \,+\, (\Theta(i-l)^2) + 2)\cdot \Theta(2^{i-1-l}(i-l)!)}\nonumber\\
		&= \Theta(2^i(i+2)!) \;\le\; 2^{(i+D_5)\log(2(i+1)/e)}, \,\mbox{ for a sufficiently large $D_5$.}\nonumber
\intertext{When $s>4$ is even,
Eqn.~(\ref{eqn:dblK-expanded}) implies, by the inductive hypothesis, that}
\dblK{s,i} &\le \dblK{s,1} \;+\; \sum_{l=0}^{i-2} \left[2^{{i-l+D_{s-1}\choose t-1}\log(2(i-l+1)/e)+1} \,+\, (2^{i-l+D_{s-2}\choose t-1}+2)2^{i-1-l+C_s\choose t}\right]\nonumber\\
		&\le 2^{i+l+D_s\choose t}, \, \mbox{ for a sufficiently large $D_s$.}\nonumber
\intertext{When $s>5$ is odd,}
\dblK{s,i} &\le \dblK{s,1} \;+\; \sum_{l=0}^{i-2} \left[2^{{i-l+D_{s-1}\choose t}+1} \,+\, (2^{{i-l+D_{s-2}\choose t-1}\log(2(i-l+1)/e)}+2)2^{{i-1-l+C_s\choose t}\log(2(i-l)/e)}\right]\nonumber\\
		&\le 2^{{i+D_s\choose t}\log(2(i+1)/e)}, \, \mbox{ for a sufficiently large $D_s$.}\nonumber
\end{align}

\end{proof}

Given that Lemma~\ref{lem:closed-form-K-dblK} holds for all $i$, one chooses $i$ to be minimum such
that the `$m$' term does not dominate, that is, the minimum $i$ for which $j\le 3$ or $(cj)^{s-2} \le n/m$.
It is straightforward to show that $i = \alpha(n,m)+O(1)$ is optimal, which immediately
gives bounds on $\PERM{r,s}(n,m)$ and $\dblPERM{r,s}(n,m)$ analogous to those
claimed for $\PERM{r,s}(n)$ and $\dblPERM{r,s}(n)$ in Theorem~\ref{thm:PERM}, excluding the case $s=3$, which is dealt with in Section~\ref{sect:dblPERMfour}.
In order to obtain bounds on $\PERM{r,s}(n)$ and $\dblPERM{r,s}(n)$ we invoke Lemma~\ref{lem:SparseVersusBlocked}.
For example, it states that $\PERM{r,s}(n) = \gamma_{r,s-2}(\gamma_{r,s}(n))\cdot \PERM{r,s}(n,3n)) + 2n$, where $\gamma_{r,s}(n)$ is a non-decreasing
upper bound on $\PERM{r,s}(n) / n$.  The $\gamma_{r,s-2}(\gamma_{r,s}(n))$ factor may not be constant, but it does not affect the
error tolerance already in the bounds of Theorem~\ref{thm:PERM}.\footnote{For example, when $s=6$, 
$\gamma_{r,s-2}(\gamma_{r,s}(n)) = O\paren{2^{\alpha\paren{2^{\alpha^2(n)/2 \,+\, O(\alpha(n))}}}} = O(2^{\alpha(\alpha(n))})$ is non-constant.
Nonetheless
$O(2^{\alpha(\alpha(n))})\cdot \PERM{r,s}(n,3n) = O(2^{\alpha(\alpha(n))}) \cdot n\cdot 2^{\alpha^2(n)/2 \,+\, O(\alpha(n))} = n\cdot 2^{\alpha^2(n)/2 \,+\, O(\alpha(n))}$.}

\begin{remark}\label{rem:ackermann-invariant-bounds}
Our lower and upper bounds on $\PERM{r,s}(n)$ are tight (when $r\ge 3$)
inasmuch as they are both of the form $n\cdot 2^{\alpha^t(n)/t! \,+\, O(\alpha^{t-1}(n))}$ when $s\ge 4$ is even
and $n\cdot 2^{\alpha^t(n)(\log\alpha(n)\,+\,O(1))/t!}$ when $s\ge 5$ is odd.
However, it is only when $s$ is even that these bounds are sharp in the Ackermann-invariant sense of~\cite[Remark 1.1]{Pettie-SoCG13}, 
that is, invariant under $\pm O(1)$ perturbations in the definition of $\alpha(n)$.
For example, our lower and upper bounds on $\PERM{r,5}(n)$ are $n\cdot(\alpha(n)+O(1))!$ and $n\cdot 2^{\alpha(n)}(\alpha(n)+O(1))!$.
The $2^{\alpha(n)}$ factor gap could probably be closed by substituting Nivasch's construction of order-3 DS sequences~\cite[\S 6]{Nivasch10}
for $U_3(i,j)$ in Section~\ref{sect:lbPERM}, which would lead to sharp,
Ackermann-invariant bounds of $\PERM{r,5}(n) = n\cdot 2^{\alpha(n)}(\alpha+O(1))!$.
With a more careful analysis of the recurrence for $\K{s,i}$ it should be possible to obtain sharp, Ackermann-invariant bounds
on $\PERM{r,s}(n)$ for all odd $s$.
\end{remark}

\section{Derivation Trees}\label{sect:derivation-tree}

Derivation trees were introduced in~\cite{Pettie-SoCG13} to model hierarchical decompositions of sequences.  They are instrumental in our analysis of 
$\dblPerm{r,4}$-free sequences, in Section~\ref{sect:dblPERMfour}, and of double DS sequences, in Section~\ref{sect:dblDS}.
Throughout this section we use the sequence decomposition notation defined in Section~\ref{sect:sequence-decomposition}.

A recursive decomposition of a sequence $S$ can be represented as a 
rooted {\em derivation tree} $\Tree = \Tree(S)$.  Nodes of $\Tree$ are identified with blocks.
The leaves of $\Tree$ correspond to the blocks of $S$ whereas internal nodes correspond
to blocks of derived sequences.  Let $\block(v)$ be the block of $v\in\Tree$, which
may be treated as a {\em set} of symbols if we are indifferent to their permutation in $\block(v)$.

\paragraph{Base Case.}
Suppose $S=B_1B_2$ is a two block sequence, where each block
contains the whole alphabet $\Sigma(S)$.
The tree $\Tree(S)$ consists of three nodes $u,u_1,$ and $u_2$, where $u$ is the parent of $u_1$ and $u_2$,
$\block(u_1)=B_1$, $\block(u_2)=B_2$, and $\block(u)$ does not exist.  
For every $a\in\Sigma(S)$ call $u$
its {\em crown} and $u_1$ and $u_2$ its left and right {\em heads}, respectively.
These nodes are denoted $\crown_{|a}, \lefthead_{|a},$ and $\righthead_{|a}$.

\paragraph{Inductive Case.}
If $S$ contains $m>2$ blocks, choose a uniform block partition $\{m_q\}_{1\le q\le \Gm}$, that is, one where
$m_1,\ldots,m_{\Gm-1}$ are equal powers of two and $m_{\Gm}$ may be smaller.
This block partition induces local sequences $\{\LS_q\}_{1\le q\le \Gm}$ and an $\Gm$-block 
contracted global sequence $\GS'$.
Inductively construct derivation trees $\GTree = \Tree(\GS')$ and $\{\LTree_q\}_{1\le q\le \Gm}$, where 
$\LTree_q = \Tree(\LS_q)$.  
To obtain $\Tree(S)$, 
identify the root of $\LTree_q$ (which has no block) with the $q$th leaf of $\GTree$, 
then
place the blocks of $S$ at the leaves of $\Tree$.
This last step is necessary since only local symbols appear in the blocks of $\{\LTree_q\}$ whereas
the leaves of $\Tree$ must be identified with the blocks of $S$.
The crown and heads of each symbol $a\in \Sigma(S)$ are inherited from $\GTree$, if $a$ is global,
or some $\LTree_q$ if $a$ is local to $S_q$. See Figure~\ref{fig:derivation-tree} for a schematic.

\begin{figure}
\centerline{\scalebox{.37}{\includegraphics{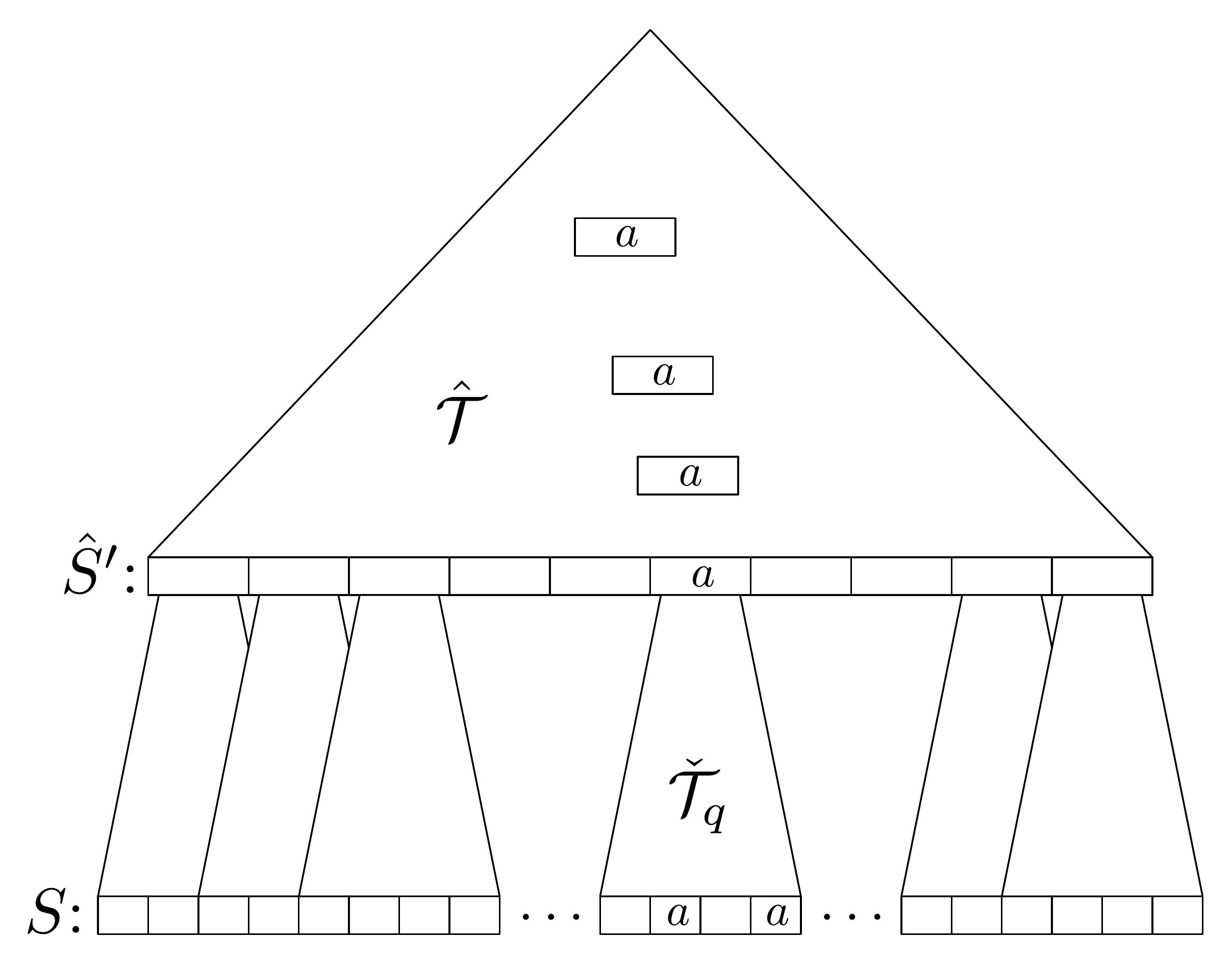}}}
\caption{\label{fig:derivation-tree}The derivation tree $\Tree(S)$ is
the composition of $\GTree = \Tree(\GS')$ and $\{\LTree_q\}_{1\le q\le \Gm}$, 
where $\LTree_q = \Tree(\LS_q)$.  A global symbol $a\in\Sigma(\GS)$ appears
in blocks at the leaf level of $\Tree$, at the leaf level of $\GTree$, and possibly at higher levels of $\GTree$.}
\end{figure}

\subsection{Special Derivation Trees}\label{sect:derivation-tree-via-Ackermann}

It is useful to constrain $\Tree$ to use a uniform block partition.
Every derivation tree generated in this fashion can be embedded in a full rooted binary tree with height $\ceil{\log m}$,
though the composition of blocks depends on how block partitions are chosen.
We will generate two varieties of derivation trees.  At one extreme is the {\em canonical} derivation tree, where block partitions
are chosen in the least aggressive way possible. At the other extreme is one where block partitions are guided by Ackermann's function.

\paragraph{Canonical Derivation Trees.}
The canonical derivation tree $\CTree(S)$ of a sequence $S$ 
is obtained by choosing the uniform block partition with $\Gm = \ceil{m/2}$.
We form $\CTree(S)$ by constructing $\CTree(\GS')$ recursively and composing it 
with the trivial three-node base case trees $\{\Tree(\LS_q)\}_{q}$.

\paragraph{Derivation Trees via Ackermann's Function.}
Given a parameter $i\ge 1$, define $j\ge 1$ to be minimal such that $m \le a_{i,j}$.
If $j=1$ then $m=a_{i,1} =2$, meaning $\Tree(S)$ must be the three-node base case tree.
When $j>1$ we choose a uniform block partition with width $w= a_{i,j-1}$ (which is a power of 2),
so $\Gm = \ceil{m/w} \le a_{i,j}/a_{i,j-1} = a_{i-1,w}$.  
The global tree $\GTree$ is constructed recursively with parameter\footnote{Note that when $i=1$ it does not matter
that $i-1=0$ is an invalid parameter.  In this case $w = a_{1,j-1} = a_{1,j}/2$ and $\Gm = 2$,
so $\GTree$ is forced to be a three-node base case tree.} 
$i-1$ and each local tree $\LTree_q$ is constructed recursively with parameter $i$.

\subsection{Projections of the Derivation Tree}

The {\em projection of $\Tree$ onto $a\in\Sigma(S)$},
written $\Tree_{|a}$, is the tree rooted at $\crown_{|a}$ on the node set $\{\crown_{|a}\} \cup \{v \in \Tree \:|\: a\in\block(v)\}$.
The edges of $\Tree_{|a}$ represent paths in $\Tree$ passing through blocks that do not contain $a$.

\begin{definition} {\bf (Anatomy of a projection tree)}
\begin{itemize}
\item The leftmost and rightmost leaves of $\Tree_{|a}$ are {\em wingtips},
denoted $\ltip_{|a}$ and $\rtip_{|a}$.
\item The left and right {\em wings}
are those paths in $\Tree_{|a}$ extending from $\lefthead_{|a}$ to $\ltip_{|a}$ and from 
$\righthead_{|a}$ to $\rtip_{|a}$.
\item Descendants of $\lefthead_{|a}$ and $\righthead_{|a}$ in $\Tree_{|a}$ are called {\em doves} and {\em hawks}, respectively.
\item A child of a wing node that is not itself on the wing is called a {\em quill}.
\item A leaf is called a {\em feather} if it is the rightmost descendant of a dove quill or leftmost descendant of a hawk quill.
\item Suppose $v$ is a node in $\Tree_{|a}$.  
Let $\wing_{|a}(v)$ be the nearest wing node ancestor of $v$,
$\quill_{|a}(v)$ the quill ancestral to $v$,
and $\feather_{|a}(v)$ the feather descending from $\quill_{|a}(v)$.
See Figure~\ref{fig:feathers} for an illustration.
\end{itemize}
If $\Tree(S)$ is specified, the terms {\em feather} and {\em wingtip} can also be applied to individual occurrences in $S$.  For example,
an occurrence of $a$ in block $\block(v)$ of $S$ is a feather if $v$ is a feather in $\Tree_{|a}$.
\end{definition}

When $\Tree(S)$ is constructed according to Ackermann's function, a short proof by induction shows that 
the height of each projection tree $\Tree_{|a}$ (distance from $\crown_{|a}$ to a leaf) is at most $i+1$.

\begin{figure}
\centering\scalebox{.35}{\includegraphics{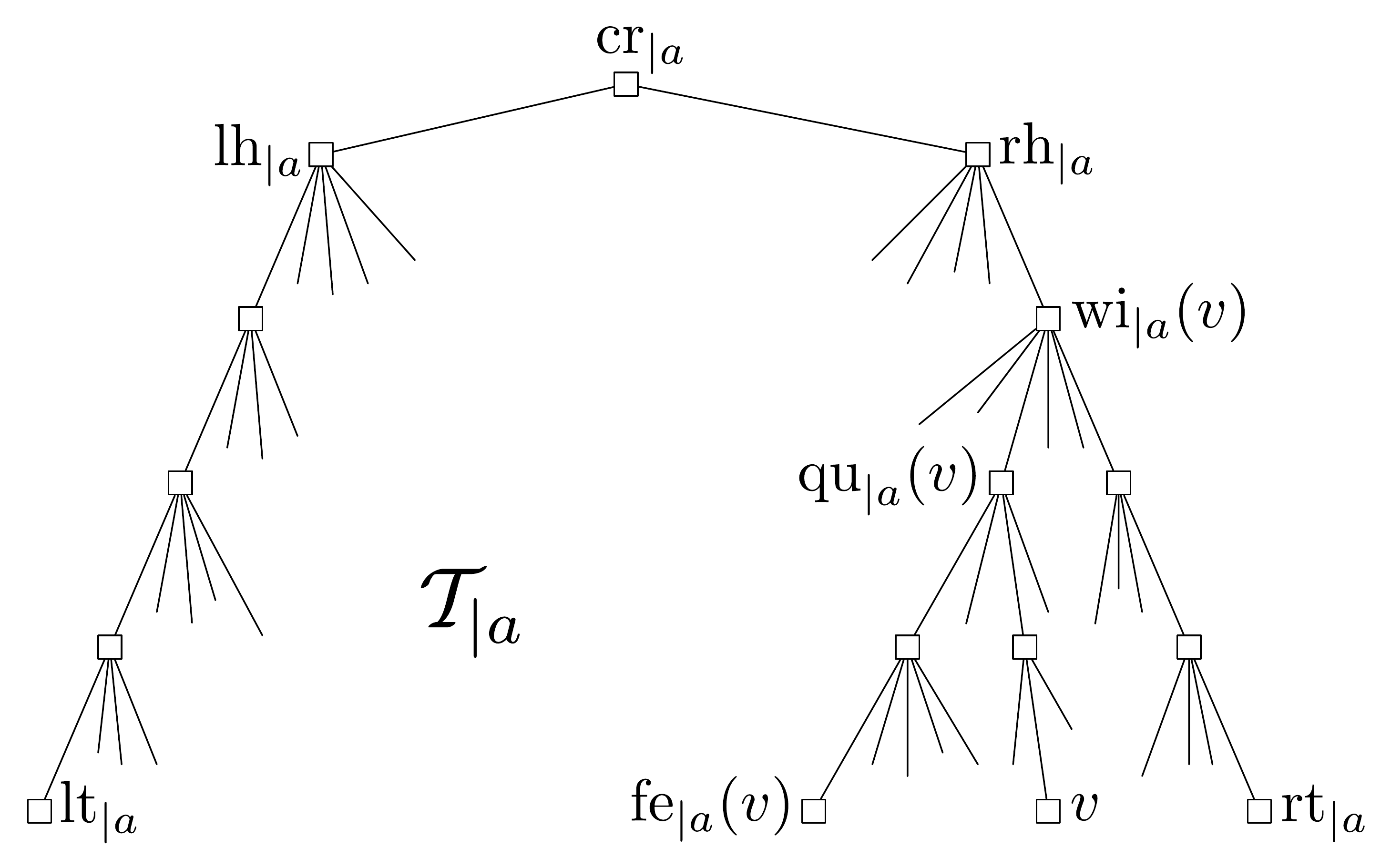}}
\caption{\label{fig:feathers} In this example $v$ is a hawk leaf in 
$\Tree_{|a}$ since it is a descendant of $\righthead_{|a}$.
Its wing node $\wing_{|a}(v)$, quill $\quill_{|a}(v)$, and feather $\feather_{|a}(v)$ are indicated.}
\end{figure}

\section{Upper Bounds on $\dblPerm{r,4}$-free Sequences}\label{sect:dblPERMfour}

Since order-$3$ DS sequences are necessarily $\Perm{2,4}$-free, we have
$\dblPERM{r,3}(n) \ge \PERM{r,3}(n) \ge \DS{3}(n) = \Theta(n\alpha(n))$.  In this Section
we prove tight upper bounds of $\dblPERM{r,3}(n) = O(n\alpha(n))$.  These bounds imply
$\dblDS{3}(n)$ is also $O(n\alpha(n))$, resolving one of Klazar's open problems~\cite{Klazar02}.

Our analysis is different in character from all previous analyses of (generalized) Davenport-Schinzel sequences.
There are two new techniques used in the proof which are worth highlighting.  
Previous analyses partition the symbols in a block based on some attributes (first, middle, last, etc.),
but do not assign any attributes to the blocks themselves.  In our analysis we must treat blocks differently
based on their context within the larger sequence, that is, according to properties that are independent of the contents of the block.
(See the definition of {\em roosts} in Section~\ref{sect:roosts}.)
The second ingredient is an accounting scheme for bounding the proliferation of symbols.  Rather than count the {\em number}
of occurrences of a symbol, say $b$, we assign each occurrence of $b$ a {\em potential} based on its context.  
If one $b$ in $\GS'$ begets
multiple $b$s in $\GS$, the number of $b$s increases, but the {\em aggregate} potential of the $b$s in $S$ may, in fact, 
be at most the potential of the originating $b$ in $\GS'$. That is, sometimes proliferating symbols ``pay for themselves.'' We only 
need to track changes in sequence potential, not sequence length.
Amortizing the analysis in this way lets us account
for the proliferation of symbols across {\em many} levels of the derivation tree, not just between $\GS'$ and $S$.

\subsection{A Potential-Based Recurrence}

Fix a $\dblPerm{r,4}$-free sequence $Z$ and $i^\star\ge 1$.  Define $j^\star$ to be minimal
such that its block count $\bl{Z} \le a_{i^\star,j^\star}$ and let $\Tree = \Tree(Z)$ be constructed as in Section~\ref{sect:derivation-tree-via-Ackermann}
with parameter $i^\star$.  In this section we analyze a sequence $S$ encountered in the recursive decomposition
of $Z$, that is, $S$ is either $Z$ itself or a sequence encountered when recursively decomposing $\GZ'$ and $\{\LZ_q\}$.
Since $S\subseq Z$, it too must be $\dblPerm{r,4}$-free but we can often say something stronger.
If each occurrence of a symbol in $S$ represents at least two occurrences
in $Z$ then $S$ must be $\Perm{r,4}$-free.\footnote{This is not quite true, 
but we can make this inference when bounding $\dblPERM{r,3}$ asymptotically.  See Remark~\ref{rem:interleave} for a discussion of this issue.}
Call an occurrence in $S$ {\em terminal} if it represents exactly one occurrence in $Z$ and {\em non-terminal} otherwise.
In terms of the derivation tree, an occurrence of $a$ in $S$ is terminal iff it has exactly one leaf descendant in $\Tree_{|a}$.

Each occurrence of a symbol in $S$
carries a nonnegative integer potential based on its context within $S$ and even within $\Tree(Z)$.  Since the length of $S$
is no more than its aggregate potential, it suffices to upper bound the potential.  Define $\dblPERMfour(n,m)$ to be the 
maximum potential of an $m$-block sequence over an $n$-letter alphabet encountered in decomposing $Z$.
The way potentials are assigned will be discussed shortly.  For the time being it suffices to know that the maximum potential is $\Pot = O(1)$,
all terminals carry unit potential, and all non-terminals carry potential at least three.

Our goal is to prove that $\dblPERMfour$ obeys the following recurrence.

\begin{recurrence}\label{rec:dblPERMfour}
\begin{align*}
\dblPERMfour(n,m) 
	&= \sum_{1\le q\le \Gm} \dblPERMfour(\Ln_q,m_q) 
			\,+\, 2\cdot \SqBrack{\Pot\cdot \PERM{r,2}(\Gn,m) \,+\, \dblPERM{r,2}(\Gn,m)+\Gn}
			\,+\, \dblPERMfour(\Gn,\Gm) \\
			&\hcm[1.5] +\, (r-1)\Pot\cdot m 
			\,+\, 2[(r-1)(i^\star-2)]^2\cdot \Gm
\end{align*}
\end{recurrence}

Decomposing $S$ as usual, it follows that the maximum potential of local sequences $\{\LS_q\}_q$ is $\sum_q \dblPERMfour(\Ln_q,m_q)$,
giving the first term of Recurrence~\ref{rec:dblPERMfour}.
The sequence $\GfS$ of global first occurrences can be partitioned into terminals $\GfSterminal$ and non-terminals $\GfSnonterminal$.
After removing the last occurrence of each symbol in $\GfSterminal$, the resulting sequence is $\dblPerm{r,3}$-free, so its length (and potential)
is $|\GfSterminal| \le \dblPERM{r,2}(\Gn,m) + \Gn$.  We endow each non-terminal in $\GfSnonterminal$ an initial potential at most $\Pot$.
(Note that occurrences of $a$ in $\GfS$ correspond to quills in $\Tree_{|a}$.)
Being $\Perm{r,3}$-free, the potential of $\GfSnonterminal$ is therefore at most $\Pot\cdot \PERM{r,2}(\Gn,m)$.  A symmetric analysis
is applied to $\GlS$, the sequences of last occurrences, which gives the second term of Recurrence~\ref{rec:dblPERMfour}.

The global contracted sequence $\GS'$ begets $\GfS,\GlS,$ and $\GmS$, the first two of which we have just accounted for.
In general $|\GmS|$ may be significantly larger than $|\GS'|$.  We account for this proliferation in symbols by showing that the aggregate
potential of $\GmS$ is nonetheless at most that of $\GS'$ plus $(r-1)\Pot\cdot m + 2[(r-1)(i^\star-1)]^2\cdot \Gm$, 
which explains the last three terms of Recurrence~\ref{rec:dblPERMfour}.
Consider the sequence $\GmS_q$ begat by the middle symbols of block $B_q$ in $\GS'$.  We decompose $\GmS_q$ as follows.
\begin{enumerate}
\item Tag any symbol occurring exactly once in $\GmS_q$.  (Its potential in $\GmS_q$ will be at most its potential in $\GS'$.)\label{item:GmS1}
\item Tag the first non-terminal occurrence of each symbol in $\GmS_q$.\label{item:GmS2}
\item Tag the first, second, and last terminal occurrence of each symbol in $\GmS_q$.\label{item:GmS3}
\item Tag the first $r-1$ untagged occurrences (terminal and non-terminal) in each block of $\GmS_q$.\label{item:GmS4}
\end{enumerate}
Symbols that are tagged in both of Steps~\ref{item:GmS2} and \ref{item:GmS3} have {\em molted}; all others are {\em unmolted}.
We will say that the non-terminal $a$ tagged in Step~\ref{item:GmS2} has {\em molted} those terminal $a$s tagged in Step~\ref{item:GmS3}.
See Figure~\ref{fig:molting} for a schematic.%
\begin{figure}
\centering\scalebox{.4}{\includegraphics{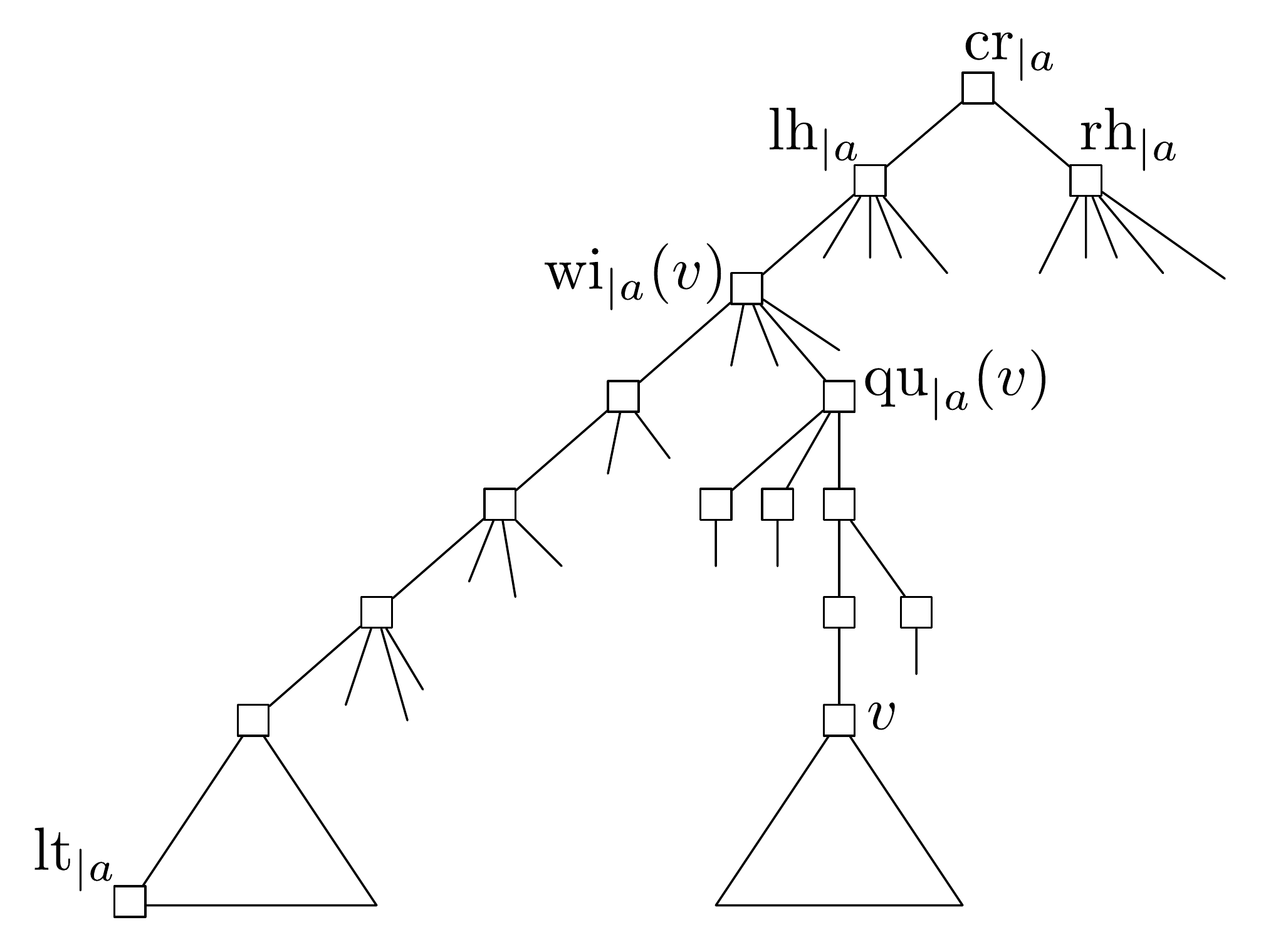}}
\caption{\label{fig:molting}Here $v$ is an internal node of $\Tree_{|a}$.  
Between $\quill_{|a}(v)$ and $v$, $a$ has molted twice: at $v$'s parent it molted one $a$ to the right and at $v$'s grandparent it molted
two $a$s to the left.}
\end{figure}

We claim $\GmS_q$ has been completely tagged after Step~\ref{item:GmS4}.  If this were not so, there must be
$r$ symbols $a_1,\ldots,a_r$ in some block $B$ in $\GmS_q$.  If $a_k$ is terminal in $B$ it must be preceded by two terminal $a_k$s
and followed by one terminal $a_k$ in $\GmS_q$; if $a_k$ is non-terminal in $B$ it must be preceded by a non-terminal $a_k$.
Dividing $\GmS_q$ at the left boundary of $B$, we see two occurrences of each of $a_1,\ldots,a_r$ on both the left and right side 
of the boundary, which may take the form of one non-terminal or two terminals.  Since $a_1,\ldots,a_k$ are categorized as global middle
in $S_q$, each appears both before and after $S_q$, yielding an instance of $\dblPerm{r,4}$ in $Z$, a contradiction.  

The aggregate potential of those symbols tagged in Step~\ref{item:GmS4} is at most $(r-1)\Pot\cdot m$, which are covered by the second-to-last term of Recurrence~\ref{rec:dblPERMfour}.  Suppose that $a\in B_q$ is non-terminal in $\GS'$ but it begets only terminal $a$s in $\GmS_q$, that is,
no $a$s are tagged in Step~\ref{item:GmS2}.  This proliferation of $a$s causes no net increase in potential since the $a\in B_q$ carries
potential at least 3, which covers the potential of the three terminal $a$s tagged in Step~\ref{item:GmS3}.  In general, for each molted symbol $a$,
we will tag one non-terminal and up to three terminals in Steps~\ref{item:GmS2} and \ref{item:GmS3}.  This will cause no net increase in potential
{\em provided} that the $a$ in $B_q$ carries at least the potential of the non-terminal $a$ in $\GmS_q$ plus 3.
In order to avoid cumbersome statements, we will treat the non-terminal $a$ tagged in Step~\ref{item:GmS2} as the ``same'' $a\in B_q$.
For example, if $B$ is a block in $\GmS_q$ and $a\in B$ is non-terminal, to say {\em the $a\in B$ has molted four times} means that, 
in $\Tree_{|a}$, $B$ has four ancestors, possibly including itself, and all strict descendants of $\quill_{|a}(B)$, which each have at least one sibling in $\Tree_{|a}$.  This sibling corresponds to an $a$ removed in Step~\ref{item:GmS3} at some stage in the decomposition of $S$.

In the remainder of this section we explain why it suffices to endow each new non-terminal quill with a {\em constant} potential $\Pot$.
The analysis above shows that $3\cdot(i^\star-1)$ suffices, which is not constant.\footnote{Observe that for any $a\in\Sigma(Z)$, 
the height of $\Tree_{|a}$ is $i^\star+1$ and all quills of $\Tree_{|a}$ are at distance at least 2 from $\crown_{|a}$.  Every non-terminal quill
can therefore molt up to $i^\star-1$ times, generating up to three terminals per molting, each of which carries unit potential.}

\subsection{Roosts, Eggs, and Fertility}\label{sect:roosts}

Our analysis considers properties of blocks (and of occurrences of symbols) that depend on their context within a larger sequence.

\begin{definition}\label{def:roost} {\bf (Roosts and Eggs)}
Let $S$ be a sequence encountered in the decomposition of $Z$.
\begin{enumerate}
\item
An interval $I$ of zero or more blocks in $S$ is a {\em $k$-roost} if there are $k$ distinct symbols $a_1,\ldots,a_k$ such that the sequence
contains
\[
a_1 \, a_2 \cdots a_k \;\; a_k^2 a_{k-1}^2\cdots a_1^2 \;\; I\;\; a_1^2 a_2^2 \cdots a_k^2 \;\; a_k a_{k-1} \cdots a_1,
\]
where $b^2$ refers to two terminal $b$s or one non-terminal $b$. 
The occurrences of $a_1$ just to the left and right of $I$ are called {\em $k$-left mature} and {\em $k$-right mature}.
A $k$-mature occurrence of a symbol whose block is a $k$-roost is {\em infertile}.  
A $k$-left mature occurrence that is not infertile is $k$-left fertile; $k$-right fertile is defined analogously.
(For any $l<k$, $k$-roosts are clearly also $l$-roosts, and $k$-mature occurrences also $l$-mature.)

\item An occurrence of $a_1$ in block $B$ of $S$ is a {\em $k$-egg} if the sequence contains
\[
a_1 \, a_2 \cdots a_k \;\; a_k^2 a_{k-1}^2\cdots a_2^2 \;\; B\;\; a_2^2 a_3^2 \cdots a_k^2 \;\; a_k a_{k-1} \cdots a_1
\]
Note that any middle occurrence of a symbol is a $1$-egg.
\end{enumerate}
\end{definition}

One may already discern from Definition~\ref{def:roost} the shape of the rest of the proof.
A $k$-roost can only exist if the sequence contains a $\dblPerm{k,4}$ sequence, so there cannot be $r$-roosts.
If the proliferation of symbols {\em necessarily} leads to $k$-roosts for ever larger $k$, we have a cap on the proliferation of symbols.
Lemma~\ref{lem:roost-props} lists some straightforward consequences of Defintion~\ref{def:roost}.

\begin{lemma}\label{lem:roost-props} {\bf (Properties of Roosts and Eggs)}
Let $S$ be an $m$-block sequence encountered in the recursive decomposition of a $\dblPerm{r,4}$-free sequence $Z$.
Define $\{S_q,\LS_q,\GS_q\}_{1\le q\le \Gm}$ and $\GS'=B_1\cdots B_{\Gm}$ as usual.

\begin{enumerate}
\item No block in $S$ is an $r$-roost.  All $r$-eggs represent at most 3 occurrences in $Z$.\label{item:roost-props1}
\item If $B_q$ is a $k$-roost in $\GS'$, every block of $S_q$ is a $k$-roost in $S$.\label{item:roost-props2}
\item Let $B$ be a block in $S_q$ containing a global symbol $a$.\label{item:roost-props3}
If $B$ is a $(k-1)$-roost in $\LS_q$ and the $a\in B_q$ is a middle occurrence in $\GS'$
then $a\in B$ is a $k$-egg in $S$.  See Figure~\ref{fig:new-k-egg}.
\item Let $B$ be a block in $S_q$ containing a global symbol $a$.
Suppose the $a\in B_q$ is $k$-left fertile in $\GS'$ and the $a\in B$ is $k$-left fertile in $S$.
All blocks following $B$ in $S_q$ are $k$-roosts in $S$.  A symmetric statement is true of $k$-right fertile 
occurrences. See Figure~\ref{fig:fertile-infertile}.
\label{item:roost-props4}
\end{enumerate}
\end{lemma}

\begin{figure}
\centering\scalebox{.37}{\includegraphics{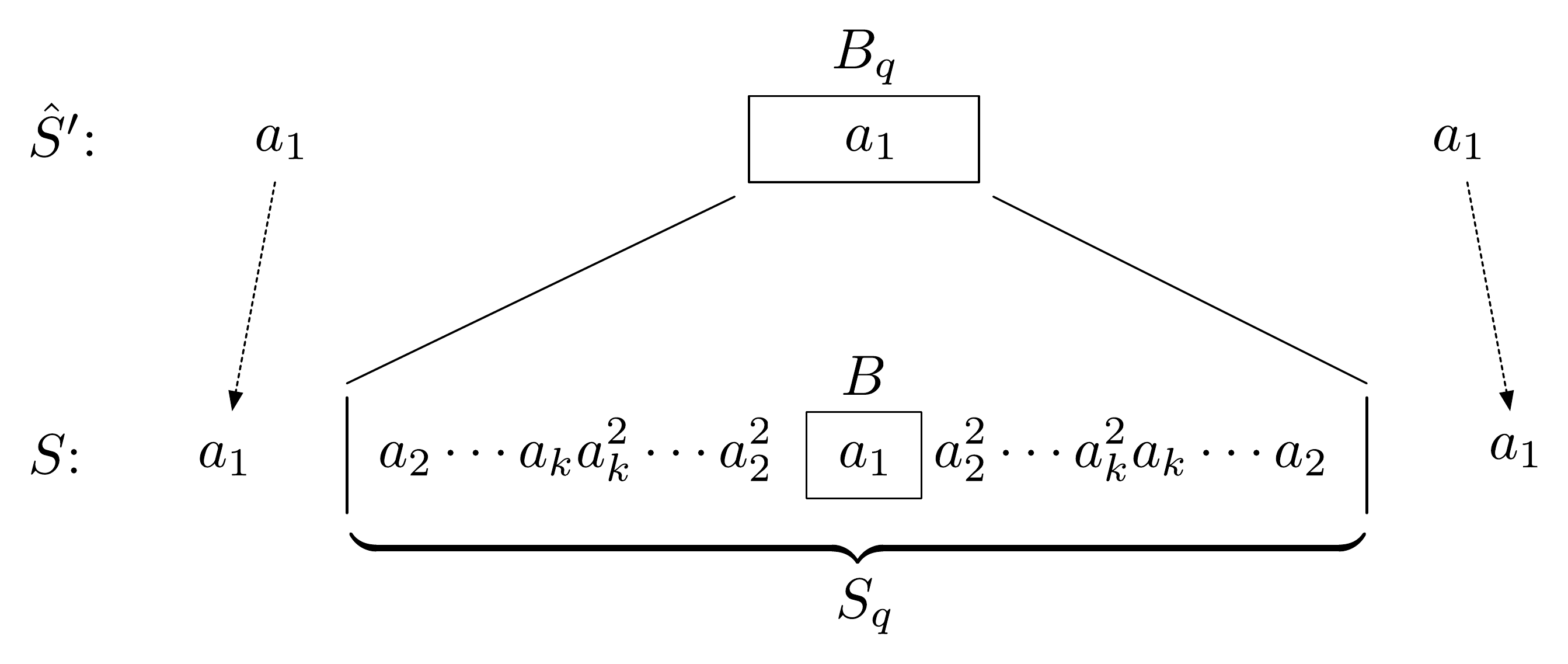}}
\caption{\label{fig:new-k-egg}A $k$-egg is formed when a middle $a_1\in B_q$ is dropped into a $(k-1)$-roost in $\LS_q$.}
\end{figure}

\begin{figure}
\centering\scalebox{.28}{\includegraphics{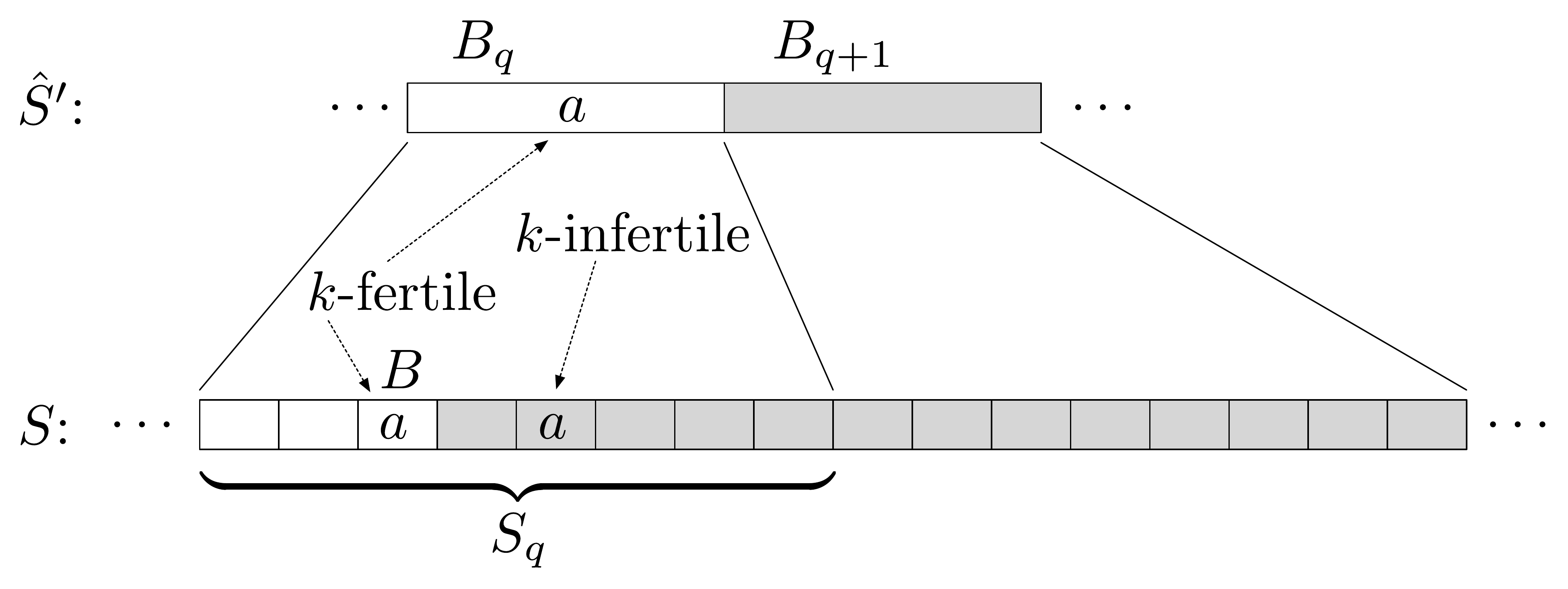}}
\caption{\label{fig:fertile-infertile}The shaded blocks are $k$-roosts.  A $k$-left fertile occurrence of $a\in B_q$ in $\GS'$ begets
at most one $k$-left fertile occurrence in $S_q$, and, in this example, one $k$-infertile occurrence.  Since $B_{q+1}$ is a $k$-roost in $\GS'$,
all blocks in $S_{q+1}$ are $k$-roosts in $S$ whether or not they were already $k$-roosts in $\LS_{q+1}$.}
\end{figure}

\subsection{Molting and the Evolution of Potentials}\label{sect:dblPERMfour-potential}

Consider the status of a non-terminal symbol $a$ as it descends, in $\Tree_{|a}$, from $\quill_{|a}(v)$ to some leaf $v$.
Since $a\in\block(\quill_{|a}(v))$ is a middle symbol at that level (it is not on either wing of $\Tree_{|a}$), this $a$ begins
as a 1-egg and may become 1-fertile (left or right), then 1-infertile, then a 2-egg, 2-fertile, 2-infertile, and so on.  
It cannot become $r$-mature (fertile or infertile) for this would mean that $\dblPerm{r,4}\subseq Z$, so there are at most $3(r-1)$ transitions.
Multiple transitions may occur simultaneously.
When a non-terminal first becomes a $k$-egg, or $k$-fertile, or $k$-infertile, its potential becomes $\PotEgg_k,\PotFertile_k,$ or $\PotInfertile_k$,
where
\[
\Pot = \PotEgg_1 > \PotFertile_1 > \PotInfertile_1 > \cdots > \PotEgg_{r-1} > \PotFertile_{r-1} > \PotInfertile_{r-1} > \PotEgg_r = 3
\]
If we can show that each symbol molts $O(1)$ times between status transitions, it suffices to set the initial potential at 
$\Pot = O(r) = O(1)$.  This is clearly true of $k$-egg $\rightarrow$ $k$-mature transitions.  Any $k$-egg $a$ that molts three $a$s
must have molted two of them to the same side, left or right, making it $k$-mature.  Since a non-terminal can molt up to 3 terminals
in the molting event that makes it $k$-mature, it suffices to set $\PotEgg_k - \PotFertile_{k} = 5$.  
(If this $a$ transitions directly from a $k$-egg to $k$-infertile, all the better, for $\PotInfertile_k < \PotFertile_k$.)
We now analyze the $k$-fertile $\rightarrow$ $k$-infertile and $k$-infertile $\rightarrow$ $(k+1)$-egg transitions.

\begin{lemma}\label{lem:left-fertile}
Fix a block index $q\le \bl{\GS'}$ and let $F \subset B_q$ be those symbols newly $k$-left fertile, that is, 
they were not $k$-left fertile at any ancestor of $B_q$ in their respective derivation trees.
The total number of terminals molted by $F$-symbols before they become $k$-infertile is at most $2|F| + (r-1){i^\star-1 \choose 2}$.
\end{lemma}

\begin{proof}
Lemma~\ref{lem:roost-props}(\ref{item:roost-props4}) implies that so long as symbols in $F$ remain
$k$-fertile, as they travel from $B_q$ to a block in $S_q$, to blocks at lower levels of the derivation tree, 
they will always be contained in a {\em single} block at that level of the tree.
In other words, there is a sequence of nodes $(B_q=v_1,v_2,\ldots,v_l)$ in $\Tree$ 
lying on a path from $B_q=v_1$ (in $\GS'$), to $v_2$ (in $S$), to a descendant leaf $v_l$ (where $l\le i^\star$)
such that any symbol $a\in F$ is $k$-left fertile in some {\em prefix} of the list $\block(v_1),\block(v_2),\ldots,\block(v_l)$. 
See Figure~\ref{fig:left-fertile}.
Call a symbol $a\in F$ {\em type $(f,g)$} if $a$ molted a terminal to the right at both $\block(v_f)$ and $\block(v_g)$, for $1< f < g \le l$.\footnote{Note that a symbol that molts exactly twice to the right has one type.  In general, a symbol that molts $h$ times to the right is of ${h \choose 2}$ distinct types.}
That is, in $\Tree_{|a}$, $\block(v_f)$ and $\block(v_g)$ have right siblings.  Note that during the time in which this $a$ is $k$-left fertile it can molt
at most once to the left: molting two $a$s to the left would make it $k$-infertile.

\begin{figure}
\centering\scalebox{.28}{\includegraphics{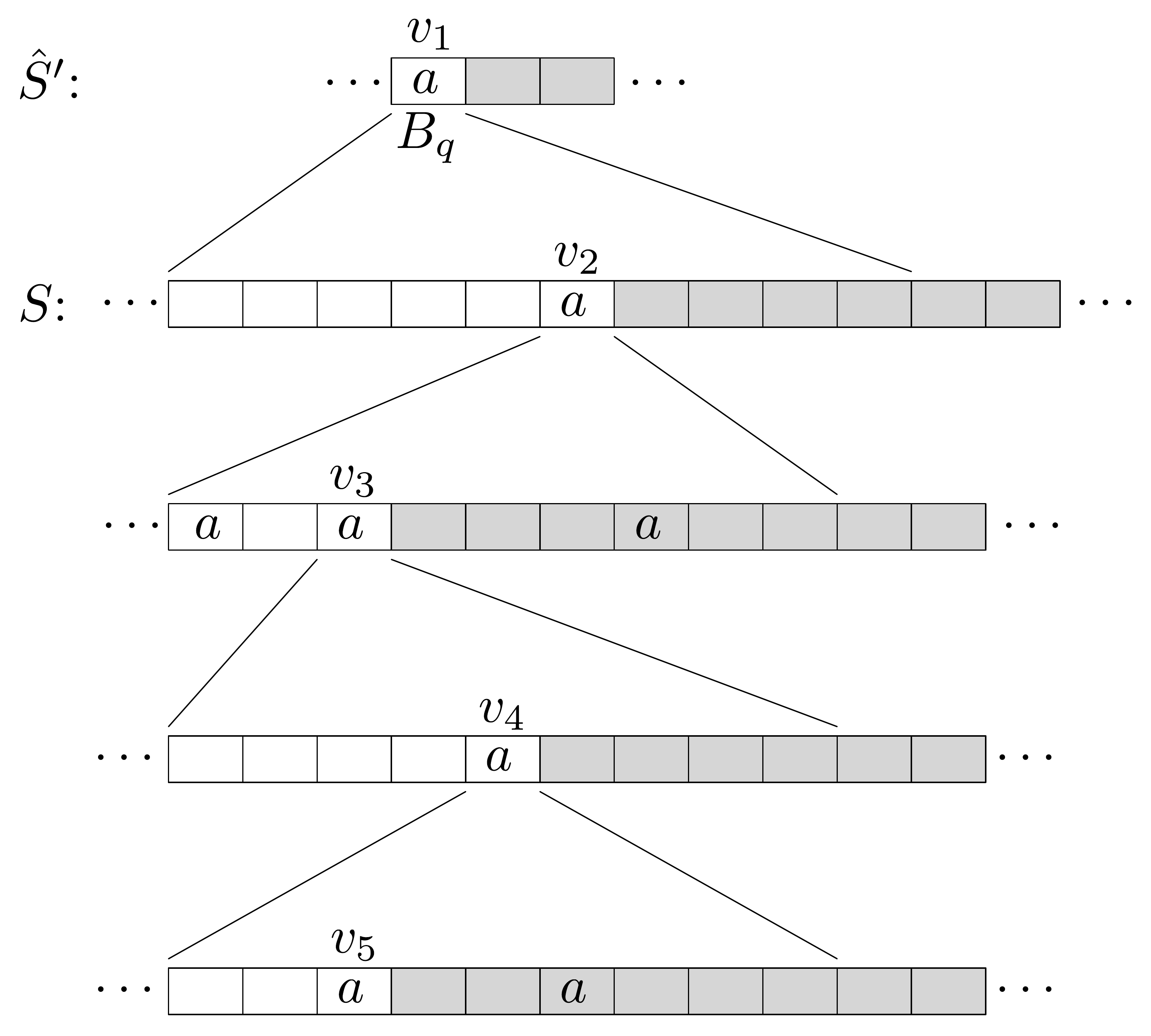}}
\caption{\label{fig:left-fertile}A newly $k$-left-fertile symbol $a\in B_q = \block(v_1)$ in $\GS'$.  As $a$ progresses down $\Tree_{|a}$ it continues
to be $k$-left fertile at $\block(v_2),\ldots,\block(v_5)$.  Since it molts to the right at blocks $\block(v_3)$ and $\block(v_5)$ it has type $(3,5)$.
It also molts to the left at $\block(v_3)$.  Were it to molt twice to the left at $\block(v_3)$, 
$\block(v_3)$ would then become a $k$-roost and the $a\in\block(v_3)$ $k$-infertile.}
\end{figure}
By the pigeonhole principle, if $(r-1){i^\star-1 \choose 2} + 1$ symbols in $F$ molted twice to the right then a subset $F'\subset F$ of 
$r$ of them have the same type,
say $(f,g)$.  However, this would imply that $Z$ is not $\dblPerm{r,4}$-free.  Since $k$-fertile symbols are middle symbols, every symbol in $F'$
appears at least once before and after $B_q$.  The occurrences of $F'$-symbols in $\block(v_g)$ are non-terminal, so they each
represent at least two occurrences in $Z$.  Finally, the $F'$-symbols appear twice at descendants of $B_q$ but to the right of $\block(v_g)$.
See Figure~\ref{fig:left-fertile}. 

To sum up, we let each $F$-symbol molt once to the left and once to the right while $k$-left fertile.  Some subset can molt more than once to the right,
but the total number of such terminals molted by these symbols is at most $(r-1){i^\star -1\choose 2}$.
\end{proof}

A nearly symmetric analysis can be applied to right fertile symbols.  The asymmetry comes from the fact that non-terminals
can molt two terminals to the left but only one to the right.

\begin{lemma}\label{lem:right-fertile}
Fix a block index $q\le \bl{\GS'}$ and let $F \subset B_q$ be those symbols newly $k$-right fertile, that is, 
they were not $k$-left fertile at any ancestor of $B_q$ in their respective derivation trees.
The total number of terminals molted by $F$-symbols before they become $k$-infertile is at most $2|F| + (r-1)({i^\star-1 \choose 2} + i^\star-1)$.
\end{lemma}

\begin{proof}
The argument is the same as above, except that we allow types $(f,f)$ if a symbol molts twice to the left at $\block(v_f)$.
There are now at most $({i^\star-1 \choose 2} + i^\star-1)$ possible types, and we cannot see $r$ symbols of the same type.
\end{proof}

According to Lemmas~\ref{lem:left-fertile} and \ref{lem:right-fertile}, it suffices to set $\PotFertile_k = \PotInfertile_k + 2$.
The total number of molted terminals unaccounted for, over all $q$, all $k<r$, counting both $k$-left fertile and $k$-right fertile symbols in $B_q$,
is $\Gm\cdot (r-1)^2(2{i^\star-1\choose 2} + i^\star-1) < \Gm\cdot [(r-1)(i^\star-1)]^2$,
which are covered by the last term of Recurrence~\ref{rec:dblPERMfour}.

The remaining task is to analyze the $k$-infertile $\rightarrow$ $(k+1)$-egg transition.

\begin{lemma}\label{lem:infertile-to-egg}
Let $u,v,w$ be distinct nodes such that $a,b\in\block(u),a\in\block(v),b\in\block(w)$, where $v$ is the parent
of $u$ in $\Tree_{|a}$ and $w$ is the parent of $u$ in $\Tree_{|b}$.
If $a,b$ were $k$-infertile in blocks $\block(v)$ and $\block(w)$ then at least one of $a,b$ became a 
$(k+1)$-egg when it was inserted into $\block(u)$.
\end{lemma}

\begin{proof}
This is a consequence of Lemma~\ref{lem:roost-props}(\ref{item:roost-props2},\ref{item:roost-props3}).
Without loss of generality $w$ is a strict ancestor of $v$, so $a$ was inserted into $\block(u)$ before $b$
was inserted into $\block(u)$.  Since the $a\in\block(v)$ was $k$-infertile, $\block(v)$ was a $k$-roost, by definition.
By Lemma~\ref{lem:roost-props}(\ref{item:roost-props2}), $\block(u)$ became a $k$-roost after $a$ was inserted there.
By Lemma~\ref{lem:roost-props}(\ref{item:roost-props3}), when $b$ was inserted in $\block(u)$ it became a $(k+1)$-egg.
\end{proof}

\begin{lemma}\label{lem:infertile-to-egg2}
Let $I\subset \Sigma(\GS_q)$ be those non-terminals that were $k$-infertile, non-$(k+1)$-eggs in $B_q$ but became 
$(k+1)$-eggs in $S_q$.  The number of terminals molted by $I$ symbols while they were $k$-infertile, non-$(k+1)$-eggs
is at most $2|I| + (r-1)(2{i^\star - 2\choose 2} + i^\star - 2)$.
\end{lemma}

\begin{proof}
Lemma~\ref{lem:infertile-to-egg} implies that on a path from $B_q$ to the root of $\Tree$ we encounter nodes $v_1=B_q,v_2,\ldots,v_l$, 
not necessarily adjacent, such that,
for each symbol $a\in I$, the set of blocks in which $a$ is $k$-infertile and not a quill is some {\em prefix} of 
$\block(v_1),\dots,\block(v_l)$, where $l\le i^\star-2$.  Call an $a\in I$ type $(\rightarrow,f,g)$
if it molted a terminal to the right in both $\block(v_f)$ and $\block(v_g)$, where $1\le f<g\le l$.
Call it type $(\leftarrow,f,g)$, where $1\le f\le g\le l$, if it molted a terminal to the left in both $\block(v_f)$ and $\block(v_g)$,
or two terminals to the left if $f=g$.  There are $2{l\choose 2} + l$ distinct types.  There cannot be $r$ symbols
of one type, for this would imply that $Z$ is not $\dblPerm{r,4}$-free.  
(The argument is the same as in the proof of Lemma~\ref{lem:left-fertile}.)
Since every symbol that molts more than two terminals is of at least one type, the total 
number of terminals molted by $I$ while being $k$-infertile, non-$(k+1)$-eggs is 
$2|I| + (r-1)(2{i^\star - 2\choose 2} + i^\star - 2)$.
\end{proof}

We set $\PotInfertile_k - \PotEgg_{k+1} = 2$, so the total number of terminals
unaccounted for, over all $q<\Gm$ and $k<r$, is at most $\Gm\cdot [(r-1)(i^\star-2)]^2$, which is covered by the last term of Recurrence~\ref{rec:dblPERMfour}.
Given the constraints we have established on potentials it suffices
to set $\Pot = \PotEgg_1 = 7(r-1)+1$, since $|\PotEgg_k - \PotFertile_k| = 5, |\PotFertile_k - \PotInfertile_k| = |\PotInfertile_k - \PotEgg_{k+1}| = 2$,
and $\PotEgg_r = 3$.  

\begin{remark}\label{rem:}
Observe the asymmetry in the arguments of Lemmas~\ref{lem:left-fertile}--\ref{lem:right-fertile} and Lemma~\ref{lem:infertile-to-egg2}.
In Lemmas~\ref{lem:left-fertile} and \ref{lem:right-fertile} we are tracking moltings that will happen ``in the future'' (below the level of $S$ in $\Tree$)
whereas in Lemma~\ref{lem:infertile-to-egg2} we are accounting for moltings that have already occurred at and above the level of $\GS'$ in $\Tree$.
\end{remark}

\subsection{Wrapping Up the Analysis}\label{sect:dblPERMfour-analysis}

Since $\PERM{r,2}(\cdot,\cdot)$ and $\dblPERM{r,2}(\cdot,\cdot)$ are both linear 
and $\Gm < m$, we can simplify Recurrence~\ref{rec:dblPERMfour} to
\begin{equation*}
\dblPERMfour(n,m) \le \sum_{1\le q\le \Gm} \dblPERMfour(\Ln_q,m_q) + \dblPERMfour(\Gn,\Gm) + C [\Gn + (i^\star)^2m],
\end{equation*}
for some constant $C$ depending only on $r$.  A straightforward proof by induction 
shows that for any $i\le i^\star$ and $j$ minimal such that $m \le a_{i,j}$, $\dblPERMfour(n,m) \le Ci(n + (i^\star)^2 j m)$.
Putting it all together we have, for $\|Z\|=n^\star$ and $\bl{Z}=m^\star$,
\begin{equation}\label{eqn:dblPERMfour}
|Z| \le \dblPERM{r,3}(n^\star,m^\star) \le \dblPERMfour(n^\star,m^\star) \le Ci^\star n^\star + C(i^\star)^3 j^\star m^\star.
\end{equation}
Eqn.~(\ref{eqn:dblPERMfour}) leads to an upper bound of $\dblPERM{r,3}(n,m) = O(n\alpha(n,m) + m\alpha^3(n,m))$,
which, by Lemma~\ref{lem:SparseVersusBlocked}, implies an upper bound of $\dblPERM{r,3}(n) = O(n\alpha^3(n))$.
Theorem~\ref{thm:dblPERMfour} reduces this to $O(n\alpha(n))$, which is asymptotically tight since $\dblPERM{r,3}(n) = \Omega(\DS{3}(n))$.

\begin{theorem}\label{thm:dblPERMfour}
For any $r\ge 2$, $\dblPERM{r,3}(n) = \Theta(n\alpha(n))$ and $\dblPERM{r,3}(n,m) = \Theta(n\alpha(n,m) + m)$.
\end{theorem}

\begin{proof}
Let $S$ be a $\dblPerm{r,4}$-free sequence.  To bound $|S|$ asymptotically we can assume, using Lemmas~\ref{lem:SparseVersusBlocked} and \ref{lem:orders12},
that $S$ consists of $m\le 2n$ blocks.  (If there are $m>2n$ blocks, remove up to $r-1$ symbols at block boundaries to make it $r$-sparse.
If the sequence is $r$-sparse, we can discard a constant fraction of occurrences to partition the sequence into $2n$ blocks.)
Choose $i$ to be minimal such that $m\le a_{i,j}$, where $j=\max\{3,\ceil{n/m}\}$.  Partition $S=S_1\cdots S_{\Gm}$ into $\Gm = \ceil{m/i^2}$ 
intervals, each consisting of $i^2$ blocks.  Define $\GS,\GS',\LS_q,$ etc. as usual.  
Applying Eqn.~(\ref{eqn:dblPERMfour}) with $i^\star=i$, we have $|\GS'| \le C(i\Gn + i^3 j \Gm) \le C(i(\Gn + jm)) = O(in)$.
Since each $\GfS_q,\GlS_q,$ and $\GmS_q$ is $\dblPerm{r,3}$-free and $\dblPERM{r,2}(n_q,m_q)=O(n_q+m_q)$ is linear, 
it follows that $|\GS| = O(in + m) = O(in)$.  We now apply Eqn.~(\ref{eqn:dblPERMfour}) to local symbols with $i^\star = 1$,
that is, for each index $q\le \Gm$, $j$ is chosen to be minimal such that $m_q \le a_{1,j}$.  Since $a_{1,j}=2^j$, $j = \ceil{\log m_q} \le \ceil{\log i^2}$.
It follows that $|\LS| = \sum_q |\LS_q| \le \sum_q C(\Ln_q + m_q\log m_q) = O(\Ln + m\log(i^2)) = O(n\log i)$.
Since $i=\alpha(n,m)+O(1)$, $|S| = |\GS| + |\LS| = O(n\alpha(n,m)) = O(n\alpha(n))$.
\end{proof}

Theorem~\ref{thm:dblPERMfour} and Lemma~\ref{lem:dblPermequiv}
immediately give us asymptotically sharp bounds on the extremal functions
for certain doubled forbidden sequences.

\begin{corollary}\label{cor:dblababa}
(See Nivasch~\cite[Rem.~5.1]{Nivasch10}, Pettie~\cite{Pettie-SoCG11}, 
Geneson, Prasad, and Tidor \cite{GenesonPT13}, and Klazar~\cite[p.~13]{Klazar02}.)
\begin{align*}
\dblDS{3}(n) &= \Theta(\dblPERM{2,3}(n)) = \Theta(n\alpha(n)),			\\
\Ex(\dbl(abcacbc),n) &= \Theta(\dblPERM{4,3}(n)) = \Theta(n\alpha(n)),	& \mbox{See~\cite{Pettie-SoCG11}}\\
\Ex(\dbl(abc abc a),n) &= \Theta(\dblPERM{3,3}(n)) = \Theta(n\alpha(n)),	& \mbox{See~\cite{Nivasch10}}\\
\intertext{and, more generally,}
\Ex(\dbl(1\cdots k\, 1\cdots k\, 1), n) &= \Theta(\dblPERM{r,3}(n)) = \Theta(n\alpha(n)),	
\intertext{where $r = (k-1)^3+1$.}
\end{align*}
\end{corollary}

\section{Double Davenport-Schinzel Sequences}\label{sect:dblDS}

Recall from Section~\ref{sect:derivation-tree-via-Ackermann} that the {\em canonical} derivation tree $\CTree(S)$ 
is obtained by decomposing $S$ in the least aggressive way possible, choosing $\Gm = \ceil{\bl{S}/2}$ whenever $\bl{S}>2$.
Figure~\ref{fig:canonical-derivation-tree} gives an example of such a tree.

\begin{figure}
\centering
\scalebox{.33}{\includegraphics{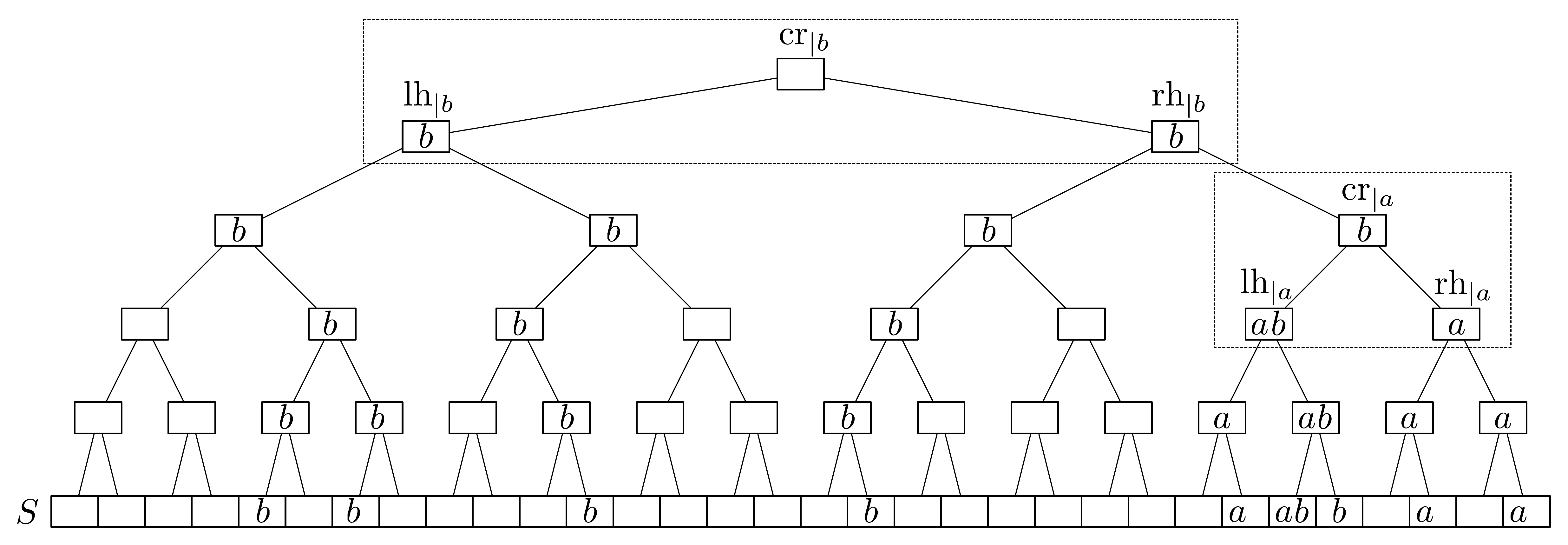}}
\caption{\label{fig:canonical-derivation-tree}An example of a canonical derivation tree for $S$.
Dashed boxes isolate the base case trees that assign $a,b\in\Sigma(S)$ their crowns and heads.}
\end{figure}

The structure of the canonical derivation tree is, in many respects, simpler than general derivation trees.  
For example, all wing nodes in any projection tree $\Tree_{|a}$, where $a\in\Sigma(S)$,
have either one or two children.  Those with two children ({\em branching} nodes) are associated
with precisely one quill and therefore one feather,\footnote{Recall that a {\em feather} of $\Tree_{|a}$ is the rightmost descendant
of a dove quill or leftmost descendant of a hawk quill.}
so counting the number of feathers is tantamount to counting branching wing nodes.

Nesting was a concept introduced in~\cite{Pettie-SoCG13} to analyze odd-order DS sequences.
Here we generalize it to deal with double DS sequences.

\begin{definition} {\bf (Nesting)}
Let $B$ be a block of $S$ containing $a,b\in\Sigma(S)$.  If $S$ contains either 
\[
a\, b\, b\; B\; b\, b\, a \hcm\mbox{ or }\hcm b\, a\, a\; B\; a\, a\, b
\] then $a$ and $b$ are called {\em double-nested} in $B$.
\end{definition}

Lemma~\ref{lem:double-nesting} can be thought of as a generalization of \cite[Lem.~4.3]{Pettie-SoCG13} to deal with double-nestedness.
Whereas~\cite[Lem.~4.3]{Pettie-SoCG13} assumed any derivation tree, 
Lemma~\ref{lem:double-nesting} refers to the canonical derivation tree $\CTree(S')$ as this makes the proof slightly simpler.
This assumption is actually without much loss of generality since any derivation tree obtained with uniform block partitions is
``contained'' in the canonical derivation tree, that is, its blocks are subsequences of the corresponding blocks in the canonical tree.

\begin{lemma}\label{lem:double-nesting}
Consider a sequence $S'$, its canonical derivation tree $\CTree(S')$,
and a leaf $v$ for which $a,b\in\block(v)$.  Let $S$ be obtained from $S'$ by 
substituting, for each leaf $u\neq v$, a sequence $S(u)$ containing at least two copies
of each symbol in $\block(u)$.  (The block $\block(v)$ appears verbatim in $S$.)
If $v$ is neither a wingtip nor feather in both $\CTree_{|a}$ and $\CTree_{|b}$ then,
in $S$, $a$ and $b$ are double-nested in $\block(v)$.
\end{lemma}

\begin{proof}
Without loss of generality we can assume that $v$ is a dove in $\CTree_{|a}$
and $\crown_{|b}$ is ancestral to $\crown_{|a}$.
Because $v$ is neither a wingtip nor feather in $\CTree_{|a}$, it must be distinct from the leftmost and rightmost
leaf descendants of $\wing_{|a}(v)$, namely $\ltip_{|a}$ and $\feather_{|a}(v)$.
Moreover, since $v$ is a dove in $\CTree_{|a}$ it descends from the right child of $\wing_{|a}(v)$, namely $\quill_{|a}(v)$.
Partition $S$ into four intervals
\begin{description}
\setlength{\itemsep}{0pt}
\item[$I_1$]: everything preceding $\block(\ltip_{|a})$.
\item[$I_2$]: everything from $I_1$ to the beginning of $\block(v)$.
\item[$I_3$]: everything from the end of $\block(v)$ to the end of $\block(\feather_{|a}(v))$.
\item[$I_4$]: everything following $I_3$.
\end{description}
If $b$ appeared in both $I_1$ and $I_4$ then $a,b\in\block(v)$ would clearly be double-nested in $S$.
Therefore it suffices to consider two cases,
(1) $I_1$ contains no $b$s,
and
(2) $I_4$ contains no $b$s.  
Figures~\ref{fig:I2I4} and \ref{fig:I1I3} illustrate the two cases.

\paragraph{Case 1.} 
The wingtip $\ltip_{|b}$ must be in interval $I_2$, though it may be identical to $\ltip_{|a}$.
Since $\wing_{|a}(v)$ is ancestral to both $\ltip_{|b}$ and $v$, and is a strict descendant of $\crown_{|b}$, 
it follows that $v$ is a dove in $\CTree_{|b}$ and that $\wing_{|b}(v)$ is a descendant of $\wing_{|a}(v)$.
The rightmost descendant of $\wing_{|b}(v)$ in $\Tree_{|b}$ is $\feather_{|b}(v)$, which is distinct from $v$.
Since $\wing_{|a}(v)$ is a descendant of $\lefthead_{|a}$, any descendant of $\righthead_{|a}$, such as $\rtip_{|a}$,
lies to the right of $\feather_{|b}(v)$, in interval $I_4$.  By the same reasoning, $\rtip_{|b}$ lies in $I_4$.

Regardless of whether $\ltip_{|a}$ and $\ltip_{|b}$ are identical or distinct, $\block(v)$ is preceded, in $S$,
by either $abb$ or $baa$.
In the first case $\ltip_{|a},\ltip_{|b},v,\feather_{|b}(v),\rtip_{|a}$ certify that $a,b$ are double-nested in $\block(v)$;
see Figure~\ref{fig:I2I4}.
In the latter case $\ltip_{|b}=\ltip_{|a},v,\feather_{|a}(v),\rtip_{|b}$ certify that $a,b$ are double-nested in $\block(v)$.
\begin{figure}
\centering
\scalebox{.35}{\includegraphics{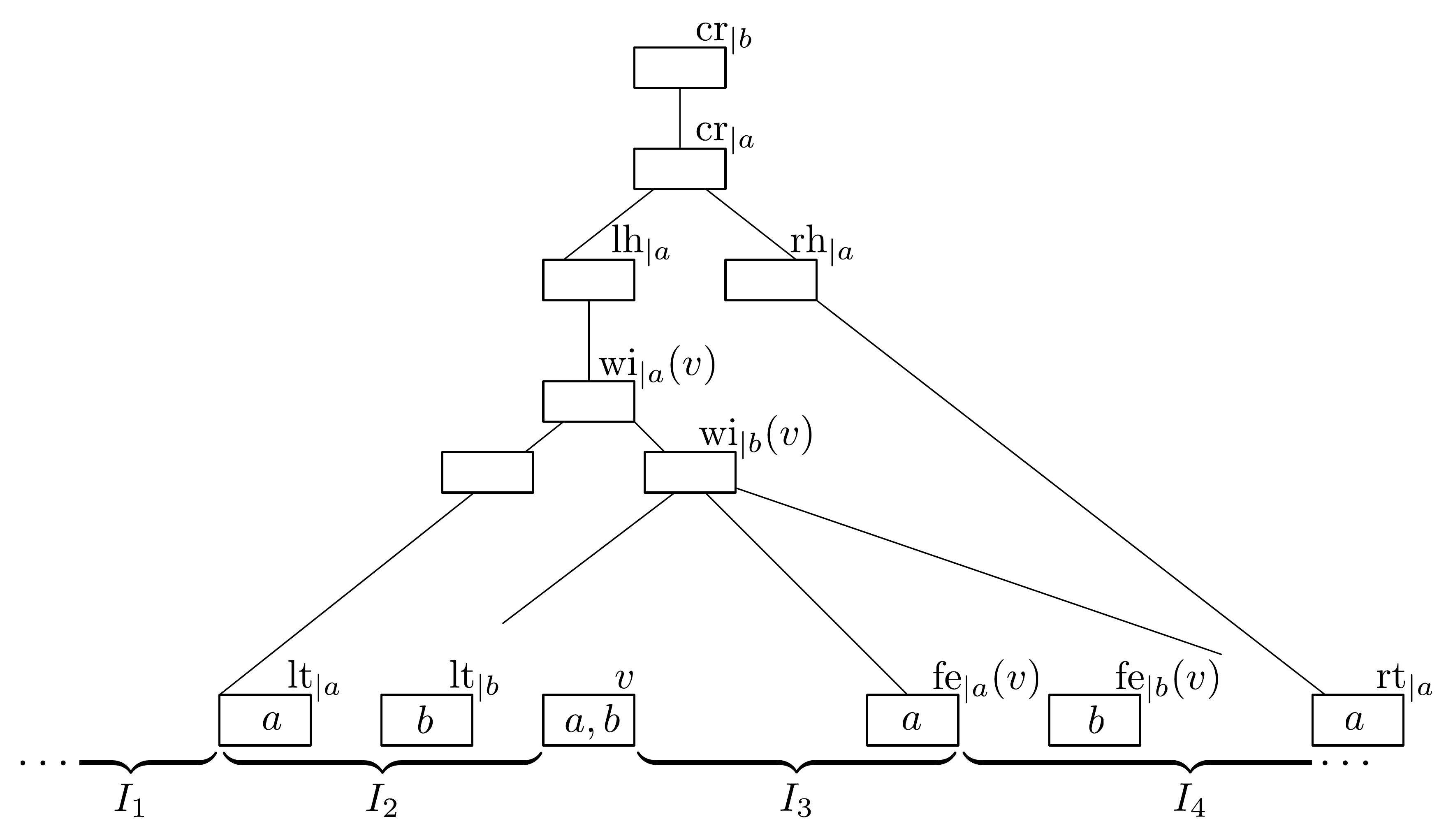}}
\caption{\label{fig:I2I4}In Case 1 interval $I_1$ contains no $b$s.  
Contrary to the depiction, $\ltip_{|a}$ and $\ltip_{|b}$ are not necessarily distinct,
nor are $\wing_{|a}(v)$ and $\wing_{|b}(v)$
or $\crown_{|a}$ and $\crown_{|b}$.  In this depiction $\quill_{|a}(v)$, 
the right child of $\wing_{|a}(v)$, happens to be identical to $\wing_{|b}(v)$.}
\end{figure}

\paragraph{Case 2.}
The wingtip $\rtip_{|b}$ must lie in $I_3$, so $v$ and $\rtip_{|b}$ are both descendants
of $\quill_{|a}(v)$, the right child of $\wing_{|a}(v)$.  
It follows that $v$ is a hawk in $\CTree_{|b}$ and that no descendants of $\wing_{|b}(v)$ are in interval $I_1$.
Since $\feather_{|b}(v)$ is the leftmost descendant of $\wing_{|b}(v)$ in $\CTree_{|b}$, and $\feather_{|b}(v)\neq v$,
the distinct nodes $\ltip_{|a},\feather_{|b}(v),v,\rtip_{|b},\rtip_{|a}$ certify that $a,b$ are double-nested in $\block(v)$.
See Figure~\ref{fig:I1I3}.
\begin{figure}
\centering
\scalebox{.35}{\includegraphics{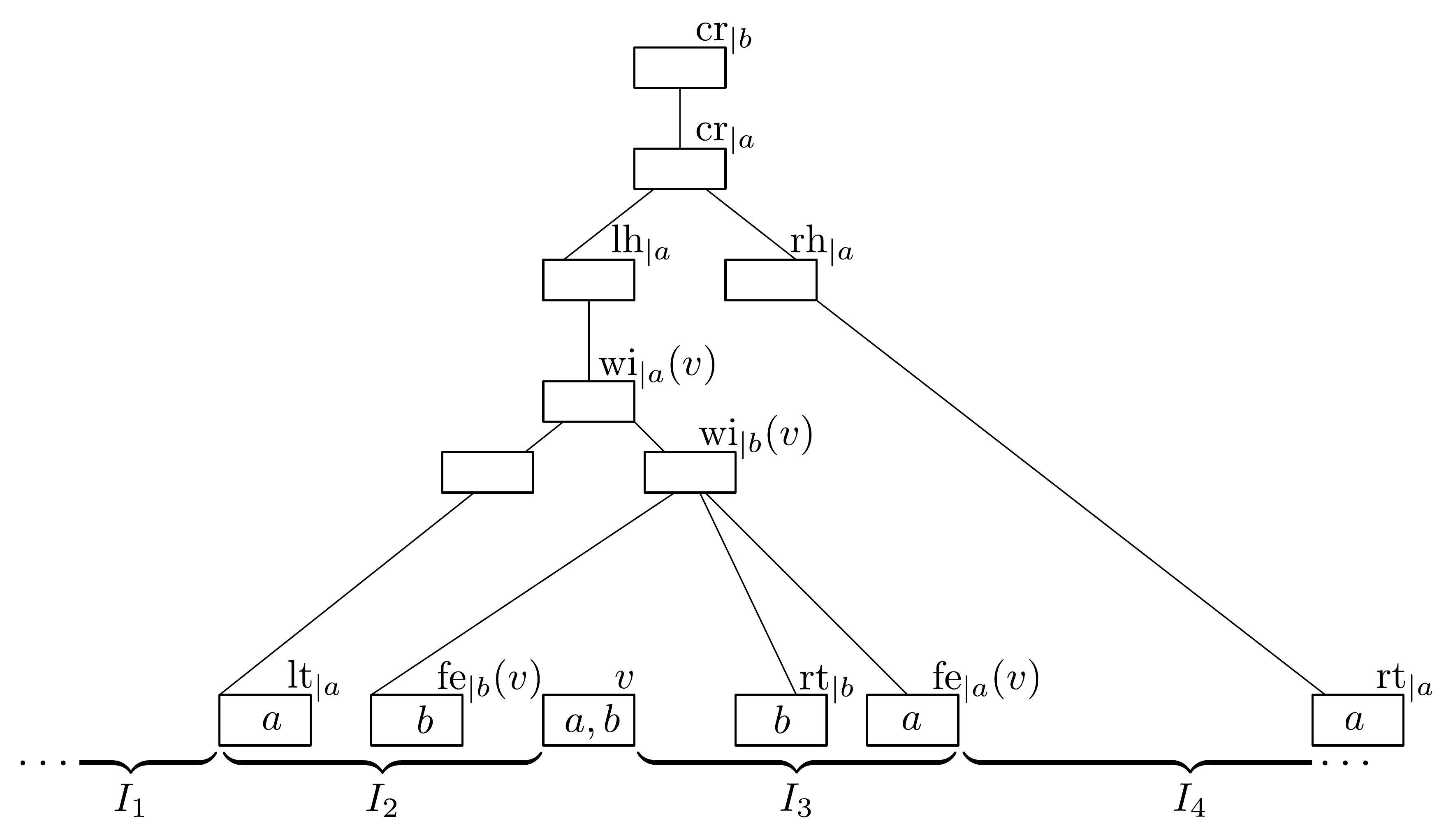}}
\caption{\label{fig:I1I3}In Case 2 interval $I_4$ contains no $b$s. 
Contrary to the depiction, $\rtip_{|b}$ and $\feather_{|a}(v)$ are not necessarily distinct.}
\end{figure}
\end{proof}

Recurrence~\ref{rec:feathers} gives a significantly simpler method for bounding the number of feathers, compared to
\cite[Recs.~5.1 and 7.6]{Pettie-SoCG13}.  Whereas~\cite{Pettie-SoCG13} considered feathers in an arbitrary derivation tree,
Recurrence~\ref{rec:feathers} {\em only} considers the canonical derivation tree.

\begin{recurrence}\label{rec:feathers}
Let $S$ be an $m$-block, order-$s$ DS sequence over an $n$-letter alphabet
and $\Tree = \CTree(S)$ be its canonical derivation tree.
Define $\Feather{s}(n,m)$
to be the maximum number of feathers of one type (dove or hawk)
in such a sequence, where {\em feather} is with respect to $\Tree$.
For any $s\ge 2$,
\begin{align*}
\Feather{s}(n,2) &= 0\\
\Feather{2}(n,m) &< m\\
\intertext{and for any uniform block partition $\{m_q\}_{1\le q\le \Gm}$ and alphabet partition $\{\Gn\}\cup\{\Ln_q\}_{1\le q\le \Gm}$,}
\Feather{s}(n,m) &\le \sum_{q=1}^{\Gm} \Feather{s}(\Ln_q,m_q) + \Feather{s}(\Gn,\Gm) + \Feather{s-1}(\Gn,m) + \Gn\\
\end{align*}
\end{recurrence}

\begin{proof}
Suppose we only wish to bound dove feathers.
If there are only two blocks then all occurrences are wingtips and feathers are not wingtips.  This gives the first equality.
In the most extreme case every non-wingtip is a dove feather, so $\Feather{s}(n,m) \le \DS{s}(n,m) - 2n$.  
In particular, $\Feather{2}(n,m) \le \DS{2}(n,m) - 2n < m$.
Decompose $S$ into $\GS,\GS',\GfS_q,\GlS_q,\GmS_q$ in the usual way with respect to the given uniform block partition.
Let $\GTree = \CTree(\GS')$ be the canonical derivation tree of the contracted global sequence $\GS'$.
It follows that $\GfS_q$ is an order-$(s-1)$ DS sequence.
Define $\GfTree_q = \CTree(\GfS_q)$ to be its canonical derivation tree.  
\begin{figure}
\centering
\scalebox{.4}{\includegraphics{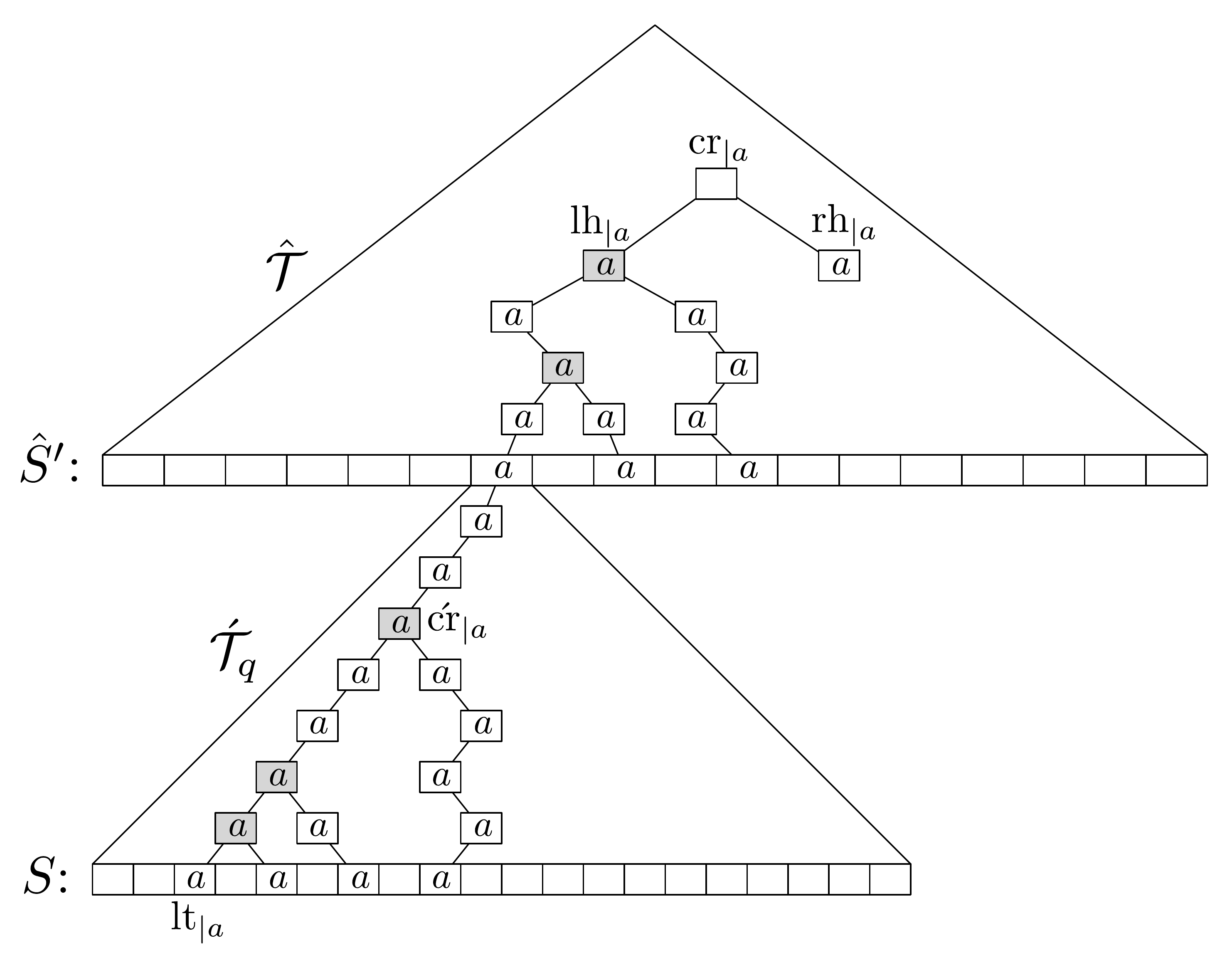}}
\caption{\label{fig:feather-count}Counting dove feathers in $T_{|a}$ is tantamount to counting branching nodes on the left wing of $\Tree_{|a}$.}
\end{figure}
The branching nodes on the left wing of $\Tree_{|a}$, where $a\in\Sigma(\GfS_q)$, consist of (i) the branching nodes on the left wing of 
$\GTree_{|a}$, (ii) the branching nodes on the left wing of $(\GfTree_q)_{|a}$, and (iii) the crown $\fcrown_{|a}$ of $(\GfTree_q)_{|a}$, 
which is on the left wing of $\Tree_{|a}$ but not $(\GfTree_q)_{|a}$.  
Each branching node is identified with one feather in $\Tree_{|a}$.
The total number of branching nodes/feathers covered by (i), summed over all $a\in\Sigma(\GS)$, is at most $\Feather{s}(\Gn,\Gm)$.
The total number covered by (ii), summed over all $q\le \Gm$ and $a\in \Sigma(\GS_q)$,
is $\sum_q \Feather{s-1}(\Gfn_q,m_q) \le \Feather{s-1}(\Gn,m)$.  The number covered by (iii) is clearly $\Gn$, which gives the last inequality.
\end{proof}

Recurrence~\ref{rec:dblDS} generalizes~\cite[Rec.~3.1]{Nivasch10} and \cite[Recs.~3.3, 5.2, and 7.7]{Pettie-SoCG13}, from
DS sequences to double DS sequences.  When $s=3$ or $s\ge 4$ is even, Recurrence~\ref{rec:dblDS} is substantively no different
than Recurrence~\ref{rec:dblPERM} for $\dblPerm{r,s+1}$-free sequences.

\begin{recurrence}\label{rec:dblDS}
Let $s,n,$ and $m$ be the order, alphabet size, and block count parameters.
Let $\{m_q\}_{1\le q\le \Gm}$ be a uniform block partition, where $\Gm\ge 2$,
and $\{\Gn\}\cup\{\Ln_q\}_{1\le q\le \Gm}$ be an alphabet partition.
When $\Gm=2$, for any $s\ge 3$,
\begin{align*}
\dblDS{s}(n,m) &\le \sum_{q\in\{1,2\}} \dblDS{s}(\Ln_q,m_q) + \dblDS{s-1}(2\Gn,m) + 2\Gn.
\intertext{When $\Gm>2$ and either $s=3$ or $s\ge 4$ is even,}
\dblDS{s}(n,m) &\le \sum_q \dblDS{s}(\Ln_q,m_q)
				+ \dblDS{s}(\Gn,\Gm)
				+ 2\cdot\dblDS{s-1}(\Gn,m)
				+ \dblDS{s-2}(\DS{s}(\Gn,\Gm),m) + 2\cdot \DS{s}(\Gn,\Gm),
\intertext{and when $s\ge 5$ is odd,}
\dblDS{s}(n,m) &\le \sum_{q=1}^{\Gm} \dblDS{s}(\Ln_q,m_q) 
				+ \dblDS{s}(\Gn,\Gm)
				+ 2\cdot \dblDS{s-1}(\Gn,m)
				+ \dblDS{s-2}(2\cdot \Feather{s}(\Gn,\Gm),m) + 4\cdot \Feather{s}(\Gn,\Gm)\\
				&\hcm + \dblDS{s-3}(\DS{s}(\Gn,\Gm),m) + 2\cdot \DS{s}(\Gn,\Gm)
\end{align*}
\end{recurrence}

\begin{proof}
First consider the case when $s\ge 5$ is odd.
Let $S$ be an order-$s$ double DS sequence, decomposed into $\GS$ and $\{\LS_q\}$ as usual.  
The contribution of local symbols is $\sum_q \dblDS{s}(\Ln_q,m_q)$.
If a global symbol occurs exactly once in an $\GS_q$ this occurrence is a {\em singleton}. 
Let $\GSsingle\subseq\GS$ be the subsequence of singletons and $\GSnonsingle\subseq \GS$ be the subsequence
of non-singletons.  By definition $\GSsingle$ is partitioned into $\Gm$ blocks, 
so $|\GSsingle| \le \dblDS{s}(\Gn,\Gm)$.
Symbols in $\Sigma(\GSnonsingle_q)$ are classified as {\em first}, {\em last}, and {\em middle} if they appear, in $\GSnonsingle$, 
after $\GSnonsingle_q$ but not before, before $\GSnonsingle_q$ but not after, and both before and after $\GSnonsingle_q$, respectively.
In the worst case these three criteria are exhaustive. However, it may be that all non-singleton occurrences of 
a symbol appear {\em exclusively} in $\Sigma(\GSnonsingle_q)$.  In this case we call the symbol {\em first} if it appears after interval
$q$ in $\GSsingle$ and {\em last} if it is not first and appears before interval $q$ in $\GSsingle$.
Define $\GfS_q,\GlS_q,\GmS_q\subseq \GSnonsingle_q$ to be the subsequences of first, last, 
and middle occurrences in $\GSnonsingle_q$.

If we remove the last occurrence of each letter from $\GfS_q$, or the first occurrence of each letter from $\GlS_q$, the resulting
sequence is an order-$(s-1)$ double DS sequence.  The contribution of first and last non-singletons is therefore at most
\[
\sum_q \SqBrack{\dblDS{s-1}(\Gfn_q,m_q) + \Gfn_q + \dblDS{s-1}(\Gln_q,m_q)  + \Gln_q} \le 2(\dblDS{s-1}(\Gn,m) + \Gn).
\]
Obtain $\GSnonsingle'=B_1\cdots B_{\Gm}$ from $\GSnonsingle$ by contracting each interval $\GSnonsingle_q$ into a single block $B_q$.
Since occurrences in $\GSnonsingle'$ each represent at least two occurrences in $\GSnonsingle$, 
we can conclude\footnote{This is not quite true.  As discussed in Remark~\ref{rem:interleave}, 
we can make this inference when bounding $\dblDS{s}$ asymptotically.}
 that $|\GSnonsingle'| \le \DS{s}(\Gn,\Gm)$.

Let $\ddot{\Tree} = \CTree(\GSnonsingle')$ be the canonical derivation tree of $\GSnonsingle'$.
Define $\GfeatherS'$ to be the subsequence of $\GSnonsingle'$ consisting of feathers with respect to $\ddot{\Tree}$ (both dove and hawk) 
and let $\GfeatherS$ be the subsequence
of $\GSnonsingle$ begat by symbols in $\GfeatherS'$.  It follows that $|\GfeatherS'| \le 2\cdot \Feather{s}(\Gn,\Gm)$ since $\Feather{s}$ only counts feathers of one type (dove or hawk).
Define $\GnonfeatherS'\subseq \GSnonsingle'$ to be the subsequence of non-feather, non-wingtips with respect to $\ddot{\Tree}$, 
and define $\GnonfeatherS\subseq\GSnonsingle$ analogously.  
Since $\GfeatherS$ consists solely of middle symbols, removing the first and last occurrence of each letter
in $\GfeatherS_q$ leaves an order-$(s-2)$ double DS sequence, hence
\begin{align*}
|\GfeatherS| = \sum_q |\GfeatherS_q| &\le \sum_q (\dblDS{s-2}(\Gnf_q,m_q) + 2\Gnf_q)\\
					&\le \dblDS{s-2}\paren{\sum_q \Gnf_q, m} + 2\sum_q \Gnf_q\\
					&\le \dblDS{s-2}(|\GfeatherS'|,m) + 2(|\GfeatherS'|)\\
					&\le \dblDS{s-2}(2\cdot \Feather{s}(\Gn,\Gm),m) + 4\cdot \Feather{s}(\Gn,\Gm)
\end{align*}
We have accounted for every part of $S$ except for $\GnonfeatherS$.  Fix an interval $q$ and $a,b\in\Sigma(\GnonfeatherS_q)$.
Since $a,b\in B_q$ are neither feathers nor wingtips in $\ddot{\Tree}$, 
Lemma~\ref{lem:double-nesting} implies that $\GSnonsingle$ contains $a\, b\, b\, \GSnonsingle_q\, b\, b\, a$.  
Suppose we remove the first and last occurrence of each letter in $\GnonfeatherS_q$.  (These letters are underlined below.) 
The resulting sequence must be an order-$(s-3)$ double
DS sequence, for if it contained a doubled alternating sequence with length $s-1$, which is even, we would see either
\begin{align*}
a \; b \; b \; \left| \;\; \overbrace{\underline{a} \; a \; b \; b \; \cdots a \; a \; b \; \underline{b}}^{s-1 \, \operatorname{alternations}} \;\; \right| \; b \; b \; a
\intertext{or}
a \; b \; b \; \left| \;\; \overbrace{\underline{b} \; b \; a\; a \; \cdots b \; b \; a \; \underline{a}}^{s-1 \, \operatorname{alternations}} \;\; \right| \; b \; b \; a,
\end{align*}
contradicting the fact that $S$ is an order-$s$ double DS sequence.
We can therefore bound $|\GnonfeatherS|$ by
\begin{align*}
\sum_q |\GnonfeatherS_q| &\le  \sum_q (\dblDS{s-3}(\Gnnf_q,m_q) + 2\Gnnf_q)\\
					&\le \dblDS{s-3}\paren{\sum_q \Gnnf_q, m} + 2\sum_q \Gnnf_q\\
					&\le \dblDS{s-3}(|\GnonfeatherS'|,m) + 2|\GnonfeatherS'|\\
					&\le \dblDS{s-3}(|\GSnonsingle'|-2\Gn,m) + 2(|\GSnonsingle'|-2\Gn)\\
					&\le \dblDS{s-3}(\DS{s}(\Gn,\Gm)-2\Gn,m) + 2(\DS{s}(\Gn,\Gm)-2\Gn)
\end{align*}
This establishes the recurrence for odd $s\ge 5$.
When $s=3$ or $s\ge 4$ is even, we ignore the distinction between feathers and non-feathers
and bound $|\GmS|$ by $\dblDS{s-2}(\DS{s}(\Gn,\Gm)-2\Gn,m) + 2(\DS{s}(\Gn,\Gm)-2\Gn)$.
When $S=S_1S_2$ consists of $\Gm=2$ intervals, no symbols are classified as middle, so it suffices
to account for first, last, and local occurrences only.  After discarding the last occurrence
of each symbol from $\GS_1$ and the first from $\GS_2$, what remains are order-$(s-1)$ double 
DS sequences, so $|\GS| \le 2\Gn + \dblDS{s-1}(\Gn,m_1) + \dblDS{s-1}(\Gn,m_2) \le 2\Gn + \dblDS{s-1}(2\Gn,m)$.
\end{proof}

Recurrence~\ref{rec:DS} is similar to \cite[Rec.~5.2]{Pettie-SoCG13} but presented in the style of Recurrence~\ref{rec:dblDS}.
The proof is essentially the same as that of Recurrence~\ref{rec:dblDS} except that we do not need to distinguish singletons
from non-singletons, nor do we need to remove symbols from 
$\GfS_q,\GlS_q,\GfeatherS_q,\GnonfeatherS_q,$ or $\GmS_q$ in order to make them double DS sequences with
order $s-1$ or $s-2$ or $s-3$, as the case may be.

\begin{recurrence}\label{rec:DS}
Let $s,n,$ and $m$ be the order, alphabet size, and block count parameters.
Let $\{m_q\}_{1\le q\le \Gm}$ be a uniform block partition, where $\Gm\ge 2$,
and $\{\Gn\}\cup\{\Ln_q\}_{1\le q\le \Gm}$ be an alphabet partition.
When $\Gm=2$, for any $s\ge 3$,
\begin{align*}
\DS{s}(n,m) &\le \sum_{q\in\{1,2\}} \DS{s}(\Ln_q,m_q) + \DS{s-1}(2\Gn,m).
\intertext{When $\Gm>2$ and either $s=3$ or $s\ge 4$ is even,}
\DS{s}(n,m) &\le \sum_q \DS{s}(\Ln_q,m_q)
				+ 2\cdot\DS{s-1}(\Gn,m)
				+ \DS{s-2}(\DS{s}(\Gn,\Gm)-2\Gn,m)
\intertext{and when $s\ge 5$ is odd,}
\DS{s}(n,m) &\le \sum_{q=1}^{\Gm} \DS{s}(\Ln_q,m_q) 
				+ 2\cdot \DS{s-1}(\Gn,m)
				+ \DS{s-2}(2\cdot \Feather{s}(\Gn,\Gm),m) 
				+ \DS{s-3}(\DS{s}(\Gn,\Gm),m)
\end{align*}
\end{recurrence}

Lemma~\ref{lem:rec-ds-dblds} states some bounds on $\Feather{s},\DS{s},$ and $\dblDS{s}$ in terms
of coefficients $\{\fea{s,i},\ds{s,i},\dblds{s,i}\}$ and the $i$th row-inverse of Ackermann's function, for any $i\ge 1$.
Refer to~\cite[Appendices B and C]{Pettie-SoCG13} for proofs of similar lemmas,
and to the discussion following Lemma~\ref{lem:ub-PERM-dblPERM}.

\begin{lemma}\label{lem:rec-ds-dblds}
Fix parameters $i\ge 1$, $s\ge 3$, and $c\ge s-2$ and let $n,m$ be the alphabet size
and block count.  Let $j$ be minimal such that $m \le (a_{i,j})^c$.  Then $\Feather{s},\DS{s},$ and $\dblDS{s}$ are bounded
by
\begin{align*}
\Feather{s}(n,m) &\le \fea{s,i}\paren{n + O((cj)^{s-2}m)}\\
\DS{s}(n,m) 	&\le \ds{s,i}\paren{n + O((cj)^{s-2}m)}\\
\dblDS{s}(n,m) &\le \dblds{s,i}\paren{n + O((cj)^{s-2}m)}
\end{align*}
where $\{\fea{s,i},\ds{s,i},\dblds{s,i}\}$ are defined as follows.
\begin{align*}
\fea{2,i} &= 0							& \mbox{all $i$}\\
\fea{s,1} &= \zero{\fea{s-1,1} + 1}\hcm[11]	& \mbox{$s\ge 3$}\\
\fea{s,i} &= \fea{s,i-1} + \fea{s-1,i} + 1		& \mbox{$s\ge 3, i\ge 2$}\istrut[3]{0}\\
\ds{1,i} &= 1							& \mbox{all $i$}\\
\ds{2,i} &= 2							& \mbox{all $i$}\\
\dblds{1,i} &= 2							& \mbox{all $i$}\\
\dblds{2,i} &= 5							& \mbox{all $i$}\istrut[3]{0}\\
\ds{s,1} &= 2\ds{s-1,1} = 2^{s-1}				& \mbox{$s\ge 3$}\\
\dblds{s,1} &= 2(\dblds{s-1,1}+1) = 2^{s+1}-2^{s-2} - 2	& \mbox{$s\ge 3$}\istrut[3]{0}\\
\ds{s,i} &= 
\zero{\left\{
\begin{array}{l}
\istrut[3]{0}\zero{2\ds{s-1,i}  + \ds{s-2,i}(\ds{s,i-1}-2)}						\\
\zero{2\ds{s-1,i}  + 2\ds{s-2,i}\fea{s,i-1} + \ds{s-3,i}\ds{s,i-1}}\hcm[9.5]	
\end{array}
\right.
\begin{array}{r}
\istrut[3]{0}\mbox{$s=3$ or even $s\ge 4$}\\
\mbox{odd $s\ge 5$}
\end{array}
}
\istrut[7]{0}
\\
\dblds{s,i} &= 
\zero{\left\{
\begin{array}{l}
\istrut[3]{0}\zero{\dblds{s,i-1} + 2\dblds{s-1,i} + (\dblds{s-2,i}+2)\ds{s,i-1}}							\\
\zero{\dblds{s,i-1} + 2\dblds{s-1,i} + 2(\dblds{s-2,i}+2)\fea{s,i-1} + (\dblds{s-3,i}+2)\ds{s,i-1}}\hcm[9.5]	
\end{array}
\right.
\begin{array}{r}
\istrut[3]{0}\mbox{$s=3$ or even $s\ge 4$}\\
\mbox{odd $s\ge 5$}
\end{array}
}
\\
\end{align*}
\end{lemma}

When applying Lemma~\ref{lem:rec-ds-dblds}, the tightest bounds are obtained by setting $i=\alpha(n,m)+O(1)$, 
which is $\alpha(n)+O(1)$ whenever $j=O(1)$.  Lemma~\ref{lem:fea-dblds-closedform} gives closed form bounds on
the coefficients $\{\ds{s,i},\dblds{s,i},\fea{s,i}\}$, which immediately yield sharp bounds on the extremal functions 
$\DS{s}(n,m)$ and $\dblDS{s}(n,m)$ for DS and double DS sequences partitioned into blocks.

\begin{lemma}\label{lem:fea-dblds-closedform}
{\bf (Closed Form Bounds)}
For all $s\ge 3, i\ge 1$, we have 
\begin{align*}
\fea{s,i} &= {i+s-2\choose s-2}-1\\
\ds{3,i} &= 2i+2\\
\dblds{3,i} &= \Theta(i^2)\\
\ds{4,i},\dblds{4,i} &= \Theta(2^i)\\
\ds{5,i}, \dblds{5,i} &= \Theta(i2^i)\\
\ds{s,i},\dblds{s,i} &\le 2^{i+O(1) \choose t} & \mbox{where $t = \floor{\f{s-2}{2}}$.}
\end{align*}
\end{lemma}

\begin{proof}
The expression for $\fea{s,i}$ holds in the base cases, when $s=2$ or $i=1$.  By Pascal's identity 
it holds in general since
\[
\fea{s,i} = \fea{s,i-1} + \fea{s-1,i} + 1 = {i+s-3\choose s-2} + {i+s-3\choose s-3} - 1 = {i+s-2\choose s-2} - 1.
\]
When $s\in\{3,4\}$, $\ds{s,i}$ and $\dblds{s,i}$ are identical to $\K{s,i}$ and $\dblK{s,i}$, and therefore satisfy
the same bounds from Lemma~\ref{lem:closed-form-K-dblK}.
Define $C_4$ such that $\ds{4,i} \le 2^{i+C_4}$.  Assuming inductively that for some sufficiently large 
$C_5$, $\ds{5,i-1} \le (i-1)2^{(i-1)+C_5}$,
we have
\begin{align}
\ds{5,i} &\le 2\ds{4,i} + 2\ds{3,i}\fea{5,i-1} + \ds{2,i}\ds{5,i-1}\nonumber\\
	&\le 2^{i+C_4+1} + 2(2i+2)\cdot\mbox{${i+2\choose 3}$} + 2\cdot (i-1)2^{i-1 + C_5}\nonumber\\
	&\le i2^{i+C_5}.\nonumber
\intertext{We claim that there are constants $\{C_s\}$ such that, for all $s>5$, $\ds{s,i} \le 2^{i+C_s\choose t}$.
When $s>4$ is even,}
\ds{s,i} &\le 2\ds{s-1,i} + \ds{s-2,i}\ds{s,i-1}\nonumber\\
	&\le 2^{{i+C_{s-1}\choose t-1}+1} + 2^{i+C_{s-2}\choose t-1}2^{i-1 + C_s\choose t}\nonumber\\
	&\le 2^{i+C_s\choose t}, \, \mbox{ for some $C_s > C_{s-1} > C_{s-2}$.}\nonumber
\intertext{
When $s>5$ is odd, whether $s-2=5$ or not, $\ds{s-2,i} \le i2^{i+C_{s-2}\choose t-1}$ by the inductive hypothesis,
so}
\ds{s,i} &\le 2\ds{s-1,i} + 2\ds{s-2,i}\fea{s,i-1} + \ds{s-3,i}\ds{s,i-1}\nonumber\\
	&\le 2^{{i+C_{s-1}\choose t}+1} + i2^{{i+C_{s-2}\choose t-1}+1}\cdot\mbox{${i+s-3\choose s-2}$} + 2^{i+C_{s-3}\choose t-1}2^{i-1+C_s\choose t}\nonumber\\
	&\le 2^{{i+C_{s-1}\choose t}+1} + i2^{{i+C_{s-2}\choose t-1}+1}\cdot\mbox{${i+s-3\choose s-2}$} + 2^{-(C_s - C_{s-3})} 2^{{i+C_s\choose t-1}+{i-1+C_s\choose t}}\label{eqn:ds:s-odd-1}\\
	&\le 2^{i+C_s\choose t}.\label{eqn:ds:s-odd-2}
\end{align}
Inequality~(\ref{eqn:ds:s-odd-1}) follows since $t-1\ge 1$ and Inequality~(\ref{eqn:ds:s-odd-2}) follows since, for $C_s$ sufficiently large,
$2^{i+C_s\choose t}$ dominates both $\poly(i)\cdot 2^{i+C_{s-2}\choose t-1}$ and $2^{{i+C_{s-1}\choose t}+1}$.
It is straightforward to show the same bounds hold on $\dblds{s,i}$, for $s\ge 4$, with respect to different constants $\{D_s\}$. 
That is, $\dblds{s,i} \le 2^{i+D_s\choose t}$ when $s\neq 5$ and $\dblds{5,i} \le i2^{i+D_5}$.
\end{proof}

Choosing $i=\alpha(n,m)+O(1)$, Lemmas~\ref{lem:rec-ds-dblds} and \ref{lem:fea-dblds-closedform} 
imply that 
\begin{align*}
\DS{3}(n,m) &= O((n+m)\alpha(n,m))\\
\dblDS{3}(n,m) &= O((n+m)\alpha^2(n,m))\\
\DS{4}(n,m),\dblDS{4}(n,m) &= O((n+m)2^{\alpha(n,m)})\\
\DS{5}(n,m),\dblDS{5}(n,m) &= O((n+m)\alpha(n,m)2^{\alpha(n,m)})\\
\DS{s}(n,m),\dblDS{s}(n,m) &= O((n+m)2^{\alpha^t(n,m)/t! \,+\, O(\alpha^{t-1}(n,m))})
\end{align*}
When $m=O(n)$ these bounds are all sharp, with the exception of $\dblDS{3}$, which was already handled in Section~\ref{sect:dblPERMfour}.
Using the best transformations from 2-sparse to blocked sequences from Lemma~\ref{lem:SparseVersusBlocked},
we obtain all the bounds on $\DS{s}$ and $\dblDS{s}$ claimed in Theorem~\ref{thm:PERM}, except at $s=5$,
where we only get $\DS{5}(n) = O(\alpha(\alpha(n)))\cdot \DS{5}(n,3n)$ and $\dblDS{5}(n) = O(\alpha(\alpha(n))) \cdot \dblDS{5}(n,3n)$.
Refer to~\cite[\S 7.3]{Pettie-SoCG13} for an ad hoc method to eliminate this $\alpha(\alpha(n))$ factor.

\section{Generalized Constructions of Nonlinear Sequences}\label{sect:NMZC}

Recall from Section~\ref{sect:composition-and-shuffling} that the difference between postshuffling and preshuffling is in how blocks of one sequence
are merged with copies of another.  In $\Usub\postshuffle\Ubot$ symbols from $\Usub$ are inserted at 
the {\em end} of blocks in copies of $\Ubot$ whereas in $\Usub\preshuffle\Ubot$ they are inserted at the beginning of blocks.
It is not immediately clear why these two shuffling strategies should yield sequences with different properties.
Consider the projection of symbols $R = \{a,\ldots,z\}$ in a common block $B$ of $\Utop$, where all symbols in $R$ are middle occurrences in $B$.
If $\Utop$ was constructed via a series of composition and {\em post}shuffling operations, the
projection of $\Utop$ onto $R$, ignoring repetitions, is $ab\cdots z (zy\cdots a) zy\cdots a$,
whereas if {\em pre}shuffling were used the projection onto $R$ would be $ab\cdots z (ab\cdots z) zy\cdots a$.
In a subsequent composition event $\Usub = \Utop \compose \Umid$, the canonical ordering of $R$ in $\Umid(B)$
is  identical to their ordering in $\Utop$, in the case of preshuffling, or the reversal of that ordering in the case of postshuffling.

In this section we explore the complexity of sequences avoiding ``zig-zagging'' patterns, which can be viewed as one
natural generalization of Davenport-Schinzel sequences.  Recall the definitions of $N_k,M_k,$ and $Z_k$.
\begin{align*}
N_k &= 12\cdots (k+1)k\cdots 12\cdots (k+1)\\
M_k &= 12\cdots (k+1)k\cdots 12\cdots (k+1)k\cdots 1\\
Z_k &= 12\cdots (k+1)k\cdots 12\cdots (k+1)k\cdots 12\cdots (k+1)
\end{align*}

Note that $N_1=abab,M_1=ababa,$ and $Z_1=ababab$ generalize order-$2$, -$3$, and -$4$ Davenport-Schinzel sequences.
Klazar and Valtr~\cite{KV94} and Pettie~\cite{Pettie-SoCG11} proved
that $\Ex(N_k,n) = \Theta(\DS{2}(n))=\Theta(n)$ and that for any $k\ge 1$,
$\Ex(\{M_k,ababab\},n) = \Theta(\DS{3}(n)) = \Theta(n\alpha(n))$.
(That is, avoiding {\em both} $M_k$ and $ababab$ are equivalent to just avoiding $M_1$.)
One might guess that zig-zagging patterns, in general, mimic the behavior of the corresponding 
order-$s$ DS sequences.

We prove two results that, taken together, are rather surprising. Theorems~\ref{thm:M2k} and \ref{thm:Z3k} state the following in a more precise fashion.
\begin{itemize}
\item[(1)] For all $t$, there exists a $k$ such that $\Ex(M_k,n) = \Omega(n\alpha^t(n))$.
\item[(2)] For all $t$, there exists a $k$ such that $\Ex(Z_k,n) = \Omega(n2^{(1+o(1))\alpha^t(n)/t!})$.
\end{itemize}

\paragraph{Overview.}
We define two classes of non-linear sequences.  Class I sequences have lengths $\Theta(n\alpha^t(n))$
and Class II sequences have length $n2^{(1+o(1))\alpha^t(n)/t!}$, for any $t\ge 1$.  
Both Class I and Class II sequences are parameterized by a binary 
{\em pattern} $\pattern = \pattern_1 \pattern_2 \cdots \pattern_{|\pattern|} \in \{\ascend, \descend\}^*$.
The diagonals in $\pattern$ have the following interpretation.
Consider any set $\{a_1,\ldots,a_l\}$ of symbols in a sequence $T_\pattern$ of type $\pattern$.
A {\em maximally} intertwined configuration is one in which each pair of symbols in $\{a_1,\ldots,a_l\}$
alternate the maximum number of times.  In $T_\pattern$ all maximally intertwined configurations will take
the form $A^{\pattern_1}A^{\pattern_2}\cdots A^{\pattern_{|\pattern|}}$, where 
$A^{\ascendsm} = a_1\cdots a_l$ and $A^{\descendsm} = a_l\cdots a_1$.
Class I and II sequences are defined in Sections~\ref{sect:classI} and \ref{sect:classII} and their
forbidden sequences analyzed in Section~\ref{sect:classI-II-analysis}.

\subsection{Class I Sequences}\label{sect:classI}

The sequence $T_{\pattern}(i,j)$ consists of a mixture of live and dead blocks.
It is parameterized by a pattern $\pattern$, which always begins with $\ascend$.
The base cases for $T_\pattern$ are given below.
(Recall that live blocks are indicated with parentheses and dead blocks with angular brackets.)

\begin{align*}
T_{\ascendsm \descendsm}(i,j) &= (12\cdots j) \, \angbrack{j\cdots 21}		& \mbox{ one live block, one dead, for any $i$}\\
T_{\ascendsm \ascendsm}(i,j) &= (12\cdots j) \, \angbrack{12\cdots j}		& \mbox{ one live block, one dead, for any $i$}\\
T_{\pattern}(1,j) &= \left\{\begin{array}{l}
(12\cdots j) \, \angbrack{j\cdots 21}	\\
(12\cdots j) \, \angbrack{12\cdots j}
\end{array}\right.
&
\begin{array}{l}
\mbox{if $\pattern_{|\pattern|} = \descend$ and $|\pattern|>2$}\\
\mbox{if $\pattern_{|\pattern|} = \ascend$ and $|\pattern|>2$}
\end{array}\\
T_{\pattern}(i,0) &= (\, )^{2}							& \mbox{ two empty live blocks, any $\pattern$}
\end{align*}

Note that $T_{\pattern}(1,j)$ is identical to either $T_{\ascendsm \descendsm}(\cdot, j)$
or $T_{\ascendsm \ascendsm}(\cdot, j)$, depending on the last character of $\pattern$.
For the inductive case, when $i>1,j>0,$ and $|\pattern|>2$,

\begin{align*}
T_{\pattern}(i,j) &=
\left\{
\begin{array}{l}
\Tsub \preshuffle \Tbot	= (\Ttop\compose\Tmid)\preshuffle\Tbot	\\
\ \\
\Tsub \postshuffle \Tbot	= (\Ttop\compose\Tmid)\postshuffle\Tbot	
\end{array}
\right.
&
\begin{array}{l}
\mbox{ if $\pattern_{|\pattern|} = \descend$ }\\
\ \\
\mbox{ if $\pattern_{|\pattern|} = \ascend$ }
\end{array}\\
\ \\
\mbox{ where } \; \; \Tbot &= T_{\pattern}(i,j-1)\\
\Tmid &= T_{\pattern^-}\paren{i,\livebl{\Tbot}}	 & \pattern^- = \pattern_1\cdots\pattern_{|\pattern|-1}	\\
\Ttop &= T_{\pattern}(i-1,  \|\Tmid\|) 	& \\
\end{align*}

The following facts can easily be proved about $T_{\pattern}(i,j)$ by induction.
\begin{enumerate}
\item The first occurrence of every symbol appears in a live block 
and live blocks consist solely of first occurrences.
\item All live blocks have length exactly $j$.  The length of dead blocks varies,
as does the number of dead blocks between consecutive live blocks.
\item Each symbol occurs with the same multiplicity, $\nu_{\pattern,i}$, defined below.  
Hence $|T| = \nu_{\pattern,i} \|T\| = \nu_{\pattern,i}\cdot j \cdot \livebl{T}$.
\end{enumerate}

The construction of $T_{\pattern}$ gives us an inductive expression for the multiplicity $\nu_{\pattern,i}$ of symbols
in $T_{\pattern}(i,j)$.

\begin{align*}
\nu_{\pattern,i} &= 2					& \mbox{ for $|\pattern|=2$ and all $i$}\\
\nu_{\pattern,1} &= 2				& \mbox{ for all $\pattern$}\\
\nu_{\pattern,i} &= \nu_{\pattern,i-1} + \nu_{\pattern^-,i} - 1	& \mbox{ where $\pattern^- = \pattern_1\cdots\pattern_{|\pattern|-1}$}\\
\intertext{A short proof by induction shows that $\nu_{\pattern,i}$ has the closed form}
\nu_{\pattern,i} &= {i+|\pattern| - 3 \choose |\pattern|-2} + 1 		& \mbox{ for all $i\ge 1, |\pattern| \ge 2$}\\
\end{align*}

It can be shown that $i = \alpha(n,m) + O(1)$, 
where $n = \|T_{\pattern}(i,j)\|$ and $m=\bl{T_{\pattern}(i,j)})$,
from which it follows that 
$T_{\pattern}(i,j)$ has length $\Theta(n\alpha^{|\pattern|-2}(n,m))$, and length $\Theta(n\alpha^{|\pattern|-2}(n))$ if $j=O(1)$.
Theorem~\ref{thm:ada-aad} summarizes two results from~\cite{Pettie-DS-nonlin11,Pettie-GenDS11,HS86} using the $T_{\pattern}$ notation.

\begin{theorem}\label{thm:ada-aad} (\cite{HS86,Pettie-DS-nonlin11,Pettie-GenDS11})
\begin{enumerate}
\item $ababa,abcaccbc\nsubseq T_{\ascendsm\descendsm\ascendsm}$.
\item $abaaba,abcacbc\nsubseq T_{\ascendsm\ascendsm\descendsm}$.
\end{enumerate}
As a consequence both $\Ex(ababa,n)$ and $\Ex(abcacbc,n)$ are $\Omega(n\alpha(n))$, which is asymptotically tight.
\end{theorem}

\subsection{Class II Sequences}\label{sect:classII}

Class II Sequences consist {\em solely} of live blocks.  They are parameterized by binary patterns, which are
restricted to being even-length palindromes, starting with $\ascend$ and ending with $\descend$.
If $\pattern = \pattern_1\cdots\pattern_{|\pattern|}$, its {\em flip} $\flip{\pattern}$ 
is obtained by flipping the direction of each diagonal and its {\em truncation} $\pattern^{-}$
is obtained by trimming $\pattern_1$ and $\pattern_{|\pattern|}$.
For example, if $\pattern = \ascend \descend\descend \descend \ascend \ascend\ascend \descend$, 
$\flip{\pattern^-} = \ascend\ascend\ascend\descend\descend\descend$.

The base cases for $U_\pattern$ are given below.  The sequence $U_{\pattern}(i,j)$ has the property
that each block has length $j$ and each symbol has multiplicity $\mu_{\pattern,i}$, 
which will be defined below.
\begin{align*}
U_{\ascendsm \descendsm}(i,j) &= (12\cdots j) \, (j\cdots 21)		& \mbox{ two blocks, for any $i$}\\
U_{\pattern}(1,j) &= (12\cdots j) \, (j\cdots 21)				& \mbox{ two blocks, for any $\pattern$}\\
U_{\pattern}(0,j) &= (12\cdots j)							& \mbox{ one block, for any $\pattern$}\\
U_{\pattern}(i,1) &= (1)^{\mu_{\pattern,i}}					& \mbox{ $\mu_{\pattern,i}$ identical blocks}
\end{align*}
For the inductive case, when $i>1,j>0,$ and $|\pattern|>2$, we have
\begin{align*}
U_{\pattern}(i,j) &=
\left\{
\begin{array}{ll}
\zero{\Usub \preshuffle \Ubot = (\Utop \compose \Tmid) \preshuffle \Ubot}\hcm[5]		&\istrut[4]{0}\hcm \mbox{ if $\pattern_{2}\pattern_{|\pattern|-1} = \ascend\descend$}\\
\zero{\Usub \postshuffle \Ubot = (\Utop \compose \Tmid) \postshuffle \Ubot}\hcm[5]		&\hcm \mbox{ if $\pattern_{2}\pattern_{|\pattern|-1} = \descend\ascend$}\\
\end{array}
\right.\ \\
\mbox{ where } \; \; \Ubot &= \istrut{4.5}U_{\pattern}(i,j-1)\\
\istrut{8}\Umid &= \left\{\begin{array}{ll}
\zero{U_{\pattern^-}\paren{i,\bl{\Tbot}}}\hcm[5]	 &\istrut[4]{0}\hcm \mbox{ if $\pattern_{2}\pattern_{|\pattern|-1} = \ascend\descend$}\\
\zero{U_{\flip{\pattern^-}}\paren{i,\bl{\Tbot}}}\hcm[5] &\hcm \mbox{ if $\pattern_{2}\pattern_{|\pattern|-1} = \descend\ascend$}
\end{array}\right.\\
\Utop &= \istrut{4.5}U_{\pattern}(i-1,  \|\Tmid\|) 	& 
\end{align*}

The construction of $U_{\pattern}$ is a strict generalization of the $U_s$ sequences defined in Section~\ref{sect:lbPERM}, for even $s$.
Note that when $\pattern = (\ascend\descend)^{s/2}$, only postshuffling is used, since $\flip{\pattern^-} = (\ascend\descend)^{s/2-1}$.
The multiplicity $\mu_{\pattern,i}$ of symbols in $U_{\pattern}(i,j)$ is not affected by which shuffling operation is used, so the analysis
from Section~\ref{sect:lbPERM} still holds: $\mu_{\pattern,i} = 2^{i+t-1\choose t} \ge 2^{i^t/t!}$, where $t=(|\pattern|-2)/2$,
and $i = \alpha(\|U_{\pattern}(i,j)\|, \bl{U_{\pattern}(i,j)}) + O(1)$.

\subsection{Analysis of $T_\pattern$ and $U_\pattern$}\label{sect:classI-II-analysis}

Lemmas~\ref{lem:uni} and \ref{lem:abaaba} isolate some properties of $T_\pattern$ useful in the
analysis of $M$-shaped sequences and comb-shaped sequences.

\begin{lemma}\label{lem:uni}
Let $\Tsh = T_{\pattern}(i,j)$, where $i$ and $j$ are arbitrary.
Let $\chi = \pattern_{|\pattern|}$ and $\chi' = \pattern_{|\pattern|-1}$ be the last and second to last
characters of $\pattern$, and
let $\Ttop,\Tmid,\Tsub,$ and $\Tbot$ be the sequences arising in the formation of $\Tsh$.
\begin{enumerate}
\item If $abba \subseq \Tsh$ or $baba \subseq \Tsh$ then it cannot be that $b\in \Sigma(\Tsub)$ while $a\in \Sigma(\Tbot^*)$.\label{item:abba}
\item If $a<b$ share a live block in one of $\Ttop,\Tbot,$ or $\Tsh$, then this sequence's projection onto $\{a,b\}$ has the form \label{item:chi}
$(ab)a^* b^*$ if $\chi = \ascend$
and 
$(ab)b^* a^*$ if $\chi = \descend$.
\item If $a_1 < \cdots < a_l$ share a live block in $\Tsub$, then its projection onto $\{a_1,\ldots,a_l\}$ has the form \label{item:chichi}
$(a_1\ldots a_l) A^{\chi'} A^{\chi}$ where $A^{_{\ascend}} = a_1^* \ldots a_l^*$ and $A^{_{\descend}} = a_l^* \cdots a_1^*$.
\end{enumerate}
\end{lemma}

\begin{lemma}\label{lem:abaaba}
Whereas $ababa\nsubseq T_{\ascend\descend\ascend}$,
$abaaba \nsubseq T_{\pattern}$, for any $\pattern\in\{\ascend\ascend\descend,\ascend\descend\descend,\ascend\ascend\ascend\}$.
\end{lemma}

\begin{proof}
Lemma~\ref{lem:uni}(\ref{item:abba}) implies that $ababa$ cannot be introduced by a shuffling event,
but must first appear in $\Tsub = \Ttop\compose\Tmid$ from a composition event.
Moreover, $abaaba$ could not arise in $\Tsub$ from an occurrence of $ababa$ in 
$\Ttop$ since, in such an occurrence, the middle $a$ would necessarily be in a dead block and could therefore not beget multiple $a$s in $\Tsub$.
It must be that $a$ and $b$ share a common live block in $\Ttop$, so its projection onto $\{a,b\}$ 
is contained in $(ab) a^*b^*$, if $\pattern_3 = \ascend$, and $(ab) b^* a^*$ if $\pattern_3=\descend$.
Since $\Tmid$ is either $T_{\ascend\ascend}$ or $T_{\ascend\descend}$, \
the projection of $\Tsub$ onto $\{a,b\}$ is one of
\[
(ab)\, \angbrack{ba}\, a^*b^* 
\;\;\;\mbox{ or }\;\;\;
(ab)\, \angbrack{ba}\, b^* a^*
\;\;\;\mbox{ or }\;\;\;
(ab)\, \angbrack{ab}\, a^*b^*
\;\;\;\mbox{ or }\;\;\;
(ab)\, \angbrack{ab}\, b^*a^*.
\]
The first is $ababa$-free while the remaining are $abaaba$-free.
\end{proof}

In Theorem~\ref{thm:M2k} we prove that $\Ex(M_{2^k},n) = \Omega(n\alpha^{k+1}(n))$ by induction.
Lemma~\ref{lem:M2} handles the base case for $M_2$.

\begin{lemma}\label{lem:M2}
$M_2 = abcbabcba \nsubseq T_\pattern$, for any of the length-4 patterns $\pattern \in \ascend \{\ascend,\descend\}^2 \ascend$.
\end{lemma}

\begin{proof}
Since $M_2$ contains $ababa$, any instance of $M_2$ must first arise in $\Tsub=\Ttop\compose\Tmid$ from a composition event, not in 
$\Tsh=\Tsub\postshuffle\Tmid$
from a shuffling event.  Here $\Tmid$ is defined by any of the four patterns $\pattern^- \in \ascend \{\ascend,\descend\}^2$.
It must be that $a,b,c$ share a live block in $\Ttop$.  If only $b$ and $c$ shared a live block then the projection of $\Ttop$
onto $\{a,b,c\}$ would have the form $a^* (bc \mbox{ or } cb) a^* b^* c^* b^* a^*$, violating Lemma~\ref{lem:uni} since neither
$(bc)$ and $(cb)$ can be followed by $bcb$.  
If only $a$ and $b$ shared a live block the projection onto $\{a,b,c\}$
would have the form $a^* b^* c^* (ba \mbox{ or } ab) c^* b^* a^*$, which violates the property that live blocks contain only first occurrences.

We have deduced that $a,b,$ and $c$ share a live block $B$ in $\Ttop$, but they do not necessarily appear in that order.
To form a copy of $M_2$, some prefix must arise from substituting the type $\pattern^-$ sequence $\Tmid(B)$
for $B$; the remaining suffix must follow $a,b,$ and $c$'s live block in $\Ttop$.  The split between prefix and suffix 
can be (i) $abcbab\, |\, cba$, or (ii) $abcbabc\, |\, ba$, or (iii) $abcbabcb\, |\, a$.  
In cases (i) and (ii), $b$ must precede $a$ in $B$, meaning $b<a$ in the canonical ordering of $\Tmid(B)$.  
As a consequence, any occurrence of the prefix $abcbab$ (or $abcbabc$) in $\Tmid$ implies an occurrence
of $babbab\subseq \Tmid$, contradicting Lemma~\ref{lem:abaaba}.  In case (iii) the prefix contains $bcbbcb$, also contradicting Lemma~\ref{lem:abaaba}.
\end{proof}

\begin{theorem}\label{thm:M2k}
For any $k\ge 1$, $M_{2^k} \nsubseq T_{\pattern}$, where $\pattern \in \ascend\{\ascend,\descend\}^2\ascend^k$.
As a consequence, $\Ex(M_{2^k},n) = \Omega(n\alpha^{k+1}(n))$.
\end{theorem}

\begin{proof}
The proof is by induction on $k$; the base case is covered by Lemma~\ref{lem:M2}.  For succinctness let $K=2^k$.    As in the proof of 
Lemma~\ref{lem:M2} we can restrict our attention to the case where $M_{K}$, say over the alphabet $a_1,\ldots,a_{K+1}$, 
arises in $\Tsub$ after a composition event.  Moreover, we can assume $a_1,\ldots,a_{K+1}$ appear in a common live block $B$,
so the projection of $\Ttop$ onto $\{a_1,\ldots,a_{K+1}\}$ is $(a_1\cdots a_{K+1}) a_1^* \cdots a_{K+1}^*$.
If substituting $\Tmid(B)$ for $B$ creates an instance of $M_{K}$, some prefix must come from $\Tmid(B)$ and the remaining suffix
from the sequence $a_1^*\cdots a_{K+1}^*$ following $B$.  There are two cases: either the suffix contains a strict majority of the $K+1$ symbols
or a strict minority.  In the former case we have $a_{K/2+1} < \cdots < a_{K+1}$ according to the canonical ordering of $\Tmid(B)$,
so any instance of the $N$-shaped pattern 
$a_{K+1} a_{K} \cdots a_{K/2+1} a_{K/2+2} \cdots a_{K+1} a_K \cdots a_{K/2+1}$ in $\Tmid(B)$
implies that it also contains
\[
\istrut{5}M_{K/2} = a_{K/2+1} \rb{1}{ $a_{K/2+2}$} \ldotsup{2}  \rb{3}{ $a_{K+1}$} \rb{2}{ $a_{K}$} \ldotsdown{1} a_{K/2+1} \rb{1}{ $a_{K/2+2}$} \ldotsup{2} \rb{3}{ $a_{K+1}$} \rb{2}{ $a_K$} \ldotsdown{1} \,a_{K/2+1},
\]
which contradicts the hypothesis that $\Tmid$ is $M_{K/2}$-free.  If, on the other hand, the suffix of $M_K$ following $B$ contains
a strict minority of $\{a_1,\ldots,a_{K+1}\}$, then $\Tmid(B)$ must contain an instance of $M_{K/2}$ on the alphabet $a_1,\ldots,a_{K/2 + 1}$,
also contradicting the inductive hypothesis.
\end{proof}

We now turn to the analysis of the forbidden sequences of $U_\pattern$.

\begin{theorem}\label{thm:Z3k}
For any $k\ge 0$, $Z_{3^{k}} \nsubseq U_{\pattern}$, where $\pattern = \ascend^{k+1}\,\descend\ascend\,\descend^{k+1}$.
As a consequence, $\Ex(Z_{3^{k}},n) > n\cdot 2^{(1+o(1))\alpha^{k+1}(n)/(k+1)!}$.
\end{theorem}

\begin{proof}
The proof is by induction on $k$.  For succinctness we let $K=3^{k}$.  In the base case $k=0$, $Z_{K} = ababab$,
and $U_{\pattern} = U_{\ascendsm\descendsm\ascendsm\descendsm}$ is $ababab$-free, by Lemma~\ref{lem:perm-U}.
In the general case $k\ge 1$ and $\pattern = \ascend^{k+1}\,\descend\ascend\,\descend^{k+1}$, so 
$U_{\pattern} = \Usub \preshuffle \Ubot = (\Utop \compose \Umid) \preshuffle \Ubot$
is formed by composing $\Utop$ with $\Umid$, a type $\pattern^-$ sequence, then {\em pre}shuffling
it with $\Ubot$.
We can assume that any occurrence of $Z_{K}$ arises from the composition event
$\Usub = \Utop \compose \Umid$ since $ababab\subseq Z_{K}$ cannot be introduced by shuffling.
Write $Z_{K}$ as 
\[
\istrut{5}
a_1 \rb{1}{ $a_2$} \ldotsup{2} \rb{3}{ $a_{K+1}$} \rb{2}{ $a_{K}$} \ldotsdown{1} a_1 \rb{1}{ $a_2$} \ldotsup{2} \rb{3}{ $a_{K+1}$} \rb{2}{ $a_{K}$} \ldotsdown{1} a_1 \rb{1}{ $a_2$} \ldotsup{2} \rb{3}{ $a_{K+1}.$}
\]
It is easy to verify that if $Z_{K}$ occurs in $\Usub$, it must be that $\{a_1,\ldots,a_{K+1}\}$ share a {\em single} block $B$
in $\Utop$.  (Note, however, that their canonical orderings in $\Utop$ and $\Umid(B)$ are {\em not} necessarily $a_1 < \cdots < a_{K+1}$.)
Some prefix of $Z_{K}$ appears before $B$ in $\Utop$, some suffix of $Z_{K}$ after $B$ in $\Utop$, and the remaining middle
portion appears in $\Umid(B)$.  Suppose $a_1\cdots a_{l}$ is the prefix and $a_{l'}a_{l'+1}\cdots a_{K+1}$ the suffix, for some indices $l,l'$.
It follows that $a_1 <  a_2 < \cdots < a_l$ and $a_{K+1} < a_{K} < \cdots < a_{l'}$ according to the canonical ordering of $\Umid(B)$,
which implies $l\le l'$.\footnote{Since preshuffling is used, the canonical ordering of {\em middle} symbols in $B$ is the same in $\Utop$ and $\Umid(B)$, though the same is not true of symbols making their first appearance in $B$.}
At least one of the following must be true 
\begin{enumerate}
\item[(i)] the prefix contains at least $K/3+1$ symbols and is disjoint from the suffix, that is, $l \ge K/3+1$ and $l < l'$.
\item[(ii)] the suffix contains at least $K/3+1$ symbols and we are not in case (i), that is, $l' \le 2K/3 +1$.
\item[(iii)] there are at least $K/3+1$ symbols in neither the prefix nor suffix, that is, $l\le K/3$ and $l' \ge 2K/3 + 2$.
\end{enumerate}
Case (iii) is the simplest.  To form a copy of $Z_{K}$ in $\Usub$, we would need $\Umid(B)$ to contain a copy of $Z_{K/3}$ on the alphabet
$\{a_{K/3+1},\ldots,a_{2K/3 + 1}\}$, contradicting the inductive hypothesis.
In Case (i), $\Umid(B)$ must contain $a_{K/3+1}\cdots a_{1}\cdots a_{K/3+1}\cdots a_1\cdots a_{K/3+1}$, which, by the canonical
ordering $a_1<\cdots< a_{K/3+1}$, implies $\Umid(B)$ also contains a copy of $Z_{K/3}$, a contradiction.
Case (ii) is symmetric to Case (i).
\end{proof}

\subsection{Comb-shaped Sequences}\label{sect:comb-shaped}

The results of~\cite{HS86,KV94,Pettie-GenDS11,Pettie-SoCG11} show that $ababa$ and $abcacbc$ are the only 
minimally non-linear 2-sparse forbidden sequences over a three-letter alphabet, both with extremal function $\Theta(n\alpha(n))$.
Just as $ababa$ can be generalized to $M$-shaped sequences, $C_1 = abcacbc$ can be generalized to 
the one-sided {\em comb-shaped}
sequences $\{C_k\}_{k\ge 1}$, where
\[
\istrut{10}
C_k =
1\,\rb{1}{ $2$}\,\rb{2}{ $3$} \ldotsUP{4} \;\rb{7}{ $(k+2)$}\mbox{ $1$}
\rb{7}{ $(k+2)$}\rb{1}{ $2$}\rb{7}{ $(k+2)$}\rb{2}{ $3$}\rb{7}{ $(k+2)$} \ldotsup{4}\,\rb{5}{ $(k+1)$}\rb{7}{ $(k+2).$}
\]

Our parameterized sequences let us obtain non-trivial lower bounds on comb-shaped sequences. 

\begin{theorem}\label{lem:comb}
For all $k\ge 1$, $C_k\nsubseq T_{\pattern}$, where $\pattern = \ascend\ascend\descend^{k}$.
Consequently, $\Ex(C_k,n) = \Omega(n\alpha^k(n))$.  
\end{theorem}

\begin{proof}
The proof is by induction on $k$.  Theorem~\ref{thm:ada-aad} (see~\cite{Pettie-GenDS11}) takes care of the base case $C_1 = abcacbc$.
We will focus on $C_2 = abcdadbdcd$, then note why the argument works for any $k$.
Define $\Ttop,\Tsub,\Tbot,$ and $\Tmid$ as usual, where $\Tmid$ is now a type $\ascend\ascend\descend$ sequence.
We first argue that $\{a,b,c,d\}\subseteq \Sigma(\Ttop)$.  
One may check that the only case that does not immediately violate Lemma~\ref{lem:uni}(\ref{item:abba}) 
is that $a\in \Sigma(\Tbot^*)$ while $b,c,d \in \Sigma(\Ttop)$.  This means that $(bcd) dbdcd \subseq \Tsub$, where the live block 
$(bcd)$ was shuffled into $a$'s copy of $\Tbot$.  However, Lemma~\ref{lem:uni}(\ref{item:chichi}) implies that the projection
of $\Tsub$ onto $\{b,c,d\}$ has the form $(bcd) d^* c^* b^* d^* c^* b^*$, which does not contain $(bcd) dbdcd$.

One can see that $a,b,c,$ and $d$ must share a live block $B$ in $\Ttop$.  If the first two $a$s in $C_2\subseq \Tsub$ arose
from the composition that created $\Tsub$ then $b,c,$ and $d$ must have been in $a$'s live block.  If not then $C_2$ 
would have already appeared in $\Ttop$.  Thus, some prefix of $C_2$ arose from substituting $\Tmid(B)$ for $B$
and the remaining suffix followed $B$ in $\Ttop$.  Lemma~\ref{lem:uni}(\ref{item:chi}) implies that the suffix cannot be $dcd$ for
otherwise $(cd) cd\subseq \Ttop$ or $(dc) dc\subseq \Ttop$.  This 
implies that $abd ad bd = C_1 \subseq \Tmid(B)$ (a type $\ascend\ascend\descend$ sequence), 
which contradicts Theorem~\ref{thm:ada-aad}.

For $k>2$ write $C_k = a_1a_2\cdots a_{k+1}ba_1ba_2b\cdots ba_{k+1}b$.  The same argument from above 
shows that $\{a_1,\ldots,a_{k+1},b\}$ are contained in a single block $B$ of $\Ttop$.  For 
$C_k$ to arise in $\Tsub$ a prefix of it must come from $\Tmid(B)$ and a suffix from the part of $\Ttop$ following $B$.
By Lemma~\ref{lem:uni}(\ref{item:chi}) the suffix cannot be $ba_{k+1}b$, which means the prefix in $\Tmid(B)$ must 
contain $a_1\cdots a_{k} b a_1ba_2b\cdots ba_kb = C_{k-1}$, contradicting the inductive hypothesis.
\end{proof}

\section{Conclusions}\label{sect:conclusions}

In Theorem~\ref{thm:PERM} we  established sharp bounds on the functions $\PERM{r,s}$ and $\dblPERM{r,s}$, for all values of $r$ and $s$, 
and showed, perhaps surprisingly, that these extremal functions are essentially the same. 
Moreover, they match $\DS{s}$ and $\dblDS{s}$ only when $s\le 3$, or $s\ge 4$ is even, or $r=2$.
However, Theorem~\ref{thm:PERM} is {\em not} the last word on $\dblPERM{r,s}$.
In Cibulka and \Kyncl's~\cite{CibulkaK12} application of $\dblPERM{r,s}(n,m)$, $s$ is a fixed parameter
whereas $r$ is variable and cannot be bounded as a function of $s$.  Cibulka and \Kyncl{} require upper bounds on $\dblPERM{r,s}(n,m)$ 
that are linear in $r$ whereas the leading constant in our bounds matches that of 
$\dblPERM{r,2}(n,m)$, currently known to be at most $O(6^r)$.  See Lemma~\ref{lem:orders12}.
In other words, we now have two incomparable upper bounds on 
$\dblPERM{r,2}(n,m)$ when $r$ is not treated as a constant, namely $O((n+rm)\alpha(n,m))$~\cite{CibulkaK12}, which is optimal as a function of $r$, 
and $O(6^r(n+m))$, which is optimal for fixed $r$.
Whether $\dblPERM{r,2}(n,m) = O(n + rm)$ or not is an intriguing open question.

We have shown that {\em doubling} various forbidden patterns (alternating sequences and catenated permutations) 
has no significant effect on their extremal functions.  
It is an open problem whether
$\Ex(\dbl(\sigma),n)$ is asymptotically equivalent to $\Ex(\sigma,n)$ for {\em every} $\sigma$.
We conjecture the answer is no 
when $\sigma$ can be a {\em set} of forbidden sequences, though it seems plausible
the answer is yes for any single forbidden sequence.

\begin{conjecture}
In general, it is not true that $\Ex(\dbl(\sigma),n) = \Theta(\Ex(\sigma,n))$.  In particular, 
whereas $\Ex(\dbl(\{ababa,abcacbc\}),n) = \Theta(n\alpha(n))$, 
we conjecture $\Ex(\{ababa,abcacbc\},n)=O(n)$.
\end{conjecture}

The main open problem in the realm of generalized Davenport-Schinzel sequences
is to characterize linear forbidden sequences, or equivalently, to enumerate all minimally non-linear
forbidden sequences.  The number of minimally non-linear sequences (with respect to the partial order $\subseq$)
is almost certainly infinite~\cite{Pettie-GenDS11}, but whether there are infinitely many {\em genuinely} different
non-linear sequences is open.  Refer to~\cite{Pettie-GenDS11} for a discussion of 
how ``genuinely'' might be formally defined.

\begin{conjecture}
(Informal) Every nonlinear sequence $\sigma$ (having $\Ex(\sigma,n)=\omega(n)$) contains $ababa$, $abcacbc$,
or some sequence morally equivalent to $abcacbc$.
\end{conjecture}

Our lower bounds on $\Ex(M_k,n)$ are weak, as a function of $k$, and we have provided no non-trivial upper bounds.
It may be possible to generalize the proof of Theorem~\ref{thm:dblPERMfour} to show 
$\Ex(M_k,n) = O(n\poly(\alpha(n)))$, where the degree of the polynomial depends on $k$.

\ignore{
}

\appendix

\section{Proofs}\label{appendix:proofs}

\subsection{Proof of Lemma~\ref{lem:dblPermequiv}}

Recall that $\dbl(\Perm{r,s+1}) = \{\dbl(\sigma) \;|\; \sigma\in\Perm{r,s+1}\}$
whereas sequences in $\dblPerm{r,s+1}$ are formed by taking the concatenation of $s+1$ sequences, the first and last
being a permutation of $\{1,\ldots,r\}$ and all the rest containing two occurrences of $\{1,\ldots,r\}$.  For example, $abc\, abaccb\, bca \in \dblPerm{3,3}$
whereas $abbcc\, ccbbaa\, bbcca\in \dbl(\Perm{3,3})$.  We restate Lemma~\ref{lem:dblPermequiv}.\\

\noindent{\em
{\bf Lemma~\ref{lem:dblPermequiv}}\
The following bounds hold for any $r\ge 2, s\ge 1$.
\begin{align*}
\Ex(\dbl(\Perm{r,s+1}),n,m) &\le r\cdot \dblPERM{r,s}(n,m) + 2rn\\
\Ex(\dbl(\Perm{r,s+1}),n) &= O(\dblPERM{r,s}(n)).
\end{align*}}

\begin{proof}
Let $S$ be a $\dbl(\Perm{r,s+1})$-free sequence over an $n$-letter alphabet.  Obtain $S'$ from $S$ by discarding the first occurrence and last
$r$ occurrences of each letter, then retaining every $r$th occurrence of each letter, discarding the rest.  Clearly $S'$
has the property that each $b$ is preceded and followed by at least $r$ $b$s in $S$, and between two $b$s in $S'$ there are at least $r-1$ $b$s in $S$.
It follows that $|S'|\ge (|S|-2rn)/r$.
Suppose $|S'|$ contained some sequence $\sigma_1'\cdots \sigma_{s+1}'\in\dblPerm{r,s+1}$.  
(Recall that $\sigma_1'$ and $\sigma_{s+1}'$ contain one copy of $\{1,\ldots,r\}$ whereas $\sigma_2',\ldots,\sigma_s'$ contain two copies of $\{1,\ldots,r\}$.)
This implies that $S$ contains a sequence $\sigma_1\cdots \sigma_{s+1}$ where each $\sigma_k$ contains $r+1$ copies of $\{1,\ldots,r\}$.
We claim each $\sigma_k$ contains a doubled permutation of $\{1,\ldots,r\}$, which implies that $S$ is not $\dbl(\Perm{r,s+1})$-free, a contradiction.
Find the symbol $b$ in $\sigma_k$ whose second occurrence
is earliest, that is, we can write $\sigma_k = \sigma_k'\,b\,\sigma_k''\,b\,\sigma_k'''$, where $\sigma_k'\sigma_k''$ contains at most one copy
of each symbol.  Since $\sigma_k'''$ contains at least $r$ copies of the $r-1$ symbols in $\{1,\ldots,r\} \backslash\{b\}$ we can continue
to find a doubled permutation of $\{1,\ldots,r\}\backslash\{b\}$ by induction.
If $S$ is an $m$-block sequence then $S'$ is too, giving the first bound.  
When $S$ is merely $r$-sparse we can only bound $S'$ by $\dblPERM{r,s}(n)$ if it, too, is $r$-sparse.
This is done as follows.

Greedily partition $S=S_1S_2\ldots S_m$ into maximal sequences $\{S_q\}$ over alphabets of size exactly $2r^2$, with $\|S_m\|$ perhaps smaller.  
Since each $S_q$ has length at most $\Ex(\dbl(\Perm{r,s+1}),2r^2) = O(1)$, it follows that $m = \Omega(|S|)$.
Obtain $T$ be replacing each $S_q$ with a block consisting of its alphabet $\Sigma(S_q)$.
If $|T|\le 2r^2n$ there is nothing to prove since $|S|=\Theta(|T|) = O(n) = O(\dblPERM{r,s}(n))$, so assume otherwise.
Obtain $T'$ from $T$ by discarding the first occurrence and last $r$ occurrences of each letter, then retaining every $r$th occurrence of each letter.
It follows that $|T'| \ge (|T|-2rn)/r \ge |T|\f{r-1}{r^2}$, that is, the average length of blocks in $T'$ is $2(r-1)$. 
Let $T''$ be an $r$-sparse subsequence of $T'$ obtained by scanning $T'$ from left to right, removing a symbol if it is identical to one of the preceding $r-1$ symbols.  At most $r-1$ letters from each block of $T'$ can be removed in this process.  The average block length of $T''$ is at least $2(r-1)-(r-1) \ge 1$,
hence $|T''| \ge m = \Omega(|S|)$.  Since $T''$ is $\dblPerm{r,s+1}$-free, we have $|S| = O(\dblPERM{r,s}(n))$.
\end{proof}

\subsection{Proof of Lemma~\ref{lem:SparseVersusBlocked}}

There is no theorem to the effect that $\Ex(\sigma,n) = O(\Ex(\sigma,n,O(n)))$.
Lemma~\ref{lem:SparseVersusBlocked} restates the best known reductions from $r$-sparse to blocked sequences.
Some ad hoc reductions are known to be superior, for example, those for order-5 DS sequences~\cite{Pettie-SoCG13}.\\

\noindent{\em{\bf Lemma~\ref{lem:SparseVersusBlocked}}\
(Cf.~Sharir~\cite{Sharir87}, \Furedi{} and Hajnal~\cite{FurediH92}, and Pettie~\cite{Pettie-SoCG13}.)
Define $\gamma_s,\dblgamma_s,\gamma_{r,s},\dblgamma_{r,s} : \mathbb{N}\rightarrow\mathbb{N}$ 
to be non-decreasing functions bounding the leading factors of $\DS{s}(n),\dblDS{s}(n),\PERM{r,s}(n),$ and $\dblPERM{r,s}(n)$,
e.g., $\dblPERM{r,s} \le \dblgamma_{r,s}(n) \cdot n$.  The following bounds hold.
\begin{equation*}
\begin{array}{r@{\hcm[.1]}l@{\hcm[1]}r@{\hcm[.1]}l}
\DS{s}(n) &\le \gamma_{s-2}(n)\cdot \DS{s}(n,2n)
& \dblDS{s}(n) &\le (\dblgamma_{s-2}(n)+4)\cdot \dblDS{s}(n,2n)\istrut[2.5]{0}\\
\DS{s}(n) &\le \gamma_{s-2}(\gamma_{s}(n))\cdot \DS{s}(n,3n)
& \dblDS{s}(n) &\le (\dblgamma_{s-2}(\dblgamma_{s}(n))+4)\cdot \dblDS{s}(n,3n)\istrut[4]{0}\\
\PERM{r,s}(n) &\le \gamma_{r,s-2}(n)\cdot \PERM{r,s}(n,2n) + 2n
& \dblPERM{r,s}(n) &\le (\dblgamma_{r,s-2}(n)+O(1))\cdot \dblPERM{s}(n,2n))\istrut[2.5]{0}\\
\PERM{r,s}(n) &\le \gamma_{r,s-2}(\gamma_{r,s}(n))\cdot \PERM{r,s}(n,3n) + 2n
& \dblPERM{r,s}(n) &\le (\dblgamma_{r,s-2}(\dblgamma_{r,s}(n))+O(1))\cdot \dblPERM{s}(n,3n))\\
\end{array}
\end{equation*}
\noindent where the $O(1)$ term in the last two inequalities depends on $r$ and $s$.\\
}

\begin{proof}
All the bounds are obtained from the following sequence manipulations, which were first
used by Hart and Sharir~\cite{HS86} and Sharir~\cite{Sharir87}.  Let $S$ be an $r$-sparse sequence
avoiding some set $\sigma$ of subsequences over an $r$-letter alphabet, so $|S|\le \Ex(\sigma,n)$.  
Greedily parse $S$ into $m$ intervals $S_1S_2\cdots S_m$ by choosing $S_1$ to be the 
maximum-length prefix satisfying some property $\mathcal{P}$, $S_2$ to be the maximum-length prefix of the remaining
sequence satisfying $\mathcal{P}$, and so on.  Form $S' = \Sigma(S_1)\Sigma(S_2)\cdots\Sigma(S_m)$ 
by replacing each interval $S_i$ with a single block $\Sigma(S_i)$ containing its alphabet, listed in order of first appearance.
Since $S'$ is a subsequence of $S$, $|S'|\le \Ex(\sigma,n,m)$.  To bound $|S|$ we only need to determine upper bounds
on $m$ and the {\em shrinkage} factor $|S|/|S'|$.

\paragraph{Bounds on $\DS{s}$.} If we parse $S$ into maximal order-$(s-2)$ sequences
then each $S_i$ must contain either the first or last occurrence of some symbol, hence $m\le 2n$.  The shrinkage factor is
$|S_i|/\|S_i\| \le \gamma_{s-2}(\|S_i\|) \le \gamma_{s-2}(n)$, which gives the first inequality.  Now consider 
parsing $S$ into $m$ maximal sequences that are both order-$(s-2)$ DS sequences {\em and} have length at most $\gamma_s(n)$.
It follows that $m\le 3n$: at most $n$ sequences were terminated because they reached length $\gamma_s(n)$ (by definition of $\gamma_s$)
and the remaining sequences number at most $2n$ since each must contain the first or last occurrence of some letter.

\paragraph{Bounds on $\dblDS{s}$.} Let $\sigma_{s+2}$ be the alternating sequence with length $s+2$.  Order-$s$ double DS sequences are 
$\dbl(\sigma_{s+2})$-free.  Obtain $\sigma'_{s+2}$ by doubling each letter of $\sigma_{s+2}$, including the first and last.  
It is easy to show that $\Ex(\sigma'_{s+2},n) \le \dblDS{s}(n)+4n$ so we can take 
$\dblgamma_s(n)+4$ to be the leading factor in this extremal function.  Consider parsing an order-$s$ double DS sequence $S$.
If we parse $S$ into maximal $\sigma'_{s}$-free sequences then each subsequence must contain the first or last occurrence of some symbol, 
so $m\le 2n$ and the shrinkage factor is at most $\dblgamma_{s-2}(n)+4$.  If, further, we truncate any subsequence in the parsing at length $\dblgamma_s(n)$, then $m\le 3n$ and the shrinkage factor is at most $\dblgamma_{s-2}(\dblgamma_s(n))+4$.

\paragraph{Bounds on $\PERM{r,s}$ and $\dblPERM{r,s}$.}
The argument is the same, except that during the parsing step, we discard any symbol that triggers the termination of a subsequence.
For example, if $S$ is a $\Perm{r,s+1}$-free sequence we parse it into $S_1 a_1 S_2 a_2 \cdots a_{m-1} S_m a_m$, where the $\{S_i\}$
are maximal $\Perm{r,s-1}$-free sequences and $\{a_i\}$ the single letters following them, where $a_m$ might not be present.
Since $S_ia_i$ contains some element of $\Perm{r,s-1}$, $S_ia_i$ must contain the first or last occurrence of some letter, hence $m\le 2n$.
We form $S'$ by contracting each $S_i$ to a single block, discarding $a_i$, so the shrinkage factor is at most 
$\gamma_{r,s-2}(n)$.  It follows that $|S| \le \gamma_{r,s-2}(n)\cdot \PERM{r,s}(n,2n) + 2n$.
The procedure for $\dblPERM{r,s}$ is a straightforward combination of the procedures described above, for $\PERM{r,s}$ and $\dblDS{s}$.
\end{proof}

\subsection{Proof of Lemma~\ref{lem:DSPERM}}

We restate the lemma.\\

\noindent{\em{\bf Lemma~\ref{lem:DSPERM} }
The extremal functions for order-$s$ (double) Davenport-Schinzel sequences
and $\Perm{2,s+1}$-free ($\dblPerm{2,s+1}$-free) sequences are equivalent up to constant factors.
In particular,
\begin{equation*}
\begin{array}{rcl}
\DS{s}(n) \le& \PERM{2,s}(n) &< 3\cdot \DS{s}(n) + 2n\\
\DS{s}(n,m) \le& \PERM{2,s}(n,m) &< 2\cdot \DS{s}(n,m) + n\\
\dblDS{s}(n) \le& \dblPERM{2,s}(n) &< 5\cdot \dblDS{s}(n) + 4n\\
\dblDS{s}(n,m) \le& \dblPERM{2,s}(n,m) &< 3\cdot \dblDS{s}(n,m) + 2n
\end{array}
\end{equation*}
}

\begin{proof}
Order-$s$ DS sequences are $\Perm{2,s+1}$-free, which gives the 1st and 3rd inequalities.
Let $S$ be a 2-sparse $\Perm{2,s+1}$-free sequence.  Form $S'\subseq S$ by filtering $S$ as follows.
\begin{enumerate}
\item[(i)] Discard the 1st occurrence of each letter.
\item[(ii)] Discard up to $n$ additional occurrences to restore 2-sparseness.
\item[(iii)] Discard every even occurrence of each letter.
\item[(iv)] Discard additional occurrences to restore 2-sparseness.
\end{enumerate}
We claim $S'$ has length at least $(|S| - 2n)/3$.  The number of letters removed in steps (i) and (ii) 
is at most $2n$.  The number removed in step (iii) is at most $(|S|-2n)/2$ and the number removed
in step (iv) is at most 1/3 of that of step (iii).  This is because between any two even occurrences of some letter
$a$, there must be another $a$ and, due to 2-sparseness, at least two other letters.  Thus, each letter removed
in step (iv) corresponds to at least three removed in step (iii).
Suppose that $S'$ were not an order-$s$ DS sequence, that it contained an alternating
subsequence $a\cdots b \cdots a \cdots b \cdots$ with length $s+2$.
Together with the first occurrence of $b$ and the missing odd occurrences of $a$ and $b$ from $S$,
we can form a $\Perm{2,s+1}$ subsequence in $S$, a contradiction.  This gives the 2nd inequality.
If $S$ is composed of $m$ blocks then we only need to form $S'$ using steps (i) and (iii). 
The 4th inequality follows.  
The 5th and 7th inequalities follow since order-$s$ double DS sequences are $\dblPERM{2,s+1}$-free.
To obtain the 6th inequality, let $S$ be a 2-sparse $\dblPERM{2,s+1}$-free sequence and $S'$ be derived
as follows.
\begin{enumerate}
\item[(i)] Discard the first and last occurrence of each letter.
\item[(ii)] Discard up to $2n$ additional occurrences to restore 2-sparseness.
\item[(iii)] Retain every third occurrence of each letter; discard all others.
\item[(iv)] Discard additional occurrences to restore 2-sparseness.
\end{enumerate}
By the same argument as above, the number of letters discarded in step (iii) is at most 
$2(|S|-4n)/3$ and the number discarded in step (iv) at most 1/5th that of step (iii), hence
$|S'| \ge (|S|-4n)/5$.  Suppose $S'$ contained a doubled alternating sequence $abbaabb\cdots$ having $s+2$ runs of $a$s and $b$s.
This implies that $S$ contains $\underline{a}ab\underline{bb}ba\underline{aa}ab\underline{bb}b\cdots$,
where the underlined letters appear in $S$ but not $S'$, and therefore that $S$ contains an instance of $\dblPerm{2,s+1}$.
The 6th inequality follows. The 8th follows from the same argument, omitting steps (ii) and (iv) in the construction of $S'$.
\end{proof}

\subsection{Proof of Lemma~\ref{lem:orders12}}

Some of the results cited in Lemma~\ref{lem:orders12} refer to (or implicitly use) results
on forbidden 0-1 matrices.  
See \Furedi{} and Hajnal~\cite{FurediH92} and Pettie~\cite{Pettie-GenDS11,Pettie-FH11,Pettie-SoCG11}
for more details on the connection between matrices and sequences.\\

\noindent{\em{\bf Lemma~\ref{lem:orders12} }
At orders $s=1$ and $s=2$, 
the extremal functions 
$\DS{s},\dblDS{s},\PERM{r,s},$ 
and $\dblPERM{r,s}$ obey the following.
\begin{equation*}
\begin{array}{r@{\hcm[.1]}l@{\hcm[1]}r@{\hcm[.1]}l@{\hcm[1]}r}
\DS{1}(n) &= n		& \DS{1}(n,m) &= n+ m-1\\
\DS{2}(n) &= 2n-1	& \DS{2}(n,m) &= 2n+m-2		& \mbox{(Davenport-Schinzel~\cite{DS65})}\\
\dblDS{1}(n) &= 3n-2	& \dblDS{1}(n,m) &= 2n+m-2	& \mbox{(Dav.-Sch.~\cite{DS65b},Klazar~\cite{Klazar02})}\\
\dblDS{2}(n) &< 8n  & \dblDS{2}(n,m) &< 5n+m		& \mbox{(Klazar~\cite{Klazar96}, \Furedi-Hajnal~\cite{FurediH92})}\\
\PERM{r,1}(n) &= \dblPERM{r,1}(n) < rn & \PERM{r,1}(n,m) &= \zero{\dblPERM{r,1}(n,m) < n + (r-1)m} & \mbox{(Klazar~\cite{Klazar92})}\\
\PERM{r,2}(n) &< 2rn & \PERM{r,2}(n,m) &< 2n + (r-1)m & \mbox{(Klazar~\cite{Klazar92})}\\
\dblPERM{r,2}(n) &< 6^rrn    & \dblPERM{r,2}(n,m) &< 2\cdot 6^{r-1}(n + m/3) & \mbox{(Pettie~\cite{Pettie-SoCG11}, cf.~\cite{KV94})}
\end{array}
\end{equation*}
}

\begin{proof}
Davenport and Schinzel~\cite{DS65} noted the bounds on $\DS{1}(n)$ and $\DS{2}(n)$; their extension to blocked sequences
is trivial.  In an overlooked note Davenport and Schinzel~\cite{DS65} observed without proof that $\dblDS{1}(n)=3n-2$, which
was formally proved by Klazar~\cite{Klazar02}.  Its extension to blocked sequences is also trivial.
Adamec, Klazar, and Valtr~\cite{AKV92} proved that $\dblDS{2}(n) = O(n)$
and Klazar~\cite{Klazar96}
bounded the leading constant between 7 and 8. 
A blocked sequence $S$ can be represented as a 0-1 incidence matrix $A_S$ whose rows correspond to symbols and columns to blocks,
where $A_S(i,j)=1$ if and only if symbol $i$ appears in block $j$.
A forbidden sequence becomes a forbidden 0-1 pattern.  The bound on $\dblDS{2}(n,m)$ follows from \Furedi{} and Hajnal's~\cite{FurediH92}
analysis of a certain 0-1 pattern.
The bounds on $\PERM{r,1}$ and $\PERM{r,2}$ were noted by Klazar~\cite{Klazar92} and Nivasch~\cite{Nivasch10}.  
They are straightforward to prove.

Since the $N$-shaped sequence $12\cdots r\: r(r-1)\cdots 1\: 12\cdots r$ over $r$ letters is contained in $\Perm{r,3}$, 
the linear upper bound on $\Ex(\dbl(12\cdots r\, r(r-1)\cdots 1\, 12\cdots r),n)$ due to Klazar and Valtr~\cite{KV94} (see also \cite{Pettie-SoCG11})
immediately extend to $\dblPERM{r,2}(n)$.  With some care the leading constants of $\dblPERM{r,2}(n)$ and $\dblPERM{r,2}(n,m)$
can be made reasonably small using the 0-1 matrix representation of (forbidden) sequences from~\cite{Pettie-SoCG11}.  
Consider an $m$-block, $\dblPerm{r,3}$-free sequence $S$.  
Without loss of generality assume the alphabet $\Sigma(S) = \{1,\ldots,n\}$ is ordered according to their first appearance in $S$.
Let $A_S$ be an $n\times m$ 0-1 matrix where $A_S(i,j) = 1$ if and only if symbol $i$ appears in block $j$.  By virtue of being $\dblPerm{r,3}$-free,
$A_S$ does not contain $P$ as a submatrix,\footnote{In this context a submatrix is obtained by deleting rows and columns from $A_S$, and possibly flipping some 1s to 0s.} where $P$ is defined below.  Following convention~\cite{Tardos05,Pettie-FH11} we use bullets for 1s and blanks for 0s.
\[
P = \scalebox{.75}{$\left(\begin{array}{c|ccccccccc|cccccc}
	&	&	&	&	&&	&	&\bu	&\bu	&\bu	&\\
	&	&	&	&	&&\bu&\bu&	&	&	&\bu	&\\
	&	&	&	&	&\reflectbox{$\ddots$}&	&&&&&&\ddots\\
	&	&	&\bu	&\bu	&&	&	&	&	&	&	&	&	\bu\\
\bu	&\bu	&\bu	&	&	&&	&	&	&	&	&	&	&	&\bu
\end{array}\right)$
\zero{\hcm[0]\rb{0}{$\left\updownarrow\rule{0mm}{13mm}\right.$}}
\zero{\hcm[0]\rb{0}{\Large $\; r$}}}
\]
The vertical bars are not part of the pattern; they mark the boundaries of the three components of a $\dblPerm{r,3}$ sequence. 
The results of~\cite{Pettie-SoCG11} imply $\dblPERM{r,2}(n,m) \le \Ex(P,n,m) \le 2\cdot 6^{r-1}(n+m/3)$, where $\Ex(P,n,m)$ is the maximum number of 1s in $P$-free $n\times m$ matrix.  To get a bound on $\dblPERM{r,2}(n)$ we will show how to convert an $r$-sparse, $\dblPerm{r,3}$-free sequence $S$ into a blocked one.  Greedily partition $S=S_1a_1 S_2a_2\cdots S_m$ into maximal $\Perm{r,3}$-free sequences $S_1,\ldots,S_m$, separated by single symbols $a_1,\ldots,a_m$.   That is, $S_1$ is $\Perm{r,3}$-free but $S_1a_1$ is not; $S_2$ is $\Perm{r,3}$-free but $S_2a_2$ is not, and so on.
Each interval $S_k$ must contain the last occurrence of some symbol, hence $m\le n$. If this were not the case then $S$ necessarily contains
a $\Perm{r,4}$ pattern, each of which is also a $\dblPerm{r,3}$ pattern, contradicting the $\dblPerm{r,3}$-freeness of $S$.
Obtain $S'$ by discarding $a_1,\ldots,a_m$ and contracting each $S_k$ to a single block containing its alphabet $\Sigma(S_k)$.
Since $|S_k|\le \PERM{r,2}(\|S_k\|) < 2r\|S_k\|$, we have $|S| \le 2r|S'| + n$.  Being an $n$-block sequence, $|S'| \le \dblPERM{r,2}(n,n) < 2\cdot 6^{r-1}(4n/3)$,
so $|S| < 6^rrn$.
\end{proof}

\end{document}